\documentclass[review]{elsarticle}

\usepackage{amsmath,amssymb,epsfig,amsthm,bm}
\usepackage{a4wide,xcolor,eucal,enumerate,mathrsfs,graphics,caption,subfig,booktabs,verbatim}
\usepackage{dsfont,natbib}

\theoremstyle{definition}
\newtheorem{de}{Definition}

\theoremstyle{plain}
\newtheorem{theo}[de]{Theorem}
\newtheorem{lemma}[de]{Lemma}

\theoremstyle{remark}

\newcommand{\R}{\mathbb{R}}

\newcommand{\p}{\mathbb{P}}

\newcommand{\E}{\mathbb{E}}

\newcommand{\F}{\mathcal{F}}

\newcommand{\1}{\mathds{1}}

\newcommand{\dd}{\mathrm{d}}
\newcommand{\J}{\mathcal{J}}

\def\argmax{\mathop{\mathrm{argmax}}}

\usepackage{lipsum}
\makeatletter
\def\ps@pprintTitle{%
 \let\@oddhead\@empty
 \let\@evenhead\@empty
 \def\@oddfoot{}%
 \let\@evenfoot\@oddfoot}
\makeatother

\begin{document}

\begin{frontmatter}

\title{On the $L_p$ error of the Grenander-type estimator in the Cox model}


\author[address1]{C\'ecile Durot}
\ead{cecile.durot@gmail.com}

\author[mymainaddress]{Eni Musta}
\ead{e.musta@tudelft.nl}

\address[address1]{Universit\'e Paris Nanterre}
\address[mymainaddress]{Delft University of Technology}

\begin{abstract}
We consider the Cox regression model and study the asymptotic global behavior of the Grenander-type estimator for  a monotone baseline hazard function. This model is not included in the general setting of \cite{durot2007}. However, we show that a similar central limit theorem holds for  {$L_p$-error} of the Grenander-type estimator. We also propose a test procedure for a Weibull baseline distribution, based on the {$L_p$-distance} between the Grenander estimator and a parametric estimator of the baseline hazard. Simulation studies are performed to investigate the performance of this test. 
\end{abstract}


\end{frontmatter}


\section{Introduction}
A common problem in survival analysis is studying lifetime distributions in the presence of right censoring and covariates. In such cases the event time $X$ of a subject with covariate vector $Z\in\R^p$ is known only if the event occurs before the censoring time $C$. Hence, one observes only the follow up time $T=\min(X,C)$ and the censoring indicator $\Delta=\1_{\{X\leq C\}}$. For a sample of size $n$, data consists of  i.i.d. triplets $(T_1,\Delta_1,Z_1),\dots,(T_n,\Delta_n,Z_n)$.
A commonly used statistical method for investigating the relationship between the survival time and the predictor variables is the Cox proportional hazards model. It assumes that 
the hazard function at time $t$ for a subject with covariate vector $z\in\R^p$ has the form
\[
\lambda(t|z)=\lambda_0(t)\,\mathrm{e}^{\beta'_0z},\quad t\in\R^+,
\]
where $\lambda_0$ represents the baseline hazard function, corresponding to a subject with $z=0$, and $\beta_0\in\R^p$ is the vector of the regression coefficients. The proportional hazard property of the Cox model allows estimation of the effects $\beta_0$ of the covariates by the maximum partial likelihood estimator while leaving the baseline hazard completely unspecified. In this paper we deal with estimation of $\lambda_0$, which is required for example when one is interested in the absolute risk for the event to happen at a certain point in time.

Often it is natural to assume that the function $\lambda_0$ is monotone. For example, the risk of a second event for patients with acute coronary syndrome is expected to be decreasing over time because the conditions of patients stabilize as a result of the treatment and the natural disease course (\cite{Geloven13}). Moreover, in \cite{Geloven13} it is shown that using nonparametric shape constrained techniques leads to more accurate estimators than the traditional ones. Estimation of the baseline hazard function under monotonicity constraints has been studied in \cite{CC94}, \cite{LopuhaaNane2013}, \cite{LopuhaaMusta2016}, \cite{LopuhaaMusta2018}, where pointwise rates of convergence and limit distributions have been investigated. However, for goodness of fit tests, results on the global error of estimates are needed.  Here we consider one isotonic estimator which is the Grenander-type estimator  defined in \cite{LopuhaaNane2013} and we are interested in the limit behavior, as $n\to\infty$, of its $L_p$ error. 

Central limit theorems for the global error of Grenander-type estimators have been established in \cite{groeneboom1985}, \cite{groeneboom-hooghiemstra-lopuhaa1999}, \cite{kulikov-lopuhaa2005} for the density model and in \cite{durot2002} for the regression
setting. Later on, a unified approach that applies to a variety of statistical models was provided in 
\cite{durot2007}. It assumes the existence of a certain approximation of the observed cumulative hazard
by a Gaussian process. That kind of embedding is not available in the Cox model so the procedure cannot be directly applied. Here we make use of an embedding that exists for a simpler process which can then be connected to the observed cumulative hazard. As a result, the final approximating process will have also other components apart from the Brownian motion. However we can show that these additional terms do not contribute to the limit distribution. 

The paper is organized as follows. In Section~\ref{sec:notation} we fix the notations and describe some common assumptions on the model. The main results are given in Section~\ref{sec:main} and a test for the Weibull baseline distribution is proposed and investigated in Section~\ref{sec:testing}. In order to keep the exposition clear and simple, the
proofs are deferred to Section~\ref{sec:proofs}, and remaining technicalities have been put in the
Supplemental Material \cite{DM_supp}.

\section{Notation and assumptions}
\label{sec:notation}
The following assumptions are common when studying asymptotics  in the Cox regression model (see  for example \cite{LopuhaaNane2013}). Given the covariate vector $Z,$ the event time $X$ and the censoring time $C$ are assumed to be independent. Furthermore, conditionally on $Z=z,$ the event time is a nonnegative r.v. with an absolutely continuous distribution function $F(x|z)$ and density $f(x|z).$ Similarly the censoring time is a nonnegative r.v. with an absolutely continuous distribution function $G(x|z)$ and density $g(x|z).$ The censoring mechanism is assumed to be non-informative, i.e. $F$ and $G$ share no parameters. Moreover, we require that:
\begin{itemize}
	\item[(A1)]
	the end points of the support of $F$ and $G$ satisfy
	\[
	\tau_G<\tau_F\leq\infty,
	\]
	\item[(A2)]
	there exists $\epsilon>0$ such that
	\[
	\sup_{|\beta-\beta_0|\leq\epsilon}\E\left[|Z|^2\,\mathrm{e}^{2\beta'Z}\right]<\infty,
	\]
	\item[(A3)]
	for all $q\geq 1$, we have  $$\E\left[\mathrm{e}^{q\beta_0'Z}\right]<\infty. $$
\end{itemize}

Here $|\ . \ |$ denotes  the Euclidean norm and $\beta'$ denotes the transpose of $\beta$.
The first assumption tells us that, at the end of the study, there is at least one subject alive while (A2) can be seen as conditions on the boundedness of the second moment of the covariates, for $\beta$ in a neighbourhood of $\beta_0$. The third assumption is an additional one needed for our analysis in order to be able to apply certain results on empirical process theory and entropy integrals.

The most common estimator of  $\beta_0$ is $\hat{\beta}_n$, the maximizer of the partial likelihood function
\[
L(\beta)
=
\prod_{i=1}^m 
\frac{\mathrm{e}^{\beta'Z_i}}{\sum_{j=1}^n
	\1_{\{T_j\geq X_{(i)}\}}\mathrm{e}^{\beta'Z_j}}
\]
as proposed in~\cite{Cox72} and~\cite{Cox75},
where $0<X_{(1)}<\cdots<X_{(m)}<\infty$ denote the ordered, observed event times.
On the other hand, the nonparametric cumulative baseline hazard
\[
\Lambda_0(t)=\int_0^t\lambda_0(u)\,\mathrm{d}u,
\]
is usually estimated by the Breslow estimator 
\begin{equation}
\label{eq:Breslow}
\Lambda_n(x)=\int \frac{\delta\1_{\{ t\leq x\}}}{\Phi_n(t;\hat{\beta}_n)}\,\mathrm{d}\p_n(t,\delta,z).
\end{equation}
where 
\begin{equation}
\label{eq:def Phin}
\Phi_n(x;\beta)=\int \1_{\{t\geq x\}} \mathrm{e}^{\beta'z}\,\mathrm{d}\p_n(t,\delta,z),
\end{equation}
and $\p_n$ is the empirical measure of the triplets $(T_i,\Delta_i,Z_i)$ with $i=1,\dots,n.$ 
From  Theorem 5 in~\cite{LopuhaaNane2013}, {under the assumption that $\tau_G<\tau_F$} we have 
\begin{equation}
\label{eqn:Breslow}
\sup_{x\in[0,M]}|\Lambda_n(x)-\Lambda_0(x)|=O_p(n^{-1/2}), \qquad\text{for all }0<M<{\tau_G}.
\end{equation}
Without loss of generality we can assume {$\tau_G>1$} and take $M=1$. 
\eqref{eqn:Breslow} suggests that estimation near the end point of the support is problematic. Thus we consider estimation {of the baseline hazard function $\lambda_0$} on $[0,1]$ {assuming that it is non-increasing on $[0,1]$}. Let $\hat{\lambda}_n$ be the Grenander-type estimator of $\lambda_0$ on the interval $[0,1]$, i.e. $\hat{\lambda}_n$ is the left hand slope of the least concave majorant of $\Lambda_n$ on $[0,1]$. 
\section{Main results}
\label{sec:main}
At a fixed point $t\in(0,1)$, under certain regularity assumptions, the Grenander-type estimator $\hat{\lambda}_n$ converges at rate $n^{1/3}$ to $\lambda_0$ (\cite{LopuhaaNane2013}). In the present paper, we are interested in the asymptotic behavior of the $L_p$-error of  $\hat{\lambda}_n$
\[
\J_n= \int_0^1\left|\hat\lambda_n(t)-\lambda_0(t)\right|^p \,\dd t,\qquad p\geq 1.
\]
We start by providing a uniform bound of the order $n^{-p/3}$ on the $L_p$-error over points that are not very close to the boundaries. In addition, we also obtain some  control of the error  in the boundary regions which is of larger order. Throughout the paper we assume that (A1), (A2) and (A3) hold. {We assume moreover that $\tau_G>1$.} 
\begin{theo}
	\label{theo:bound_expectation_L_P}
	Let $p\geq 1$. Assume that {on $[0,1]$,} the follow-up time $T$ has uniformly continuous distribution with
	bounded density $h$ such that $\inf_{u\in[0,1]}h(u)>0$. Suppose that $\lambda_0$  and $x\mapsto \Phi(x;\beta_0)$ are  continuously differentiable and $\lambda'_0$ is bounded above by a strictly negative constant {on $[0,1]$}.
	There exists $K>0$ and an event $E_n$ with $\p(E_n)\to 1$, such that 
	\[
	\E\left[\1_{E_n}|\hat\lambda_n(t)-\lambda_0(t)|^p\right]\leq Kn^{-p/3},
	\]
	for all $t\in [n^{-1/3},1-n^{-1/3}]$, and 
	\[
	\E\left[\1_{E_n}|\hat\lambda_n(t)-\lambda_0(t)|^p\right]\leq K\left[n(t\wedge(1-t))\right]^{-p/2},
	\]
	for all $t\in[n^{-1},1-n^{-1}]$.
\end{theo}

The problem of bounding the $L_p$-error of  {Grenander type estimators} has been considered also in \cite{durot2007} and \cite{durot2008}. Results in \cite{durot2007} hold for a general setting which includes the density model, the right censoring model and the regression model with fixed design points. However, it excludes the regression model with random design points, which  is considered in \cite{durot2008}, and the present case of Cox model. 

As in \cite{durot2007} and \cite{durot2008}, bounds on the $L_p$-error of {Grenander type estimators} are  obtained through the more tractable inverse process $\hat{U}_n$ defined in \eqref{def:U_n} below. First, exponential bounds for the tail probabilities of {$\hat U_n$} are derived in Lemmas \ref{le:inv_tail_prob1} and \ref{le:inv_tail_prob2}. Afterwards, by the switching relation, results on $\hat{U}_n$ are transfered to results on $\hat\lambda_n$. The key ingredient of the proof consists in being able to bound probabilities of the type
\begin{equation}
\label{eqn:bound_M_n}
\p\left(\sup_{|t-s|<\delta}|M_n(t)-M_n(s)|>x\right)
\end{equation}
where $t,\,s\in(0,1)$, $x,\,\delta>0$ and $M_n=\Lambda_n-\Lambda_0$. The setting considered in \cite{durot2007} assumes the existence of polynomial bounds of order $\delta/nx^2$ for probabilities analogous to those in~\eqref{eqn:bound_M_n} (see Assumption (A2) in~\cite{durot2007}). Such an assumption can be easily verified for models where $M_n$ is a martingale (or can easily be connected to a martingale), for example for the right censoring model without covariates. This is not the case in our setting, mainly because of the dependence introduced by $\Phi_n$ an $\hat{\beta}_n$. Here we obtain exponential bounds for probabilities as in~\eqref{eqn:bound_M_n} but the proof becomes more technical.

The main idea is as follows. 
Let $H$ and $H^{uc}$ denote respectively the distribution function of the follow-up time and the sub-distribution function of the uncensored observations, i.e.,
\begin{equation}
\label{eq:def Huc}
H^{uc}(x)=\p(T\leq x,\Delta=1)=\int \delta\1_{\{t\leq x\}}\,\mathrm{d}\mathbb{P}(t,\delta,z),
\end{equation}
where $\p$ is the distribution of $(T,\Delta,Z)$.
If we define 
\begin{equation}
\label{eq:def Phi}
\Phi(x;\beta)=\int \1_{ \{t\geq x\}}\,\mathrm{e}^{\beta'z}\,\mathrm{d}\p(t,\delta,z),
\end{equation}
we can write
\begin{equation}
\label{eqn:lambda0}
\lambda_0(x)
=
\frac{h^{uc}(x)}{\Phi(x;\beta_0)},\qquad x\in[0,\tau_H),
\end{equation}
where $h^{uc}(x)=\mathrm{d}H^{uc}(x)/\mathrm{d}x$
(e.g., see (9) in~\cite{LopuhaaNane2013}) {and $\tau_H$ is the endpoint of he support of $H$}.
For $\beta\in\R^p$ and $x\in\R$, the function $\Phi(x;\beta)$ can be estimated by $\Phi_n(x;\beta)$ defined in $\eqref{eq:def Phin}$. 
Moreover, in Lemma 4 of~\cite{LopuhaaNane2013} it is shown that
\begin{equation}
\label{eqn:Phi}
\sup_{x\in\R}|\Phi_n(x;\beta_0)-\Phi(x;\beta_0)|=O_p(n^{-1/2}).
\end{equation} 

Using \eqref{eqn:lambda0} and \eqref{eq:Breslow}, we decompose the process $M_n$ as 
\begin{equation}
\label{eqn:decomposition_M_n}
M_n(t)=\int\frac{\delta\1_{\{u\leq t\}}}{\Phi(u;\beta_0)}\,\dd(\p_n-\p)(u,\delta,z)+\int\delta\1_{\{u\leq t\}}\left(\frac{1}{\Phi_n(u;\hat\beta_n)}-\frac{1}{\Phi(u;\beta_0)}\right)\,\dd\p_n(u,\delta,z).
\end{equation}
The event $E_n$ of the previous theorem is essentially needed to bound 
\[
\left|\frac{1}{\Phi_n(u;\hat\beta_n)}-\frac{1}{\Phi(u;\beta_0)}\right|
\]
uniformly over $u\in(0,1)$ in order to get rid of the dependence introduced by $\Phi_n$ and $\hat\beta_n$. Then, we show by using Bernstein inequality that the second term of \eqref{eqn:decomposition_M_n} has a negligible contribution in the bounds for \eqref{eqn:bound_M_n}.  On the other hand, in order to deal with the first term, we start by conditioning with respect to the follow up times and then apply Doob inequality to the resulting sum of independent random variables. This idea is somewhat similar to what is done in~\cite{durot2008} where conditioning with respect to the design points is needed to obtain a martingale.
Details on the proof can be found in Section \ref{sec:proofs} and in the Supplementary Material \cite{DM_supp}.

From Theorem \ref{theo:bound_expectation_L_P}, it follows immediately that 
$n^{p/3}\E[\J_n]=O(1).$
Next, we show that $n^{p/3}\J_n$ is asymptotically Gaussian.  We follow the proof of Theorem 2 in \cite{durot2007}. There the main ingredient used  for proving asymptotic normality is Assumption (A4) which allows approximation of an analoguous of $M_n$ by a Gaussian process. Essentially we need to approximate a properly rescaled increment  $M_n(t_n)-M_n(t)  $ by some increment $B_n(t'_n)-B_n(t')$ of a  Brownian motion or bridge $B_n$.  Then if $B_n$ is a Brownian motion, the 
approximating process is also a Brownian motion up to some normalization and the independence of its increments can be used to show asymptotic normality. Otherwise, it can be shown that the Brownian bridge component $\xi_n (t_n-t)$, where $\xi_n$ is a standard normal random variable, has a negligible contribution to the limit distribution. In our case, no such embedding for the Breslow estimator exists and the main reason is the dependence introduced by $\Phi_n$ and $\hat{\beta}_n$. However, as stated in the following lemma, we can approximate $M_n(t_n)-M_n(t)  $ by a process which has a Brownian bridge component and two additional terms that depend on $\Phi_n$ and $\hat{\beta}_n$ and are linear with respect to $t_n-t$. Since $n^{1/2}\left[\Phi_n(t;\beta_0)-\Phi(t;\beta_0)\right]$ and $n^{1/2}(\hat\beta_n-\beta_0)$ are asymptotically Gaussian, these additional terms can be treated in the same way as the Brownian bridge component $\xi_n (t_n-t)$ and, without needing to specify the dependence between the limit normal distributions, we can show that they have no effect in the asymptotic behavior of the $L_p$-error. 
\begin{lemma}
	\label{le:Lambda_n-Lambda}
	Suppose that $\lambda_0$  and $x\mapsto \Phi(x;\beta_0)$ are  continuously differentiable and bounded away from zero {on $[0,1]$}. Assume that, given $z$, the follow up time $T$ has a continuous density uniformly bounded from above and below away from zero. Define
	\begin{equation}
	\label{def:A_0}
	A_0(x)=\int_0^x\frac{D^{(1)}(u;\beta_0)}{\Phi(u;\beta_0)}\lambda_0(u)\,\dd u,\qquad D^{(1)}(u;\beta)=\frac{\partial \Phi(u;\beta)}{\partial \beta}.
	\end{equation}
	Let $t\in(0,1)$, $T_n<n^{\gamma}$ {with $\gamma\in(0,1/12)$, and $q\in[2,2/(3\gamma)]$}. There exists a Brownian Bridge $B_n$ and an event $E_n$ such that 
	\[
	\begin{split}
	M_n(t_n)-M_n(t)&= n^{-1/2}\frac{1}{\Phi(t;\beta_0)}\left[B_n\left(H^{uc}(t_n)\right)-B_n\left(H^{uc}(t)\right)\right]+(\hat\beta_n-\beta_0)'A'_0(t)\left(t_n-t \right)\\
	&\quad-\left[\Phi_n(t;\beta_0)-\Phi(t;\beta_0)\right] \frac{\lambda_0(t)}{\Phi(t;\beta_0)}\left(t_n-t\right)+{r}_n(t,t_n),
	\end{split}
	\]
	with some $r_n(t,t_n)$ {such that for all $c>0$, there exists $K>0$} that satisfies 
	\[
	\p\left[\left\{\sup_{|t_n-t|\leq cn^{-1/3}T_n}|{r}_n(t,t_n)|>x \right\}\cap E_n\right]\leq K x^{-q}n^{1-q}, 
	\]
	{for all $t\in[0,1],$} $x\geq 0$. 
\end{lemma}
The asymptotic mean and variance of $\J_n$  depend on the process $X$ defined by 
\begin{equation}
\label{eq: X}
X(a)=\argmax_{u\in\R}\left\{ -(u-a)^2+W(u)\right\}
\end{equation}
where $W$ is a standard two-sided Brownian motion. Let
\[
k_p=\int_0^\infty \text{cov}\left(|X(0)|^p,|X(a)-a|^p\right)\,\text{d}a.
\]
The main result is the following.
\begin{theo}
	\label{theo:CLT}
	Let $p\in[1,5/2)$. 
	Assume that, given $z$, the follow up time $T$ has a continuous density. Suppose that $\lambda_0$,  $x\mapsto \Phi(x;\beta_0)$  are  continuously differentiable and $\lambda'_0$, $\Phi'(x;\beta_0)$ are bounded above by strictly negative constants. Moreover, assume that there exists $C'>0$ and $s>3/4$ such that 
	\begin{equation}
	\label{eqn:assumption_derivative}
	\left|\lambda'_0(t)-\lambda'_0(x)\right|\leq C'\left|t-x\right|^s,\qquad\text{for all }t,\,x\in[0,1].
	\end{equation}
	Then, it holds
	\[
	n^{1/6}\left(n^{p/3}\int_0^1\left|\hat{\lambda}_n(t)-\lambda_0(t)\right|^p\,\text{d}t-m_p\right)\xrightarrow{d}N(0,\sigma^2_p)
	\]
	where
	\[
	m_p=\E\left[\left|X(0)\right|^p\right]\int_0^1\left|\frac{4\lambda_0(t)\lambda'_0(t)}{\Phi(t;\beta_0)}\right|^{p/3}\,\text{d}t.
	\]
	and
	\[
	\sigma^2_p=8k_p\int_0^1\left|\frac{4\lambda_0(t)\lambda'_0(t)}{\Phi(t;\beta_0)}\right|^{2(p-1)/3}\frac{\lambda_0(t)}{\Phi(t;\beta_0)}\,\text{d}t.
	\]
\end{theo}
The main motivation for studying the $L_p$-error of an isotonic estimator is to construct goodness of fit tests. However, the previous theorem only allows us to test for a fixed baseline distribution which might not be of practical interest. Usually one would like to test a certain parametric baseline distribution while leaving the parameters unspecified. Following the same line of argument as in Theorem \ref{theo:CLT}, we show that the following result holds. {Here, $\lambda_0$ takes a paramteric form, $\lambda_0(t)=f(\theta,t)$ with some unkown parameter $\theta$ and some known function $f$. For simplicity, we assume that $f\in\mathcal{C}^\infty(\R^d\times[\epsilon,M])$, which means that all derivatives of $f$ exist.}
\begin{theo}
	\label{theo:CLT2}
	Let $p\in[1,5/2)$. 
	Assume that, given $z$, the follow up time $T$ has a continuous density and $\lambda_0(t)=f(\theta,t)$ for some finite dimensional parameter $\theta\in\R^d$ and a strictly positive function $f\in\mathcal{C}^{\infty}(\R^d\times[\epsilon,M])$ {for fixed $\epsilon$ and $M$ such that $0\leq\epsilon\leq M\leq 1$}.  Suppose that {on $[\epsilon,M]$,} $\lambda_0$ and $x\mapsto \Phi(x;\beta_0)$  are  continuously differentiable and $\lambda'_0$, $\Phi'(x;\beta_0)$ are bounded above by strictly negative constants. Let $\hat\theta_n$ be an estimator of $\theta$, independent of the Grenander-type estimator $\hat{\lambda}_n$ constructed on $[\epsilon,M]$, such that 
	\begin{equation}
	\label{eqn:theta_hat}
	\sqrt{n}(\hat\theta_n-\theta)=O_P(1)
	\end{equation}	
	and define $\lambda_{\hat\theta}(t)={f(\hat\theta_n,t)}$.
	Then, it holds 
	\[
	n^{1/6}\left(n^{p/3}\int_\epsilon^M\left|\hat{\lambda}_n(t)-\lambda_{\hat{\theta}}(t)\right|^p\,\text{d}t-m_p\right)\xrightarrow{d}N(0,\sigma^2_p)
	\]
	where 
	\[
	m_p=\E\left[\left|X(0)\right|^p\right]\int_\epsilon^M\left|\frac{4\lambda_0(t)\lambda'_0(t)}{\Phi(t;\beta_0)}\right|^{p/3}\,\text{d}t.
	\]
	and
	\[
	\sigma^2_p=8k_p\int_\epsilon^M\left|\frac{4\lambda_0(t)\lambda'_0(t)}{\Phi(t;\beta_0)}\right|^{2(p-1)/3}\frac{\lambda_0(t)}{\Phi(t;\beta_0)}\,\text{d}t.
	\]
\end{theo}
{The above theorem is true for instance with $\epsilon=0$ and $M=1$ so that the interval of integration is the same as in Theorem \ref{theo:CLT}. However, we consider here the more general case for $\epsilon$ and $M$ to increase the applicability of the theorem.
	For instance, in the simulation study in the next section, we cannot choose $\epsilon=0$ because the hazard for the Weibull model and its derivative diverge to infinity at the point  zero. Hence, in that case we will apply the Theorem for some $\epsilon>0$. At the same time, in the simulation study we will also consider different values for the end point $M$ (also $M\neq 1$).}
The proof  is very similar to the one of Theorem \ref{theo:CLT}. Main ideas  are given in Section \ref{sec:proofs} while  further details can be found in the Supplementary Material \cite{DM_supp}. 
\section{A goodness of fit test}
\label{sec:testing}
The Cox regression model is widely used in survival analysis and the most common and straightforward approach is to assume a parametric model for the baseline hazard $\lambda_0$.  
Frequent choices of parametric models are the Exponential, Weibull and Gompertz distribution (see for example \cite{LG97}). When one wants to model the survival distribution of a population with increasing or
decreasing risk, the Weibull distribution may be used  since depending on whether the shape parameter is greater/smaller than
one, the hazard rate is increasing/decreasing. The Weibull distribution is characterized by the shape parameter $\nu$ and the scale parameter $\mu$. The hazard rate is given by 
\[
\lambda_0(t)=\nu\mu^\nu t^{\nu-1}\quad\mbox{ for all } t\geq 0.
\]
The adequacy of the Weibull model can be assessed by looking at the quality  of the regression of  $\log(-\log (1-\hat{F}_n(t)))$ on  $\log t$, where $\hat{F}_n$ is an estimator of the baseline distribution function (\cite{miller}). Indeed, if the true distribution is Weibull, then
\[
1-F(t)=\exp\left\{-(\mu t)^\nu \right\}\quad\text{and}\quad \log(-\log (1-F(t)))=\nu\log t+\nu\log \mu.
\]
Such approach is commonly used in practice (see for example Section 5.2 in~\cite{collett2015} or \cite{ducrocq1988,zhang2016}).
Here we propose a general procedure for testing a parametric assumption on the baseline distribution, assuming that the baseline hazard function is monotone. The test is based on the $L_p$-distance between the parametric estimator and the Grenander-type estimator. We describe the test and investigate its performance for the case of the Weibull baseline distribution since it is of more practical interest. 

We assume that the baseline hazard is decreasing and we want to test whether it corresponds to a Weibull distribution, i.e. we consider the following test
\[
H_0:\quad\lambda_0(t)=\nu \mu^\nu t^{\nu-1}\quad\text{for some }\mu>0\text{ and }\nu<1
\]
against
\[
H_1:\quad\lambda_0\text{ is decreasing}.
\]
Under the null hypothesis we estimate the baseline hazard by its parametric estimator
\[
\lambda_{\hat{\theta}}(t)=\hat\nu_n \hat{\mu}_n^{\hat\nu_n} t^{\hat\nu_n-1}
\]
where $\hat\theta_n=(\hat\mu_n,\hat\nu_n)$ is the maximum likelihood estimator of $\theta=(\mu,\nu)$. 

Under $H_1$ we estimate $\lambda_0$ by the Grenander-type estimator $\hat{\lambda}_n$. In this case we need to consider estimation on $[\epsilon,M]$ with $\epsilon>0$ because $
\lim_{t\to 0}\lambda_0(t)=\lim_{t\to 0}|\lambda'_0(t)|=\infty.$
We split the sample in two parts of sizes $n_1=rn$ and $n_2=(1-r)n$ for some $r\in(0,1)$. We use the first part to estimate $\theta$ and obtain $\lambda_{\hat{\theta}}$ and the second part to construct the nonparametric estimator $\hat{\lambda}_n$. Hence $\hat\theta_n$ is independent of $\hat{\lambda}_n$.
Moreover, the conditions of Theorem \ref{theo:CLT2} are satisfied since
\[
\sqrt{n}\left(\hat\mu_n-\mu\right)=O_P(1)\quad\text{and}\quad\sqrt{n}\left(\hat\nu_n-\nu\right)=O_P(1),
\]
See for example Section 3.4 in \cite{kalbfleisch2011}.
Let $0<\epsilon<M<\tau_H$. We consider the test statistic 
\[
T_n=n_2^{1/3}\int_\epsilon^M \left|\hat{\lambda}_n(t)-\lambda_{\hat{\theta}}(t) \right|\,\dd t
\] 
{where $\hat{\lambda}_n$ is constructed on $[\epsilon,M]$. For the simulation study we consider various choices of $\epsilon$ and $M$, including cases where $M\neq 1$} {and,} at level $\alpha$  we reject the null hypothesis if $T_n>c_{n,\alpha}$ for some critical value $c_{n,\alpha}$.

In order to avoid computation of the asymptotic mean and variance in Theorem {\ref{theo:CLT2}}, we compute the critical value $c_{n,\alpha}$ by a bootstrap procedure as follows.

Let $\tilde\beta_n$ be the maximum likelihood estimator of $\beta_0$ (assuming Weibull baseline distribution).  We generate  $B=1000$ bootstrap samples as follows. We fix the covariates and we generate the event time $X_i^*$ from the parametric estimator
\[
F_{\hat{\theta}}(x|Z_i)=1-\exp\left(-\Lambda_{\hat\theta}(x)e^{\tilde\beta_nZ_i}\right)
\] 
for the cdf of $X$ conditional on $Z_i$, {where $\Lambda_{\hat\theta}$ is the cumulative baseline hasard corresponding to $\lambda_{\hat\theta}$}.
We assume that the censoring times $C_i$ are independent of the covariates and we generate $C^*_i$ from the Kaplan-Meier estimate $\hat{G}_n$ (independently of $Z_i$).
Then we take~$T^*_i=\min(X^*_i,C^*_i)$ and $\Delta^*_i=\1_{\{X^*_i\leq C^*_i\}}$.
We divide each bootstrap sample $(T_1^*,\Delta_1^*,Z_1),\dots,(T_{n}^*,\Delta_{n}^*,Z_{n})$ in two subsamples of size $n_1$ and $n_2$. Using the first subsamples we compute the maximum likelihood estimators $\hat{\theta}_{n,j}^*$
of $\hat{\theta}_n$ and obtain $\lambda^*_{\hat\theta,j} $. From the second subsamples, we compute  the Grenander-type estimators $\hat{\lambda}_{n,j}^{*}$ and the value of the test statistic
\[
T_{n,j}^{*}=n_2^{1/3}\int_\epsilon^M \left|\hat{\lambda}_{n,j}^{*}(t)-\lambda^*_{\hat{\theta},j}(t) \right|\,\dd t,\qquad j=1,\dots,B.
\]
Then we take as a critical value the $100\alpha$-th upper percentile of the values $T_{n,1}^*,\dots,T_{n,B}^*$.

Consistency of the bootstrap method follows from the next 
theorem. Let $P_n^*$ denote the conditional distribution given the data. Let $\hat\beta_n^*$ be 
{an} estimator of $\tilde{\beta}_n$ in the bootstrap sample. We assume that 
$\hat\beta_n^*-\tilde{\beta}_n\to 0$,
for almost all sequences $(T_i^*,\Delta_i^*,Z_i)$, $i=1,2,\ldots$, conditional on the sequence  $(T_i,\Delta_i,Z_i)$, $i=1,2,\ldots$,
and that 
\begin{equation}
\label{eqn:beta_n^*}
\sqrt{n}(\hat\beta_n^*-\tilde{\beta}_n)=O_{P^*}(1),
\end{equation}
meaning that for all $\epsilon>0$, there exists $M>0$ such that
\[
\limsup_{n\to\infty}
P_n^*
\left(
\sqrt{n}
|\hat\beta_n^*-\tilde{\beta}_n|
>M
\right)
<
\epsilon,
\qquad
\p-\text{almost surely}.
\]
This property is proved in \cite{hjort1985bootstrapping} {for the maximum partial likelihood estimator} in a slightly different setting. There, the bootstrap event times are  generated from a non-smooth estimator instead of our $F_{\hat\theta}$.
\begin{theo}
	\label{theo:CLT*}
	Suppose that the censoring time is independent of the covariate and the assumptions of Theorem \ref{theo:CLT2} are satisfied. Let $\hat\lambda_n^*$ be the Grenander estimator in the bootstrap sample constructed as above from a baseline distribution with hazard rate $\lambda_{\hat\theta}$. Given the data, let $\hat{\theta}_n^*$ be an estimator of $\hat{\theta}_n$ independent of $\hat\lambda_n^*$ such that 
	\begin{equation}
	\label{eqn:theta_hat*}
	\sqrt{n}(\hat{\theta}_n^*-\hat{\theta}_n)=O_{P^*}(1)
	\end{equation}
	and define $\lambda^*_{\hat\theta}(t)={f({\hat{\theta}_n^*},t)}$. {Then,} given the data $(T_1,\Delta_1,Z_1),$ $\dots,(T_{n},\Delta_{n},Z_{n})$, it holds
	\[ 
	n^{1/6}\left(n^{p/3}\int_\epsilon^M\left|\hat{\lambda}_n^*(t)-\lambda^*_{\hat\theta}(t)\right|^p\,\text{d}t-m_p\right)\xrightarrow{d}N(0,\sigma^2_p)
	\]
	{in probability} where $m_p$ and $\sigma^2_p$ are as in Theorem \ref{theo:CLT2}.

	{Therefore, the test that rejects the null hypothesis if $T_n>c_{n,\alpha}$, where $c_{n,\alpha}$ is the $100\alpha$-th upper percentile of the values $T_{n,1}^*,\dots,T_{n,B}^*$, has asymptotic level $\alpha$.}
\end{theo}
The proof follows the ones of Theorems \ref{theo:CLT} and \ref{theo:CLT2} but it is more technical because the \lq true\rq\ functions in the Bootstrap version  depend on $n$ and are less smooth than the original ones.  Some ideas are given in Section \ref{sec:proofs} and further details in the Supplemental Material \cite{DM_supp}.

To investigate the performance of the test we repeat the procedure $N=1000$ times and count the percentage of rejections. This gives an approximation of the level (or the power) of the test if we start with a sample from a Weibull distribution (or not Weibull). 

We perform the following simulation study. To investigate how the test behaves in terms of level we start with a sample of event times from the Weibull distribution with parameters $\mu>0$ and $\nu<1$. We consider five different scenarios, depending on the choice of the parameters, which correspond to slightly and strongly decreasing hazard rates. The real valued covariates and the censoring times are chosen to be uniformly distributed on $[0,1]$ and $[0,\tau]$ respectively. $\tau$ is chosen in such a way that the percentage of not censored observations is around $75\%$.
We take $\beta_0=0.5$. Unfortunately, it is not clear how to optimally choose  $\epsilon$ and $M$. In practice we take $\epsilon$ to be relatively small with respect to the range of the data and $M$ is smaller than the last observed event time. We compare this test with a likelihood ratio type of test (using the Grenander estimator instead of the maximum likelihood estimator) and a Kolmogorov-Smirnov type of test. 
For the first one, note that the likelihood in the Cox model can be written as
\[
\prod_{i=1}^n \left(\lambda_0(T_i)\exp(\beta_0Z_i)\right)^{\Delta_i}\exp\{-\Lambda_0(T_i)\exp(\beta_0Z_i)\}\prod_{i=1}^n g(T_i|Z_i)^{1-\Delta_i}\left[1-G(T_i|Z_i)\right]^{\Delta_i}.
\]
Hence we consider the test that rejects $H_0$ for small values of 
\[
LR_n=\frac{\prod_{i=1}^n \left(\lambda_{\hat\theta}(T_i)\exp(\tilde{\beta}_nZ_i)\right)^{\Delta_i}\exp\{-\Lambda_{\hat\theta}(T_i)\exp(\tilde\beta_nZ_i)\}}{\prod_{i=1}^n \left(\hat\lambda_n(T_i)\exp(\hat\beta_nZ_i)\right)^{\Delta_i}\exp\{-\hat\Lambda_n(T_i)\exp(\hat\beta_nZ_i)\}},
\]
where $\hat\Lambda_n(t)=\int_0^t\hat\lambda_n(u)\,\dd u$. The second test we consider rejects $H_0$ for large values of $S_n=\sup_{[0,\tau]}|F_{\hat\theta}(t)-{F}_n(t)|$, where $F_{\hat\theta}$ is the parametric estimator of the distribution function of the event times and ${F}_n(t)=1-\exp(\Lambda_n(t))$. Note that the last test makes no use of the monotonicity assumption. We determine the  critical values of these tests through a bootstrap procedure as described above. 

Results of the simulations for sample size $n=2000, 4000$  and $n_1=n/2$ are reported in Table~\ref{tab:level1}. Division in two subsamples of equal size seems to give better results with respect to for example $n_1=\lfloor n/3\rfloor$. 
\begin{table}[h]
	\begin{tabular}{cccccc}
		\toprule
		$(\mu;\nu)$ &$(5;0.5)$ &$(1;0.1)$ & $(1;0.9)$ & $(1;0.5)$ &$(0.1;0.5)$\\
		$(\epsilon;M;\tau)$  & $(0.1;0.5;0.7)$ & $(0.5;4.5;7)$ & $(0.5;2.5;3.5)$ & $(0.5;2.5;3.5)$ & $(1;30;35)$ \\
		\\[-10pt]
		\cline{1-6}
		\\[-10pt]
		& $0.073$  & $0.053$  & $0.050$  &  $0.058$ & $0.058$ \\
		$n=2000$  & $0.081$ & $0.067$ & $0.074$ & $0.068$ & $0.074$\\
		& $0.053$ &$0.055$  & $0.047$  &$0.056$  & $0.049$\\
		\\[-10pt]
		\cline{1-6}
		\\[-10pt]
		& $0.062$ & $0.062$ & $0.054$ & $0.056$ & $0.064$\\
		$n=4000$  & $0.086$ &$0.079$  & $0.059$  & $0.083$  & $0.077$\\
		& $0.059$  & $0.053$  & $0.036$ & $0.063$ & $0.055$\\
		\bottomrule\\
	\end{tabular}
	\caption{Simulated level of $T_n$ (first numbers), $LR_n$ (second numbers) and $S_n$ (third numbers) for $\alpha=0.05$.}
	\label{tab:level1}
\end{table}
However, in practice it seems that if we do not divide the sample in two part but we use all the data for computing both the parametric and the nonparametric estimator, the performance of the tests improves  (see Table~\ref{tab:level2}). This suggests that the independence of $\hat\theta_n$ from the Grenander estimator might be only a theoretical requirement and the bootstrap procedure works even if they are dependent. 
\begin{table}[h]
	\begin{tabular}{cccccc}
		\toprule
		$(\mu;\nu)$ &$(5;0.5)$ &$(1;0.1)$ & $(1;0.9)$ & $(1;0.5)$ &$(0.1;0.5)$\\
		$(\epsilon;M;\tau)$  & $(0.1;0.5;0.7)$ & $(0.5;4.5;7)$ & $(0.5;2.5;3.5)$ & $(0.5;2.5;3.5)$ & $(1;30;35)$ \\
		\\[-10pt]
		\cline{1-6}
		\\[-10pt]
		& $0.038$  & $0.043$  & $0.046$  &  $0.034$ & $0.038$ \\
		$n=2000$  & $0.040$ & $0.029$ & $0.068$ & $0.040$ & $0.041$\\
		&$0.046$  & $0.039$  &$0.044$  &$0.053$  & $0.047$ \\
		\\[-10pt]
		\cline{1-6}
		\\[-10pt]
		& $0.051$ & $0.051$ & $0.038$ & $0.046$ & $0.044$\\
		$n=4000$  & $0.036$  & $0.019$ & $0.064$ & $0.036$ &$0.034$ \\
		& $0.056$ &$0.050$  &$0.053$  & $0.042$ & $0.046$\\
		\bottomrule\\
	\end{tabular}
	\caption{Simulated level of $T_n$ (first numbers), $LR_n$ (second numbers) and $S_n$ (third numbers) for $\alpha=0.05$ without dividing in two subsamples.}
	\label{tab:level2}
\end{table}

To investigate the performance of the test in terms of power we start with samples of event times from a distribution with failure rate given by
\[
(a) \quad \lambda_0(t)=\frac{1}{t+c}\qquad\text{and}\qquad(b)\quad\lambda_{0}(t)=c+\frac{1}{2\sqrt{t}}.
\]
Note that both these baseline distributions have decreasing hazard rates but are not Weibull. The correspondent distribution functions are $F(t)=1-c/(t+c)$ and $F(t)=1-\exp(-ct-\sqrt{t})$.

\begin{table}[h]
	\begin{tabular}{ccccc}
		\toprule
		Model & (a) $c=1$ &(a) $c=6$ & (b) $c=1$ & (b) $c=3$ \\
		$(\epsilon;M;\tau)$  & $(0.1;5;6)$ & $(1;30;35)$ & $(0.1;1.1;1.3)$ & $(0.1;0.5;0.6)$  \\
		\\[-10pt]
		\cline{1-5}
		\\[-10pt]
		& $0.385$  & $0.333$  & $0.302$  &  $0.311$ \\
		$n=2000$  & $0.048$ & $0.036$  & $0.266$ & $0.298$  \\
		&$0.308$  & $0.291$ & $0.089$ & $0.082$  \\
		\\[-10pt]
		\cline{1-5}
		\\[-10pt]
		& $0.896$ & $0.866$ & $0.423$ & $0.465$ \\
		$n=4000$  & $0.158$ & $0.145$ & $0.338$ &  $0.398$ \\
		& $0.587$ & $0.556$ &$0.168$  & $0.169$  \\
		\bottomrule\\
	\end{tabular}
	\caption{Simulated power of $T_n$ (first numbers), $LR_n$ (second numbers) and $S_n$ (third numbers) for $\alpha=0.05$.}
	\label{tab:power1}
\end{table}
\begin{table}[h]
	\begin{tabular}{ccccc}
		\toprule
		Model & (a) $c=1$ &(a) $c=6$ & (b) $c=1$ & (b) $c=3$ \\
		$(\epsilon;M;\tau)$  & $(0.1;5;6)$ & $(1;30;35)$ & $(0.1;1.1;1.3)$ & $(0.1;0.5;0.6)$  \\
		\\[-10pt]
		\cline{1-5}
		\\[-10pt]
		& $0.996$  & $0.989$  & $0.675$  &  $0.748$ \\
		$n=2000$  & $0.097$ & $0.084$ & $0.275$ & $0.379$ \\
		&$0.990$  & $0.990$  &$0.461$  &$0.502$ \\
		\\[-10pt]
		\cline{1-5}
		\\[-10pt]
		& $1$ & $1$ & $0.915$ & $0.958$\\
		$n=4000$  & $0.600$  & $0.582$ & $0.413$  & $0.539$  \\
		& $1$ & $1$  & $0.742$ &  $0.790$  \\
		\bottomrule\\
	\end{tabular}
	\caption{Simulated power of $T_n$ (first numbers), $LR_n$ (second numbers) and $S_n$ (third numbers) for $\alpha=0.05$ without dividing in two subsamples.}
	\label{tab:power2}
\end{table}
Results of the simulations for sample size $n=2000, 4000$  and $n_1=n/2$ are reported in Table~\ref{tab:power1}.
Again, we repeat the simulation without dividing the sample in two subsamples. Results in Table~\ref{tab:power2} indicate that the performance of the tests improves significantly. 
We note that the power of the tests depends on whether we consider alternative (a) or (b) but is not very sensitive to the choice of the constant $c$. When dividing the sample in two subsamples, the test $T_n$ is most powerful under both  alternatives (a) and (b). None of the tests is uniformly more powerful. When using the whole sample for the computation of both the parametric and nonparametric estimators, the test $T_n$ becomes comparable to $S_n$ under alternative (a). Again $T_n$ is the most powerful of the three considered tests for model (b). 

To conclude, simulations show promising results for testing a parametric assumption (Weibull) on the baseline distribution of the event times in the Cox model, when knowing that the hazard rate is monotone. However, further research is needed to better understand how the constants $\epsilon$ and $M$ can be chosen in practice and to investigate the behavior in other scenarios.
\section{Auxiliary results and proofs}
\label{sec:proofs}
{In this section, we use the shorthand $\lesssim$ instead of $\leq C$ where $C$ is a constant (i.e. a real number independent of $n$) that may depend on the parameters of the model such as lower or upper bounds for the involved density functions or derivatives. Moreover, $K,K_1$, \dots and $c,c_1,\dots$ denote such constants. The constants may change from line to line.} 
\subsection{Proof of Theorem \ref{theo:bound_expectation_L_P}}
Properties of the Grenander-type estimator $\hat{\lambda}_n$ are commonly derived through the  inverse process $\hat U_n$ defined as
\begin{equation}
\label{def:U_n}
\hat{U}_n(a)
=
\argmax_{t\in[0,1]} \left\{\Lambda_n(t)-at\right\}.
\end{equation}
Afterwards, what permits us to  translate the results in terms of $\hat{\lambda}_n$ is the switching relation
\begin{equation}
\label{eqn:switching}
\hat{\lambda}_n(t)\geq a\qquad\text{if and only if}\qquad\hat{U}_n(a)\geq t,\qquad\text{for } t\in(0,1].
\end{equation} 
Moreover, let $U$ be the {generalized} inverse of $\lambda_0$ on $\R$, i.e.,
\begin{equation}
\label{def:U}
U(a)
=
\begin{cases}
1 & \mbox{{if }}a< \lambda_0(1);\\
\lambda_0^{-1}(a) & \mbox{{if }}a\in[\lambda_0(1),\lambda_0(0)];\\
0 & \mbox{{if }}a>\lambda_0(0).
\end{cases}
\end{equation}
We aim at obtaining bounds on tail probabilities of $\hat{U}_n$. A polynomial bound has been provided in Lemma 6.3 of \cite{LopuhaaMusta2018} but it is not sufficient when dealing with the global error of the estimator. Here we derive exponential bounds which hold on an event $E_n$ with $\p(E_n)\to 1$ similar to the one considered in \cite{LopuhaaMusta2018}. Let us first describe this event. 

For some $\xi_1>0$, define
\begin{equation}
\label{eqn:E1}
E_{n,1}=\left\{\left|\sqrt{p}\sum_{I_k\subseteq I}
\left(\frac{1}{n}\sum_{i=1}^n|Z_i|\,\mathrm{e}^{\gamma'_kZ_i} \right)
-
\sqrt{p}\sum_{I_k\subseteq I}\E\left[|Z|\mathrm{e}^{\gamma'_kZ} \right]\right|\leq\xi_1 \right\},
\end{equation}
where the summations are over all subsets $I_k=\{i_1,\ldots,i_k\}$ of $I=\{1,\ldots,p\}$,
and $\gamma_k$ is the vector consisting of coordinates
$\gamma_{kj}=\beta_{0j}+\epsilon/(2\sqrt{p})$, for $j\in I_k$, and
$\gamma_{kj}=\beta_{0j}-\epsilon/(2\sqrt{p})$, for $j\in I\setminus I_k$. It is shown in \cite{LopuhaaMusta2018} (see the proof of Lemma 3.1)) that $\p(E_{n,1})\to 1$ and that in this event we have
\begin{equation}
\label{eqn:dn2}
\sup_{x\in\R}|D_{n}^{(1)}(x;\beta^*_n)|\leq \sqrt{p}\sum_{I_k\subseteq I}\E\left[|Z|\mathrm{e}^{\gamma'_kZ} \right]+\xi_1=\tilde\xi_1,
\end{equation}
where $D_n^{(1)}(x;\beta)=\partial \Phi_n(x;\beta)/\partial \beta$ and $\beta^*_n$ is a sequence such that $\beta^*_n\to \beta_0$ almost surely. Hence, {under the assumption (A2),} in the event $E_{n,1}$ it holds that  $\sup_{x\in\R}|D_{n}^{(1)}(x;\beta^*_n)|$ is bounded uniformly in $n$.

Moreover, using the linear representation of the Breslow estimator in \cite{linear_breslow}, for $x\in[0,1]$ and each $\alpha>0$, we have 
\begin{equation}
\label{eqn:Breslow_linear}
\begin{split}
\Lambda_n(x)-\Lambda_0(x)&=\int \left(\frac{\delta\1_{\{t\leq x\}}}{\Phi(t;\beta_0)}-e^{\beta_0'z}\int_0^{t\wedge x}\frac{\lambda_0(v)}{\Phi(v;\beta_0)}\,\dd v \right)\,\dd (\p_n-\p)(t,\delta,z)\\
&\quad+(\hat\beta_n-\beta_0)'A_0(x)+R_n(x),
\end{split}
\end{equation}
where 
\begin{equation}\label{eq: Rn}
\sup_{x\in[0,1]}|R_n(x)|=o_P(n^{-1+\alpha}).
\end{equation}
For $\xi_2,\xi_3,\xi_4>0$ and $\alpha>0$ (small) {define}
\begin{equation}\label{eq:E2E3}
\begin{split}
E_{n,2}
&=
\left\{n^{1/2-\alpha}|\hat{\beta}_n-\beta_0|<\xi_2\right\}, \quad
E_{n,3}
=
\left\{n^{1/2-\alpha}\sup_{x\in \mathbb{R}}\left|\Phi_n(x;\beta_0)-\Phi(x;\beta_0)\right|\leq \xi_3\right\}, 
\end{split}
\end{equation}
\[
E_{n,4}=\left\{\sup_{x\in[0,1]}|R_{n}(x)|\leq \xi_4n^{-1+\alpha} \right\}.
\]
From Theorem~3.2 in~\cite{Tsiatis81} and Lemma~4 in~\cite{LopuhaaNane2013}, it follows that $\p(E_{n,2})\to 1$ and $\p(E_{n,3})\to 1$.
Moreover, {\eqref{eq: Rn}} yields also
$\p(E_{n,4})\to 1$ and as a result
for 
\begin{equation}
\label{def:E_n}
E_n=E_{n,1}\cap E_{n,2}\cap E_{n,3}\cap E_{n,4},
\end{equation}
we have $\p(E_n)\to 1$. 
This event is of interest because it allows us to bound 
\[
\sup_{{x\in[0,1]}}|\Phi_n(x;\hat{\beta}_n)-\Phi(x;\beta_0)|.
\]
Indeed, by the triangular inequality and the definition of $E_n$, we get
\begin{equation}
\label{eqn:bound_phi_n_phi}
\begin{split}
\sup_{x\in\R}|\Phi_n(x;\hat{\beta}_n)-\Phi(x;\beta_0)|
&\leq
\sup_{x\in\R}|\Phi_n(x;\hat{\beta}_n)-\Phi_n(x;\beta_0)|+\sup_{x\in\R}|\Phi_n(x;\beta_0)-\Phi(x;\beta_0)|\\
&\leq
|\hat{\beta}_n-\beta_0|\sup_{x\in\R}|D_n^{(1)}(x;\beta^*)|+ \frac{\xi_3}{n^{1/2-\alpha}}
\lesssim n^{-1/2+\alpha}.
\end{split}
\end{equation}
In particular, for sufficiently large $n$, we have
$\sup_{x\in\R}\left|\Phi_n(x;\hat{\beta}_n)-\Phi(x;\beta_0)\right|\leq \Phi(1;\beta_0)/2$,
{where $\Phi(1;\beta_0)>0$ since under the assumption that $\tau_G>1$, we have $\p(T\geq 1)>0$. This} yields that, for $x\in[0,1]$,
\begin{equation}
\label{eqn:bound Phihat}
\Phi_n(x;\hat{\beta}_n)
\geq
\Phi(x;\beta_0)- \frac{1}{2}\Phi(1;\beta_0)
\geq \frac{1}{2}\Phi(1;\beta_0)>0.
\end{equation}
Using~\eqref{eqn:bound_phi_n_phi} and \eqref{eqn:bound Phihat}, on the event $E_n$, for $n$ sufficiently large, we can write
\begin{equation}
\label{eqn:bound_inverse_Phi_n}
\begin{split}
\sup_{{x\in[0,1]}}\left|\frac{1}{\Phi_n(x;\hat{\beta}_n)}-\frac{1}{\Phi(x;\beta_0)} \right|
&\leq
\sup_{x\in[0,1]}\frac{\left|\Phi_n(x;\hat{\beta}_n)-\Phi(x;\beta_0)\right|}{\Phi_n(x;\hat{\beta}_n)\,\Phi(x;\beta_0)}\\
&\leq
\frac{2}{\Phi^2(1;\beta_0)}\sup_{x\in[0,1]}\left|\Phi_n(x;\hat{\beta}_n)-\Phi(x;\beta_0)\right|\\
&\lesssim n^{-1/2+\alpha}.
\end{split}
\end{equation}
The following two lemmas correspond to Lemmas 2, 3 and 4 in \cite{durot2007}. 
\begin{lemma}
	\label{le:inv_tail_prob1}
	Assume that the follow-up time $T$ has uniformly continuous distribution with
	continuously differentiable density $h$ such that $\inf_{u\in[0,1]}h(u)>0$. Suppose that $\lambda_0$  and $x\mapsto \Phi(x;\beta_0)$ are  continuously differentiable and $\lambda'_0$ is bounded above by a strictly negative constant.
	There exist  constants $K_1,\, K_2>0$ such that, 	for every $a\geq 0$ and $x> 0$,
	\begin{equation}
	\label{eqn:inv}
	\p\left(
	\left\{|\hat{U}_n(a)-U(a)|\geq x\right\}
	\cap E_n
	\right)
	\leq
	K_1\exp\left(-K_2nx^3\right),
	\end{equation}
	{provided that $\alpha$ in the definition of $E_n$ is chosen sufficiently small.}
\end{lemma}
\begin{proof}
	From the definition of $U$ and the fact that $\hat{U}_n$ is decreasing, it follows that
	$|\hat{U}_n(a)-U(a)|
	\leq
	|\hat{U}_n(\lambda_0(0))-U(\lambda_0(0))|$,
	if $a>\lambda_0(0)$,
	and
	\[
	|\hat{U}_n(a)-U(a)|
	\leq
	|\hat{U}_n(\lambda_0(1))-U(\lambda_0(1))|,\qquad \text{if } a<\lambda_0(1).
	\]
	Hence, it suffices to prove~\eqref{eqn:inv} only for $a\in[ \lambda_0(1),\lambda_0(0)]$.
	Moreover, it suffices to prove the inequality for $x\in[n^{-1/3},1]$ since for $x\in[0,n^{-1/3}]$, the inequality is trivially true with $K_2=1$, $K_1=\exp(1)$ and for $x>1$ the probability is zero.
	We start by writing
	\begin{equation}
	\label{eqn:main_inv_pr}
	\begin{split}
	&
	\p\left(
	\left\{|\hat{U}_n(a)-U(a)|\geq x\right\}\cap E_n\right)\\
	&\quad{\leq}
	\p\left(\left\{
	\hat{U}_n(a)\geq U(a)+x \right\}\cap E_n\right)
	+
	\p\left(\left\{\hat{U}_n(a)\leq U(a)-x\right\}\cap E_n\right).
	\end{split}
	\end{equation}
	First consider the first probability on the right hand side of~\eqref{eqn:main_inv_pr}.
	It is zero, if $U(a)+x>1$. Otherwise, if $U(a)+x\leq 1$, 
	\[
	\begin{split}
	&
	\p\left(\left\{
	\hat{U}_n(a)\geq U(a)+x \right\}\cap E_n\right)\\
	&\leq
	\p\Big(\left\{
	\Lambda_n(y)-ay\geq \Lambda_n(U(a))-aU(a),\text{for some }y\in[U(a)+x,1]
	\right\}\cap E_n
	\Big)\\
	&\leq
	\p\left(\left\{
	\sup_{y\in[U(a)+x,1]}
	\Big(
	\Lambda_n(y)-ay-\Lambda_n(U(a))+a\,U(a)
	\Big)
	\geq 0\right\}\cap E_n\right).
	\end{split}
	\]
	From Taylor's expansion, we obtain
	\begin{equation}
	\label{eqn:taylor_inv_pr}
	\Lambda_0(y)-\Lambda_0(U(a))\leq \big(y-U(a)\big)a-c \big(y-U(a)\big)^2,
	\end{equation}
	where $c=\inf_{t\in[0,1]}|\lambda'_0(t)|/2>0$. Hence 
	\[
	\begin{split}
	&
	\p\left(\left\{
	\hat{U}_n(a)\geq U(a)+x \right\}\cap E_n\right)\\
	&\leq\p\left(\left\{
	\sup_{y\in[U(a)+x,1]}
	\Big(\Lambda_n(y)-\Lambda_0(y)-\Lambda_n(U(a))+\Lambda_0(U(a))-c(y-U(a))^2\Big)\geq 0\right\}\cap E_n\right).
	\end{split}
	\] 
	Define
	\[
	S_n(t)=\int \frac{\delta\1_{\{u\leq t\}}}{\Phi(u;\beta_0)}\,\dd (\p_n-\p)(u,\delta,z),
	\]
	and
	\[
	R_n(t)=\int \delta\1_{\{u\leq t\}}\left(\frac{1}{\Phi_n(u;\hat\beta_n)}-\frac{1}{\Phi(u;\beta_0)}\right)\,\dd \p_n(u,\delta,z).
	\]
	{Since $\Lambda_n-\Lambda_0=R_n+S_n$, it} follows that the previous probability can be bounded by
	\begin{equation}
	\label{eqn:S+R}
	\begin{split}
	&\sum_{k=0}^\infty
	\p\left(\left\{
	\sup_{y\in I_k}
	\big|S_n(y)-S_n(U(a))\big|
	\geq \frac{c}{2}\,x^2\,2^{2k}\right\}\cap E_n\right)\\
	&\qquad+\sum_{k=0}^\infty
	\p\left(\left\{
	\sup_{y\in I_k}
	|R_n(y)-R_n(U(a))|
	\geq \frac{c}{2}\,x^2\,2^{2k}\right\}\cap E_n\right)
	\end{split}
	\end{equation}
	where the supremum is taken over $y\in[0,1]$, such that $y-U(a)\in[x2^k,x2^{k+1}{\wedge 1})$ with the convention that $\sup$ over an empty set is zero.
	Moreover, the process $S_n$ can be written as
	\[
	\begin{split}
	S_n(t)&=\frac{1}{n}\sum_{i=1}^n\left(\Delta_i-\E_T[\Delta_i]\right)\frac{1}{\Phi(T_i;\beta_0)}\1_{\{T_i\leq t\}}\\
	&\quad+\frac{1}{n}\sum_{i=1}^n\left(\E_T\left[\Delta_i\right]\frac{1}{\Phi(T_i;\beta_0)}\1_{\{T_i\leq t\}}-\E\left[\frac{\Delta_i}{\Phi(T_i;\beta_0)}\1_{\{T_i\leq t\}}\right]\right)\\
	&=:S_n^1(t)+S^2_n(t),
	\end{split}
	\]
	where $E_T$ denotes the conditional expectation given ${T_1},\dots,T_n$.  
	By Lemmas \ref{le:S1}, \ref{le:S2} {in the supplementary material} with 
	\[
	r(t)=\frac{1}{\Phi(t;\beta_0)}<\frac{1}{\Phi(1;\beta_0)}\qquad\text{and}\qquad m(t)=\frac{\E\left[\Delta\,|\,T=t\right]}{\Phi(t;\beta_0)}<\frac{1}{\Phi(1;\beta_0)},
	\]
	{where $m$ is continuously differentiable under our assumptions since}
	\[
	m(t)=\frac{\p(\Delta=1|T=t)}{\Phi(t;\beta_0)}=\frac{h^{uc}(t)}{h(t)\Phi(t;\beta_0)}{=\frac{\lambda_0(t)}{h(t)}},
	\]
	it follows that 
	\begin{equation}
	\label{eqn:S1+S2}
	\begin{split}
	&\p\left(\left\{
	\sup_{y\in I_k}
	\big|S_n(y)-S_n(U(a))\big|
	\geq \frac{c}{2}\,x^2\,2^{2k}\right\}\cap E_n\right)\\
	&\leq \p\left(
	\sup_{y\in I_k}
	\big|S^1_n(y)-S^1_n(U(a))\big|
	\geq \frac{c}{4}\,x^2\,2^{2k}\right)\\
	&\quad +\p\left(
	\sup_{y\in I_k}
	\big|S_n^2(y)-S_n^2(U(a))\big|
	\geq \frac{c}{4}\,x^2\,2^{2k}\right)\\
	&\leq K_1\exp\left(-K_2nx^32^{3k}\right)
	\end{split}
	\end{equation}
	for some positive constants $K_1,\, K_2$. Note that when we apply Lemma \ref{le:S1} we have 
	\[
	\tilde{\theta}=\min\left\{1, {\frac{cx^22^{2k}}{3x2^{k+1}R^2Ce^{R^2/2}}} \right\}=\min\left\{1, {\frac{cx2^k}{6R^2Ce^{R^2/2}}} \right\}={\frac{cx2^k}{6R^2Ce^{R^2/2}}}
	\]
	because $x2^k\leq 1\leq {6R^2Ce^{R^2/2}}/c$ since we can choose $R^2,\,C$ large enough. 
	For Lemma \ref{le:S2}, in the statement we can replace $\min\left\{1,xt_n^{-1}\right\}$ by $\min\{K,xt_n^{-1}\}$ for every constant $K>0$. Now we take $K=c/8$ and we have
	\[
	\min\left\{\frac{c}{8},\frac{c}{4}\,x^2\,2^{2k}\left(x2^{k+1}\right)^{-1} \right\}=\frac{c}{8}x2^k.
	\]
	{Hence,} from \eqref{eqn:S1+S2} it follows that
	\[
	\begin{split}
	&\sum_{k=0}^\infty
	\p\left(\left\{
	\sup_{y\in I_k}
	\big|S_n(y)-S_n(U(a))\big|
	\geq \frac{c}{2}\,x^2\,2^{2k}\right\}\cap E_n\right)\\
	&\leq \sum_{k=0}^\infty K_1\exp\left(-K_2nx^32^{3k}\right)\\
	&\leq K_1\exp\left(-K_2nx^3\right)
	\end{split}
	\]
	{using that $nx^3\geq 1$.}
	
	Now we deal with the second probability in \eqref{eqn:S+R}. Using \eqref{eqn:bound_inverse_Phi_n} we obtain
	\[
	\begin{split}
	&\sum_{k=0}^\infty
	\p\left(\left\{
	\sup_{y\in I_k}
	|R_n(y)-R_n(U(a))|
	\geq \frac{c}{2}\,x^2\,2^{2k}\right\}\cap E_n\right)\\
	&\leq \sum_{k=0}^\infty
	\p\left(\left\{
	\sup_{y\in I_k}
	\sum_{i=1}^n \Delta_i\1_{\{U(a)<T_i\leq y\}}\left|\frac{1}{\Phi_n(T_i;\hat\beta_n)}-\frac{1}{\Phi(T_i;\beta_0)}\right|
	\geq n\frac{c}{2}\,x^2\,2^{2k}\right\}\cap E_n\right)\\
	&\leq \sum_{k=0}^\infty
	\p\left(\left\{
	\sum_{i=1}^n \Delta_i\1_{\{U(a)<T_i\leq U(a)+x2^{k+1}\}}
	\geq \frac{c}{2}n^{3/2-\alpha}\,x^2\,2^{2k}\right\}\cap E_n\right)\\
	&\leq \sum_{k=0}^\infty
	\p\left(\left\{
	\sum_{i=1}^n \left(\Delta_i\1_{\{U(a)<T_i\leq U(a)+x2^{k+1}\}}-\E\left[\Delta\1_{\{U(a)<T\leq U(a)+x2^{k+1}\}}\right]\right)
	\geq {\frac{c}{4}}n^{3/2-\alpha}\,x^2\,2^{2k}\right\}\cap E_n\right),
	\end{split}
	\]
	{using} that $n\E\left[\Delta\1_{\{U(a)<T\leq U(a)+x2^{k+1}\}}\right]=O(nx2^{k+1}){\ll}O(n^{3/2-\alpha}\,x^2\,2^{2k})$ because $x>n^{-1/3}$ and $\alpha$ can be chosen small enough. From  Bernstein inequality it follows that
	\[
	\begin{split}
	&\sum_{k=0}^\infty
	\p\left(\left\{
	\sup_{y\in I_k}
	|R_n(y)-R_n(U(a))|
	\geq \frac{c}{2}\,x^2\,2^{2k}\right\}\cap E_n\right)\\
	&\leq{\sum}_{k=0}^\infty \exp\left\{-c\frac{n^{3-2\alpha}\,x^4\,2^{4k}}{n^{3/2-\alpha}\,x^2\,2^{2k}+nx2^{k+1}} \right\}\\
	&\leq{\sum}_{k=0}^\infty \exp\left\{-c\,2^{2k} n^{3/2-\alpha}\,x^2\right\}\\
	&\leq K_1 \exp\left\{-K_2 nx^3\right\}
	\end{split}
	\]
	{where we used that $x\leq 1$ for the last inequality.}
	Here the constant $c$ changes from line to line.
	As a result, for the first term in the right hand side of~\eqref{eqn:main_inv_pr},
	we have that there exist constants $K_1,\,K_2>0$ such that for all $a>0$, $x>0$,
	\[
	\p\left(\left\{\hat{U}_n(a)\geq U(a)+x\right\}\cap E_n\right)\leq
	K_1\exp\left(-K_2nx^3\right).
	\]
	In the same way we can also deal with the second term   of~\eqref{eqn:main_inv_pr}. 
\end{proof}
\begin{lemma}
	\label{le:inv_tail_prob2}
	Under the assumptions of Lemma \ref{le:inv_tail_prob1}, there exist positive constants $K_1,\, K_2$ such that, 	for every  $x> 0$ and $a\notin \lambda_0([0,1])$,
	\begin{equation}
	\label{eqn:inv_2}
	\p\left(
	\left\{|\hat{U}_n(a)-U(a)|\geq x\right\}
	\cap E_n
	\right)
	\leq K_1\exp\left\{-K_2nx\left|\lambda_0(U(a))-a\right| \min\left(1,\left|\lambda_0(U(a))-a\right|\right)\right\}
	\end{equation}
	{provided that $\alpha$ in the definition of $E_n$ is chosen sufficiently small.}
	
\end{lemma}
\begin{proof}
	{If $\left|\lambda_0(U(a))-a\right|\leq x$ then it follows from Lemma \ref{le:inv_tail_prob1} that
		\begin{eqnarray*}
			\p\left(
			\left\{|\hat{U}_n(a)-U(a)|\geq x\right\}
			\cap E_n
			\right)
			&\leq&
			K_1\exp\left(-K_2nx^3\right)
		\end{eqnarray*}
		and \eqref{eqn:inv_2} follows. Moreover, if $x<\left|\lambda_0(U(a))-a\right|\leq n^{-1/3}$ then the inequality in \eqref{eqn:inv_2} is trivial since the upper bound in the inequality  is larger than one for appropriate choice of $K_1$ and $K_2$. Hence, in the sequel we assume $x<\left|\lambda_0(U(a))-a\right|$ and $\left|\lambda_0(U(a))-a\right|>n^{-1/3}$.}
	We argue as in the previous lemma, using 
	\[
	\Lambda_0(y)-\Lambda_0(U(a))\leq(y-U(a))\lambda_0(U(a)).
	\]	
	instead of \eqref{eqn:taylor_inv_pr}. We obtain
	\[
	\begin{split}
	&
	\p\left(\left\{
	\hat{U}_n(a)\geq U(a)+x \right\}\cap E_n\right)\\
	&\leq\p\left(\left\{
	\sup_{y\in[U(a)+x,1]}
	\left(\Lambda_n(y)-\Lambda_0(y)-\Lambda_n(U(a))+\Lambda_0(U(a))+(y-U(a))\left(\lambda_0(U(a))-a\right)\right)\geq 0\right\}\cap E_n\right).
	\end{split}
	\]
	With {$I_k$,} $S_n$ and $R_n$ defined as in Lemma \ref{le:inv_tail_prob1}, the previous probability can be bounded by 
	\begin{equation}
	\label{eqn:S+R_2}
	\begin{split}
	&\sum_{k=0}^\infty
	\p\left(\left\{
	\sup_{y\in I_k}
	\big|S_n(y)-S_n(U(a))\big|
	\geq \frac{1}{2}\,x\,2^{k}\left|\lambda_0(U(a))-a\right|\right\}\cap E_n\right)\\
	&\qquad+\sum_{k=0}^\infty
	\p\left(\left\{
	\sup_{y\in I_k}
	|R_n(y)-R_n(U(a))|
	\geq \frac{1}{2}\,x\,2^{k}\left|\lambda_0(U(a))-a\right|\right\}\cap E_n\right).
	\end{split}
	\end{equation}
	From Lemmas \ref{le:S1}, \ref{le:S2}, it follows that, 
	\[
	\begin{split}
	&\sum_{k=0}^\infty \p\left(
	\sup_{y\in I_k}
	\big|S_n(y)-S_n(U(a))\big|
	\geq \frac{1}{2}\,x\,2^{k}\left|\lambda_0(U(a))-a\right|\right)\\
	&\leq K_1\exp\Big(-K_2nx\left|\lambda_0(U(a))-a\right|\min\left\{1,\left|\lambda_0(U(a))-a\right| \right\}\Big).
	\end{split}
	\]
	Now we deal with the second probability in \eqref{eqn:S+R_2}. Using \eqref{eqn:bound_inverse_Phi_n} we obtain
	\[
	\begin{split}
	&\sum_{k=0}^\infty
	\p\left(\left\{
	\sup_{y\in I_k}
	|R_n(y)-R_n(U(a))|
	\geq \frac{1}{2}\,x\,2^{k}\left|\lambda_0(U(a))-a\right|\right\}\cap E_n\right)\\
	&\leq \sum_{k=0}^\infty
	\p\left(\left\{
	\sup_{y\in I_k}
	\sum_{i=1}^n \Delta_i\1_{\{U(a)<T_i\leq y\}}\left|\frac{1}{\Phi_n(T_i;\hat\beta_n)}-\frac{1}{\Phi(T_i;\beta_0)}\right|
	\geq
	{\frac{n}{2}}\,x\,2^{k}\left|\lambda_0(U(a))-a\right|\right\}\cap E_n\right)\\
	&\leq \sum_{k=0}^\infty
	\p\left(\left\{
	\sum_{i=1}^n \Delta_i\1_{\{U(a)<T_i\leq U(a)+x2^{k+1}\}}
	\geq cn^{3/2-\alpha}\,x\,2^{k}\left|\lambda_0(U(a))-a\right|\right\}\cap E_n\right)\\
	&\leq \sum_{k=0}^\infty
	\p\left(
	\sum_{i=1}^n \left(\Delta_i\1_{\{U(a)<T_i\leq U(a)+x2^{k+1}\}}-\E\left[\Delta\1_{\{U(a)<T\leq U(a)+x2^{k+1}\}}\right]\right){
		\geq} cn^{3/2-\alpha}\,x\,2^{k}\left|\lambda_0(U(a))-a\right|){\cap E_n}\right)
	\end{split}
	\]
	{using} that $n\E\left[\Delta\1_{\{U(a)<T\leq U(a)+x2^{k+1}\}}\right]=O\left(nx2^{k+1}){\ll}O(n^{3/2-\alpha}\,x\,2^{k}\left|\lambda_0(U(a))-a\right|\right)$ because {$\left|\lambda_0(U(a))-a\right|>n^{-1/3}$ and we can assume $\alpha<1/6$.} 
	From  Bernstein inequality it follows that
	\[
	\begin{split}
	&\sum_{k=0}^\infty
	\p\left(\left\{
	\sup_{y\in I_k}
	|R_n(y)-R_n(U(a))|
	\geq \frac{1}{2}\,x\,2^{k}\left|\lambda_0(U(a))-a\right|\right\}\cap E_n\right)\\
	&\leq{\sum}_{k=0}^\infty \exp\left\{-c\frac{n^{3-2\alpha}\,x^2\,2^{2k}\left|\lambda_0(U(a))-a\right|^2}{n^{3/2-\alpha}\,x\,2^{k}\left|\lambda_0(U(a))-a\right|+nx2^{k+1}} \right\}\\
	&\leq{\sum}_{k=0}^\infty \exp\left\{-c\,n^{3/2-\alpha}\,x\,2^{k}\left|\lambda_0(U(a))-a\right|\right\}\\
	&\leq K_1 \exp\left\{-K_2 nx\left|\lambda_0(U(a))-a\right|\right\}.
	\end{split}
	\]
	Again, the constant $c$ changes from line to line.
	As a result,  there exist constants $K_1,\,K_2>0$ such that for all $a>0$, $x>0$,
	\[ 
	\p\left(\left\{\hat{U}_n(a)\geq U(a)+x\right\}\cap E_n\right)\leq
	K_1\exp\left(-K_2nx\left|\lambda_0(U(a))-a\right|\min\left\{1,\left|\lambda_0(U(a))-a\right|\right\}\right).
	\]
	{The same bound can be obtained similarly for $\p\left(\left\{{U}(a)\geq \hat U_n(a)+x\right\}\cap E_n\right)$ and Lemma \ref{le:inv_tail_prob2} follows.}
\end{proof}
\begin{proof}[Proof of Theorem~\ref{theo:bound_expectation_L_P}]
	Let $t\in(0,1)$. We consider first
	\[
	I_1=\E\left[\1_{E_n}(\hat\lambda_n(t)-\lambda_0(t))_+^p\right]=\int_0^{\infty}\p\left[\left\{\hat\lambda_n(t)-\lambda_0(t)>x \right\}\cap E_n \right]px^{p-1}\,\dd x.
	\]
	Using the switching relation $\hat\lambda_n(t)>\lambda_0(t)+x\Rightarrow \hat{U}_n(\lambda_0(t)+x){\geq }t$, we obtain
	\begin{equation}
	\label{eqn:I_1}
	\begin{split}
	I_1&\leq \int_0^{\infty}\p\left[\left\{\hat{U}_n(\lambda_0(t)+x){\geq }t \right\} \cap E_n\right]px^{p-1}\,\dd x\\
	&=\int_0^{\lambda_0(0)-\lambda_0(t)} \p\left[\left\{\hat{U}_n(\lambda_0(t)+x){\geq }t \right\} \cap E_n\right]\,px^{p-1}\,\dd x\\
	&\quad+\int_{\lambda_0(0)-\lambda_0(t)}^{\infty}\p\left[\left\{\hat{U}_n(\lambda_0(t)+x){\geq }t \right\} \cap E_n\right]px^{p-1}\,\dd x.
	\end{split}
	\end{equation}
	For ${x\in[0,\lambda_0(0)-\lambda_0(t))}$ there exists $c>0$ such that $U(\lambda_0(t)+x)<t-cx$. Hence, {it follows from Lemma \ref{le:inv_tail_prob1} that}
	\[
	\begin{split}
	&\int_0^{\lambda_0(0)-\lambda_0(t)}\p\left[\left\{\hat{U}_n(\lambda_0(t)+x){\geq }t \right\} \cap E_n\right]\,px^{p-1}\,\dd x\\
	&\leq\int_0^{\lambda_0(0)-\lambda_0(t)}\p\left[\left\{\left|\hat{U}_n(\lambda_0(t)+x)-U(\lambda_0(t)+x)\right|>cx \right\}\cap E_n \right]\,px^{p-1}\,\dd x\\
	&\lesssim \int_0^{\infty} \exp\left\{-K_2nx^3 \right\}\,px^{p-1}\,\dd x\\
	&\lesssim n^{-p/3}\int_0^{\infty} \exp\left\{-K_2y^3 \right\}\,py^{p-1}\,\dd y\\
	&\lesssim
	n^{-p/3}.
	\end{split}
	\]
	Since, for $x>\lambda_0(0)-\lambda_0(t)$, $U(\lambda_0(t)+x))=0$,  we obtain { from Lemma \ref{le:inv_tail_prob1} that}
	\[
	\begin{split}
	&\int_{\lambda_0(0)-\lambda_0(t)}^{\lambda_0(0)-\lambda_0(t)+t}\p\left[\left\{\hat{U}_n(\lambda_0(t)+x){\geq}t \right\} \cap E_n\right]px^{p-1}\,\dd x\\
	&\lesssim \exp\left\{-K_2nt^3 \right\} \int_{0}^{\lambda_0(0)-\lambda_0(t)+t}px^{p-1}\,\dd x\\
	&\lesssim \exp\left\{-K_2nt^3 \right\}t^{p}\\
	&\lesssim \frac{t^{p}}{(nt^3)^{p/3}}\\
	&\lesssim n^{-p/3}.
	\end{split}
	\]
	Moreover, from  Lemma \ref{le:inv_tail_prob2}, it follows that
	\begin{equation}
	\label{eqn:bound_expectation1}
	\begin{split}
	&\int_{\lambda_0(0)-\lambda_0(t)+t}^\infty \p\left[\left\{\hat{U}_n(\lambda_0(t)+x){\geq}t \right\} \cap E_n\right]px^{p-1}\,\dd x\\
	&\lesssim \int_{\lambda_0(0)-\lambda_0(t)+t}^{\lambda_0(0)-\lambda_0(t)+1} \exp\left\{-K_2nt\left|\lambda_0(0)-\lambda_0(t)-x\right|^{2} \right\} px^{p-1}\,\dd x\\
	&\quad +\int_{\lambda_0(0)-\lambda_0(t)+1}^\infty \exp\left\{-K_2nt\left|\lambda_0(0)-\lambda_0(t)-x\right|\right\} px^{p-1}\,\dd x.
	\end{split}
	\end{equation}
	For the first integral in the right hand side of \eqref{eqn:bound_expectation1} we have
	\[
	\begin{split}
	&\int_{\lambda_0(0)-\lambda_0(t)+t}^{\lambda_0(0)-\lambda_0(t)+1} \exp\left\{-K_2nt\left|\lambda_0(0)-\lambda_0(t)-x\right|^{2} \right\} px^{p-1}\,\dd x\\
	&=\int_{t}^{1} \exp\left\{-K_2nty^{2} \right\} p\left(y+\lambda_0(0)-\lambda_0(t)\right)^{p-1}\,\dd y\\
	&\lesssim\int_{t}^{1} \exp\left\{-K_2nty^{2} \right\} y^{p-1}\,\dd y\\
	&\lesssim (nt)^{-p/2}\int_{0}^{\infty} \exp\left\{-K_2z^{2} \right\}z^{p-1}\,\dd z\\
	&\lesssim (nt)^{-p/2}.
	\end{split}
	\]
	Now we consider the second integral in the right hand side of \eqref{eqn:bound_expectation1}
	\[
	\begin{split}
	&\int_{\lambda_0(0)-\lambda_0(t)+1}^{\infty} \exp\left\{-K_2nt\left|\lambda_0(0)-\lambda_0(t)-x\right| \right\} px^{p-1}\,\dd x\\
	&=\int_{1}^{\infty} \exp\left\{-K_2nty \right\} p\left(y+\lambda_0(0)-\lambda_0(t)\right)^{p-1}\,\dd y\\
	&\lesssim\int_{1}^{\infty} \exp\left\{-K_2nty \right\} y^{p-1}\,\dd y\\
	&\lesssim (nt)^{-p}\int_{0}^{\infty} \exp\left\{-K_2z \right\}z^{p-1}\,\dd z\\
	&\lesssim (nt)^{-p}.
	\end{split}
	\]
	{For $t>n^{-1}$ we have} $(nt)^{-p}\leq(nt)^{-p/2}$ and {therefore,
		\[
		I_1\lesssim n^{-p/3}+\left(nt\right)^{-p/2}.
		\]}
	In particular,  for $t\geq n^{-1/3}$ we get {$I_1\lesssim n^{-p/3}$.}
	Similarly, 
	\[
	I_2=\E\left[\1_{E_n}(\lambda_0(t)-\hat\lambda_n(t))_+^p\right]\lesssim n^{-p/3}+\left(n(1-t)\right)^{-p/2}
	\]
	and the result follows.
\end{proof}

\subsection{Proof of Theorem \ref{theo:CLT}}
{In this section, we assume that the assumptions of Theorem \ref{theo:CLT} are fulfilled.}
\begin{lemma}
	\label{le:embedding}
	Let 
	\begin{equation}
	\label{def:tilde_M_n}
	\begin{split}
	\tilde{M}_n(t)&=\frac{1}{n}\sum_{i=1}^{n}\left(\Delta_i\1_{\{T_i\leq t\}}-\E\left[\Delta_i\1_{\{T_i\leq t\}} \right]\right)\\
	&=\frac{1}{n}\sum_{i=1}^{n}\Delta_i\1_{\{T_i\leq t\}}-H^{uc}(t).
	\end{split}
	\end{equation}
	There exist a Brownian bridge $B_n$ and positive constants $K,\,c_1, \,c_2$ such that 
	\[
	\p\left[n\sup_{s\in[0,1]}\left|\tilde{M}_n(s)-n^{-1/2}B_n\left(H^{uc}(s)\right) \right|\geq x+c_1\log n \right]\leq K\exp\left(-c_2 x\right).
	\]
\end{lemma}
\begin{proof}
	Let $U_1,\dots,U_n$ be independent uniform random variables on $(1,2)$. For $i\in\{1,\dots,n\}$, define
	\[
	Y_i=\begin{cases}
	T_i & \text{ if } \Delta_i=1,\\
	U_i & \text{ if } \Delta_i=0.
	\end{cases}
	\]
	The variables $Y_i$ are i.i.d with distribution function $G_Y$ given by
	\[
	G_Y(t)=\begin{cases}
	0 & \text{ if } t<0,\\
	H^{uc}(t) & \text{ if } 0\leq t\leq 1,\\
	H^{uc}(1)+\p(\Delta=0)(t-1) & \text{ if } 1<t\leq 2,\\
	1 & \text{ if } t>2.
	\end{cases}
	\] 
	By the strong approximation result in  \cite{komlosmajortusnady1975}, there exists a Brownian bridge $B_n$ and positive constants $K,\,c_1, \,c_2$ such that 
	\[
	\p\left[n\sup_{t\in[0,2]}\left|\frac{1}{n}\sum_{i=i}^n\1_{\{Y_i\leq t\}}-G_Y(t)-n^{-1/2}B_n\left(G_Y(t)\right) \right|\geq x+c_1\log n \right]\leq K\exp\left(-c_2 x\right).
	\] 
	In particular,
	\[
	\p\left[n\sup_{t\in[0,1]}\left|\frac{1}{n}\sum_{i=i}^n\1_{\{Y_i\leq t\}}-G_Y(t)-n^{-1/2}B_n\left(G_Y(t)\right) \right|\geq x+c_1\log n \right]\leq K\exp\left(-c_2 x\right).
	\] 
	The statement follows from the fact that, when  $t\in[0,1]$, then $G_Y(t)=H^{uc}(t)$ and
	\[
	\frac{1}{n}\sum_{i=i}^n\1_{\{Y_i\leq t\}}=\frac{1}{n}\sum_{i=i}^n\Delta_i\1_{\{T_i\leq t\}}.
	\]
\end{proof}
\begin{lemma}
	\label{le:boundaries}
	$\hat{\lambda}_n(0)$ and $\hat{\lambda}_n(1)$ are stochastically bounded.
\end{lemma}
\begin{proof}
	{The sub-distribution function of the observed event times has a bounded density on $[0,1]$ since $h^{uc}(t)=\lambda_0(t)\Phi(t;\beta_0)$ and both $\lambda_0$ and $\Phi$ are assumed to be continuous (hence bounded on $[0,1]$). Therefore, }
	{Lemma \ref{le:boundaries}} follows from Lemma 1 in \cite{durot2007} because for every $\delta>0$, the probability that $\Lambda_n$ jumps in $(0,\delta/n)$ or in $(1-\delta/n,1)$ is no more than
	\[
	n\p\left[\Delta T\in (0,\delta/n)\cup(1-\delta/n,1)\right]\leq 2K\delta,
	\]
	where $K=\sup_{u\in[0,1]}h^{uc}(u)$.
\end{proof}

\begin{proof}[Proof of Lemma \ref{le:Lambda_n-Lambda}]
	Let  $M_n=\Lambda_n-\Lambda_0$ and $\pi(u;t,t_n)=\1_{\left\{u\leq t_n\right\}}-\1_{\left\{u\leq t\right\}}$. Using \eqref{eqn:Breslow_linear}, we can write
	\begin{equation}
	\label{eqn:Lambda_n-Lambda}
	\begin{split}
	M_n(t_n)-M_n(t)
	&=\int\frac{\delta {\pi}(u;t,t_n)}{\Phi(u;\beta_0)}\,\dd (\p_n-\p)(u,\delta,z)+{(\hat\beta_n-\beta_0)'}\left[A_0\left(t_n\right)-A_0(t)\right]\\
	&\quad-\int e^{{\beta_0' z}}\int_{u\wedge t}^{u\wedge t_n}\frac{\lambda_0(v)}{\Phi(v;\beta_0)}\,\dd v\,\dd(\p_n-\p)(u,\delta,z)+R_n(t_n)-R_n(t).
	\end{split}
	\end{equation}
	Let's consider the first term in the right hand side. 
	We have
	\begin{equation}
	\label{eqn:term1_1}
	\begin{split}
	&\int\frac{\delta {\pi}(u;t,t_n)}{\Phi(u;\beta_0)}\,\dd (\p_n-\p)(u,\delta,z)=\frac{1}{\Phi(t;\beta_0)}\int \delta {\pi}(u;t,t_n)\,\dd (\p_n-\p)(u,\delta,z)+R^1_n(t,t_n)
	\end{split}
	\end{equation}
	where
	\[
	R^1_n(t,t_n)=\int\delta {\pi}(u;t,t_n)\left(\frac{1}{\Phi(u;\beta_0)}-\frac{1}{\Phi(t;\beta_0)}\right)\,\dd (\p_n-\p)(u,\delta,z).
	\]
	Moreover, in view of Lemma \ref{le:embedding}, we write 
	\begin{equation}
	\label{eqn:term1_2}
	\begin{split}
	\int \delta {\pi}(u;t,t_n)\,\dd (\p_n-\p)({u},\delta,z)&=n^{-1/2}\left[B_n\left(H^{uc}(t_n)\right)-B_n\left(H^{uc}(t)\right)\right]\\
	&\quad+\left[(\tilde{M}_n-n^{-1/2}B_n\circ H^{uc})(t_n)-(\tilde{M}_n-n^{-1/2}B_n\circ H^{uc})(t)\right]
	\end{split}
	\end{equation}
	where $\tilde{M}_n$ is defined in \eqref{def:tilde_M_n}.
	For the third term in \eqref{eqn:Lambda_n-Lambda}, 
	if $t_n>t$, we have
	\[
	\begin{split}
	&\int e^{\beta'_0 z}\int_{u\wedge t}^{u\wedge t_n}\frac{\lambda_0(v)}{\Phi(v;\beta_0)}\,\dd v\,\dd(\p_n-\p)(u,\delta,z)\\
	&=\int e^{\beta'_0 z}\1_{\{u>t_n\}}\,\dd(\p_n-\p)(u,\delta,z) \int_{t}^{t_n}\frac{\lambda_0(v)}{\Phi(v;\beta_0)}\,\dd v\\
	&\quad+\int e^{\beta'_0 z}\1_{\{t<u\leq t_n\}}\int_{t}^{u}\frac{\lambda_0(v)}{\Phi(v;\beta_0)}\,\dd v\,\dd(\p_n-\p)(u,\delta,z).
	\end{split}
	\]
	Moreover
	\[
	\begin{split}
	&\int e^{\beta'_0 z}\1_{\{u>t_n\}}\,\dd(\p_n-\p)(u,\delta,z) \int_{t}^{t_n}\frac{\lambda_0(v)}{\Phi(v;\beta_0)}\,\dd v\\
	&=\int e^{\beta'_0 z}\1_{\{u>t\}}\,\dd(\p_n-\p)(u,\delta,z) \frac{\lambda_0(t)}{\Phi(t;\beta_0)}\left(t_n-t\right)\\
	&\quad+\frac{1}{2}\int e^{\beta'_0 z}\1_{\{u>t\}}\,\dd(\p_n-\p)(u,\delta,z)\int_{t}^{t_n}\left(\frac{\lambda_0}{\Phi}\right)'(s_{n,v})\left(v-t\right)^2\,\dd v\\
	&\quad-\int e^{\beta'_0 z}\1_{\{t<u\leq t_n\}}\,\dd(\p_n-\p)(u,\delta,z) \int_{t}^{t_n}\frac{\lambda_0(v)}{\Phi(v;\beta_0)}\,\dd v,
	\end{split}
	\]
	for some $t\leq s_{n,v}\leq v$.
	Similarly we can deal with the case $t_n<t$ obtaining that for all $t_n\in(0,1)$,
	\begin{equation}
	\label{eqn:second_term_breslow_linear}
	\begin{split}
	&\int e^{\beta'_0 z}\int_{u\wedge t}^{u\wedge t_n}\frac{\lambda_0(v)}{\Phi(v;\beta_0)}\,\dd v\,\dd(\p_n-\p)(u,\delta,z)\\
	&=\left[\Phi_n(t;\beta_0)-\Phi(t;\beta_0)\right] \frac{\lambda_0(t)}{\Phi(t;\beta_0)}\left(t_n-t\right)\\
	&\quad+\frac{1}{2}\left[\Phi_n(t;\beta_0)-\Phi(t;\beta_0)\right]\int_{t}^{t_n}\left(\frac{\lambda_0}{\Phi}\right)'(s_{n,v})\left(v-t\right)^2\,\dd v-R^2_n(t,t_n)
	\end{split}
	\end{equation}
	where
	\[
	R^2_n(t,t_n)=\int {e^{\beta_0' z}\pi}(u;t,t_n)\int_{u}^{t_n}\frac{\lambda_0(v)}{\Phi(v;\beta_0)}\,\dd v\,\dd(\p_n-\p)(u,\delta,z).
	\]
	Finally, for the second term in \eqref{eqn:Lambda_n-Lambda}, by a Taylor expansion we obtain
	\begin{equation}
	\label{eqn:taylor_A_0}
	\begin{split}
	A_0\left(t_n\right)-A_0(t)=A'_0(t)\left(t_n-t \right) +\frac{1}{2}A''_0(\zeta_n)\left(t_n-t \right)^2
	\end{split}
	\end{equation}
	where $|\zeta_n-t|\leq \left| t_n-t\right|$. 
	Note that by definition
	\[
	\begin{split}
	A''_0(t)=\frac{\dd}{\dd t}D^{(1)}(t;\beta_0)\frac{\lambda_0(t)}{\Phi(t;\beta_0)}+D^{(1)}(t;\beta_0)\frac{\dd}{\dd t}\frac{\lambda_0(t)}{\Phi(t;\beta_0)}
	\end{split}
	\]
	with
	\[
	\frac{\dd}{\dd t}D^{(1)}(t;\beta_0)=\frac{\dd}{\dd t}\int ze^{\beta'_0z}\1_{\{u\geq t\}}\,\dd \p(u,\delta,z)=\frac{\dd}{\dd t}\E\left[Ze^{\beta'_0Z}\p\left(T\geq t|Z\right)\right].
	\]
	Hence, $A''_0$ is continuous because we are assuming that, given $z$, the follow up time $T$ has a continuous density and $\lambda_0$, $\Phi$ are continuously differentiable. 
	
	Let $E_n$ the event in \eqref{def:E_n}.
	Combining \eqref{eqn:Lambda_n-Lambda}-\eqref{eqn:taylor_A_0}, we get
	\[
	\begin{split}
	M_n(t_n)-M_n(t)
	&=n^{-1/2}{\frac{1}{\phi(t;\beta_0)}}\left[B_n\left(H^{uc}(t_n)\right)-B_n\left(H^{uc}(t)\right)\right]+(\hat\beta_n-\beta_0)'A'_0(t)\left(t_n-t \right)\\
	&-\left[\Phi_n(t;\beta_0)-\Phi(t;\beta_0)\right] \frac{\lambda_0(t)}{\Phi(t;\beta_0)}\left(t_n-t\right)
	+R^1_n(t,t_n)+R^2_n(t,t_n)+\tilde{R}_n(t,t_n),
	\end{split}
	\]
	with
	\[
	\1_{E_n}\sup_{|t_n-t|\leq {c}n^{-1/3}T_n}|\tilde{R}_n(t,t_n)|\lesssim \sup_{t\in[0,1]}\left|\tilde{M}_n(t)-n^{-1/2}B_n(H^{uc}(t))\right|+ n^{-7/6+\alpha}T_n^2+n^{-1+\alpha}.
	\]
	{Here, we used that $E_n$ contains the events in \eqref{eq:E2E3}.}

	In order to prove the statement of the Lemma, it is sufficient to consider $x>n^{-1+1/q}$ because otherwise the bound is trivial. For such values of $x$ and for $\alpha<1/q$, from the strong approximation result in Lemma \ref{le:embedding}, it follows that
	\[
	\p\left[\left\{\sup_{|t_n-t|\leq {c}n^{-1/3}T_n}|\tilde{R}_n(t,t_n)|>x\right\}\cap E_n \right]\lesssim \exp\left\{-K_1nx\right\}\lesssim x^{-q}n^{1-q}.
	\]
	Moreover the bound is uniform over $t\in[0,1]$.
	The statement then follows from Lemmas \ref{le:R_1} and \ref{le:R_2} with
	\[
	r_n(t,t_n)=R^1_n(t,t_n)+R^2_n(t,t_n)+\tilde{R}_n(t,t_n)
	\]
\end{proof}

\begin{proof}[Proof of Theorem~\ref{theo:CLT}.]
	Let 
	\[
	\mathcal{J}_n=n^{p/3}\int_0^1\left|\hat{\lambda}_n(t)-\lambda_0(t)\right|^p\,\text{d}t.
	\]
	
	\textbf{Step 1.} It can be shown exactly as in Step 1 of \cite{durot2007} that
	\[
	\mathcal{J}_n\1_{E_n}=\1_{E_n}n^{p/3}\int_{\lambda_0(1)}^{\lambda_0(0)}\left|\hat{U}_n(a)-U(a)\right|^p\left|U'(a)\right|^{1-p}\,\text{d}a+o_P(n^{-1/6}).
	\]
	Details are given in the Supplementary Material \cite{DM_supp}.
	
	\textbf{Step 2.} Let $B_n$ be the Brownian bridge from Lemma \ref{le:Lambda_n-Lambda} and $\xi_n$  a standard Gaussian random variable independent of $B_n$. Then $W_n(t)=B_n(t)+\xi_nt$  
	is a Brownian motion.
	Let 
	\[
	d(t)=\frac{|\lambda'_0(t)|\Phi(t;\beta_0)}{2h^{uc}(t)^2}
	\]
	and 
	\[
	W_t(u)=n^{1/6}\left[W_n\left(H^{uc}(t)+n^{-1/3}u\right)-W_n\left(H^{uc}(t)\right)\right].
	\]
	For each $t\in[0,1]$, $W_t$ is a standard {two-sided} Brownian motion and we define $\tilde{V}{(t)}$ as 
	\begin{equation}
	\label{eq: Vtilde}
	\tilde{V}(t)=\argmax_{|u|\leq\log n}\left\{W_t(u)-d(t)u^2\right\}.
	\end{equation}
	We show that, for some $\eta_n$ which satisfies $\sup_{t\in[0,1]}|{\eta}_n(t)|=O_P(1)$, we have	 
	\begin{equation}
	\label{eqn:step2}
	\mathcal{J}_n=\int_b^{1-b}\left|\tilde{V}(t)-n^{-1/6}{\eta}_n(t) \right|^p\left|\frac{\lambda'_0(t)}{h^{uc}(t)}\right|^p\,\dd t+o_P(n^{-1/6}).
	\end{equation}	
	Let $k$ be a  twice differentiable, symmetric kernel function with bounded support $[-1,1]$ such that $k(u)>0$ for all $u\in(-1,1)$ and $\int k(y)\,\dd y =1$. 
	We denote by $k_b$ its scaled version $k_b(u)=b^{-1}k(u/b).$ Here, $b$  is a bandwidth that depends on the sample size and we consider $b=n^{-1/4}$. 
	Let 
	\[
	\Psi_n^s(t)={\int_{\mathbb{R}}} k_b(t-v)\left(\Phi_n(v;\beta_0)-\Phi(v;\beta_0)\right)\,\dd v.
	\]
	Note that 
	\[
	\sup_{t\in[b,1-b]}\left|\Psi_n^s(t) \right|\leq \sup_{v\in[0,1]}\left| \Phi_n(v;\beta_0)-\Phi(v;\beta_0)\right|=O_P(n^{-1/2})
	\]
	and
	\[
	\begin{split}
	\sup_{t\in[b,1-b]}\left|\frac{\dd\Psi_n^s(t)}{\dd t}\right|&=\frac{1}{b^2}\left|\int k'\left(\frac{t-v}{b}\right)\left(\Phi_n(v;\beta_0)-\Phi(v;\beta_0)\right)\,\dd v\right|\\
	&{=\frac{1}{b^2}\left|\int k'\left(\frac{t-v}{b}\right)\left[\left(\Phi_n(v;\beta_0)-\Phi(v;\beta_0)\right)-\left(\Phi_n(t;\beta_0)-\Phi(t;\beta_0)\right) \right]\,\dd v\right|}\\
	&\leq b^{-1}{\sup_{|v-t|\leq b}\left| \left(\Phi_n(v;\beta_0)-\Phi(v;\beta_0)\right)-\left(\Phi_n(t;\beta_0)-\Phi(t;\beta_0)\right)\right|}{\times \int_{\mathbb R}|k'(v)|\dd v}\\
	&={O_P\left(n^{-1/2}b^{-1/2}\right)=o_P\left(n^{-1/6}\right)}.
	\end{split}
	\]
	Here we used the fact that $\int k'(v)\,\dd v=0$ and that
	\[
	\sup_{|v-t|\leq b}\left| \left(\Phi_n(v;\beta_0)-\Phi(v;\beta_0)\right)-\left(\Phi_n(t;\beta_0)-\Phi(t;\beta_0)\right)\right|=O_P\left(n^{-1/2}\sqrt{b}\right),
	\]
	which can be shown as in Lemma~\ref{le:R_3}.

	For every $a\in\R$, define
	\[
	\begin{split}
	a^\xi=a-n^{-1/2}\xi_n\frac{h^{uc}(U(a))}{\Phi(U(a);\beta_0)}+(\hat\beta_n-\beta_0)A'_0(U(a))-\Psi_n^s(U(a))\frac{\lambda_0(U(a))}{\Phi(U(a);\beta_0)}
	\end{split}
	\]
	{where $A_0$ is taken from \eqref{def:A_0}.}
	Note that the smoothed version $\Psi_n^s$ of $\Phi_n-\Phi$ is needed in order to make the function $a\mapsto a^\xi$ differentiable. {Note also that since $h^{uc}(t)=\lambda_{0}(t)\Phi(t;\beta_0)$, it follows from the assumptions on $\lambda$ and $\Phi$ that also $h^{uc}$ is differentiable. }{The differentiability will allow us to do the change of variable $a\mapsto a^\xi$ below.} 
	
	Let $E_n$ be the event in \eqref{def:E_n}. For some proper modification of $E_n$, we can assume that on the event $E_n$ it holds {that $|\xi_n|\leq\log n$ and}
	\begin{equation}
	\label{eqn:E_n_add}
	\sup_{a\in\R}\left|a-a^\xi\right|\leq cn^{-1/2+\alpha}
	\end{equation}
	{for an appropriate $c>0$, with $\alpha$ as in \eqref{eq:E2E3}.}
	From Lemma 6(i) in \cite{durot2007} and the change of variable $a\to a^{\xi}$ we have
	\[
	\begin{split}
	\mathcal{J}_n\1_{E_n}&=\1_{E_n}n^{p/3}\int_{\lambda_0(1)}^{\lambda_0(0)}\left|\frac{H^{uc}(\hat{U}_n(a))-H^{uc}(U(a))}{h^{uc}(U(a))}\right|^p\left|U'(a)\right|^{1-p}\,\text{d}a+o_P(n^{-1/6})\\
	&=\1_{E_n}n^{p/3}\int_{\lambda_0(1)+cb}^{\lambda_0(0)-cb}\left|\frac{H^{uc}(\hat{U}_n(a))-H^{uc}(U(a))}{h^{uc}(U(a))}\right|^p\left|U'(a)\right|^{1-p}\,\text{d}a+o_P(n^{-1/6})\\
	&=\1_{E_n}n^{p/3}\int_{J_n}\left|\frac{H^{uc}(\hat{U}_n(a^\xi))-H^{uc}(U(a^\xi))}{h^{uc}(U(a))}\right|^p\left|U'(a)\right|^{1-p}\,\text{d}a+o_P(n^{-1/6}),
	\end{split}
	\]
	with $c=\inf_{t\in[0,1]}|\lambda'_0(t)|$ and  
	\[
	J_n=\left[\lambda_0(1)+n^{-1/6}/\log n,\lambda_0(0)-n^{-1/6}/\log n\right] .
	\]  {To do the change of variable above,} we first restrict the interval into $[\lambda_0(1)+cb,\lambda_0(0)-cb]$ {because} the previous bounds on $\Psi^s_n(t)$ and its derivative need $t\in[b,1-b]$. {The definition of $a^\xi$ involves} $\Psi_n^s(U(a))$ so we need $U(a)\in[b,1-b]$ {whereas} if $a>\lambda_{0}(1){+}cb$ then $U(a)<1-b$ for this choice of $c$. Afterwards we can replace $\dd a^\xi$ by $\dd a$ because, on $E_n$, $\sup_a |\dd a-\dd a^\xi|\leq C(n^{-1/2+\alpha}+n^{-3/8})\,\dd a$ for $\alpha<1/3$. The last step consists in replacing $h^{uc}(U(a^\xi))$ by $h^{uc}(U(a))$ and $U'(a^\xi)$ by $U'(a)$  using Lemma 6(i) in \cite{durot2007} with for example $r=r'=2$. Here we also use that $U'$, $(h^{uc})'$ are uniformly bounded and $U'$ satisfies \eqref{eqn:assumption_derivative}.   
	By definition of $\hat{U}_n$, it follows that 
	\[
	H^{uc}(\hat{U}_n(a^\xi))=\argmax_{u\in[H^{uc}(0),H^{uc}(1)]}\left\{ \left(\Lambda_n\circ (H^{uc})^{-1}-a^\xi(H^{uc})^{-1}\right)(u)\right\}.
	\]
	Let $  t_n=(H^{uc})^{-1}(H^{uc}(U(a))+n^{-1/3}u)$. 
	Using properties of the $\argmax$ function we obtain
	\[
	n^{1/3}\left(H^{uc}(\hat{U}_n(a^\xi))-H^{uc}(U(a))\right)=\argmax_{u\in I_n(a)}\left\{D_n(a,u)+S_n(a,u) \right\},
	\]
	where
	\[
	I_n(a)=\left[-n^{1/3}\left(H^{uc}(U(a))-H^{uc}(0)\right),n^{1/3}\left(H^{uc}(1)-H^{uc}(U(a))\right)\right],
	\]
	\[
	\begin{split}
	D_n(a,u)=n^{2/3}\Phi(U(a);\beta_0)\left\{\left(\Lambda_0(t_n)-at_n\right)- \left(\Lambda_0(U(a))-aU(a)\right) \right\}
	\end{split}
	\]
	and
	\begin{equation*}
	\begin{split}
	S_n(a,u)=n^{2/3}\Phi(U(a);\beta_0)\left\{ (a-a^\xi)\left[t_n-U(a)\right]+(\Lambda_n-\Lambda_0)\left(t_n\right)-(\Lambda_n-\Lambda_0)(U(a))\right\}.
	\end{split}
	\end{equation*}
	Using Lemma \ref{le:Lambda_n-Lambda} and the definition of $a^\xi$, we get
	\[
	S_n(a,u)=W_{U(a)}(u)+R^3_n(a,u)+R^4_n(a,u)
	\]
	where
	\[
	R^3_n(a,u)=n^{2/3}\left\{\Psi_n^s(U(a)) -\left[\Phi_n(U(a);\beta_0)-\Phi(U(a);\beta_0)\right] \right\}\lambda_0(U(a))\left[t_n-U(a)\right],
	\]
	and
	\[
	R^4_n(a,u)=n^{2/3}\Phi(U(a);\beta_0)\,r_n\left(U(a),t_n\right),
	\]
	with $r_n$ as in Lemma \ref{le:Lambda_n-Lambda}.
	Let $R_n(a,u)=R^3_n(a,u)+R^4_n(a,u)$.	
	Now we use Lemma 5 in \cite{durot2007} to show that $R_n$ is negligible. First we localize. Let $q>12$. {We will apply Lemmas 	\ref{le:Lambda_n-Lambda},
		\ref{le:R_1}, \ref{le:R_2}, \ref{le:R_3} with such a $q$ which is possible because now we choose $T_n=n^{1/(3(6q-11))}$, so $\gamma=1/(3(6q-11))$ and $q\gamma<1$.} {The choice of $T_n$ is motivated by Lemma 5 of \cite{durot2007} that is used below.} Define
	\[
	\tilde{U}_n{(a)}=\argmax_{|u|\leq T_n}\left\{D_n(a,u)+W_{U(a)}(u)+R_n(a,u) \right\}
	\]
	For large $n$, $[-T_n,T_n]\subset I_n(a)$ for all $a\in J_n$, so {we have the equivalence}
	\[
	\1_{E_n}\tilde{U}_n{(a)}\neq \1_{E_n}n^{1/3}\left(H^{uc}(\hat{U}_n(a^\xi))-H^{uc}(U(a))\right)\quad\Leftrightarrow\quad \1_{E_n}n^{1/3}\left|H^{uc}(\hat{U}_n(a^\xi))-H^{uc}(U(a))\right|>T_n
	\]
	Since $h^{uc}$ and $U'$ are bounded, by monotonicity of $\hat{U}_n$ {together with \eqref{eqn:E_n_add}, that holds on $E_n$,} and Lemma \ref{le:inv_tail_prob1}, {for all $a\in\R$, $x\geq n^{-1/3}$ and sufficiently large $n$} we then have 
	\begin{equation}
	\label{eqn:bound_inv_a_xi}
	\begin{split}
	&\p\left(\left\{\left|H^{uc}(\hat{U}_n(a^\xi))-H^{uc}(U(a))\right|>  x\right\}\cap E_n\right)\\
	&\leq\p\left(\left\{\left|\hat{U}_n(a^\xi)-U(a)\right|>Cx \right\}\cap E_n\right)\\
	&\leq\p\left(\left\{\hat{U}_n(a-cn^{-1/2+\alpha})>Cx+U(a) \right\}\cap E_n\right)+\p\left(\left\{\hat{U}_n(a+cn^{-1/2+\alpha})<U(a)-Cx \right\}\cap E_n\right)\\
	&\leq\p\left(\left\{\left|\hat{U}_n(a-cn^{-1/2+\alpha})-U(a-cn^{-1/2+\alpha})\right|>\frac{C}{2}x \right\}\cap E_n\right)\\
	&\quad+\p\left(\left\{\left|\hat{U}_n(a+cn^{-1/2+\alpha})-U(a+cn^{-1/2+\alpha})\right|>\frac{C}{2}x \right\}\cap E_n\right)\\
	&\leq K_1\exp(-K_2nx^3),
	\end{split}
	\end{equation}
	using that $\alpha<1/6$ for the penultimate inequality,  {since $\alpha$ can be chosen arbitrarily small.} Since the upper bound is greater than one for appropriate $K_1,K_2$ and all $x\in[0,n^{-1/3}]$, the above inequality holds true for all $a\in\R$, $x\geq 0$ and $n$. Combining the two preceding displays yields
	\begin{equation}
	\label{eq: lem6}
	\p\left(\left\{\tilde{U}_n(a)\neq n^{1/3}\left(H^{uc}(\hat{U}_n(a^\xi))-H^{uc}(U(a))\right)\right\}\cap E_n\right)\leq K_1\exp(-K_2T_n^3).
	\end{equation}
	Moreover,
	integrating inequality \eqref{eqn:bound_inv_a_xi}, we obtain 
	\begin{equation}
	\label{eqn:bound_exp_inv_a_xi}
	\sup_{a\in\R}\E\left[\1_{E_n}n^{q'/3}\left|H^{uc}(\hat{U}_n(a^\xi))-H^{uc}(U(a))\right|^{q'}\right]\leq K,
	\end{equation}
	for every $q'\geq 1$. From \eqref{eqn:bound_inv_a_xi}, \eqref{eqn:bound_exp_inv_a_xi} and Lemma 6(ii) in \cite{durot2007} with $r$ and $r'$ chosen arbitrarily we get
	\[
	\begin{split}
	\mathcal{J}_n\1_{E_n}&=\1_{E_n}\int_{J_n}\left|\tilde{U}_n(a)+n^{1/3}\left(H^{uc}(U(a))-H^{uc}(U(a^\xi))\right) \right|^p\frac{|U'(a)|^{1-p}}{h^{uc}(U(a))^p}\,\dd a+o_P(n^{-1/6})\\
	&=\1_{E_n}\int_{J_n}\left|\tilde{U}_n(a)-n^{-1/6}\eta_n(a) \right|^p\frac{|U'(a)|^{1-p}}{h^{uc}(U(a))^p}\,\dd a+o_P(n^{-1/6})
	\end{split}
	\]
	where
	\[
	\eta_n(a)=n^{1/2}(a-a^\xi)|U'(a)|h^{uc}(U(a))=O_P(1).
	\]
	{The rest in the penultimate equality is of order $o_P(n^{-1/6})$ because $h^{uc}$ and $U$ are continuously differentiable (this follows from the assumptions on $\lambda_{0}$ and $\Phi$), hence their derivatives are uniformly bounded, and $\alpha$ can be chosen as small as we want.} 
	Now we approximate $\tilde{U}_n{(a)}$ by 
	\[
	\tilde{\tilde{U}}_n(a)=\argmax_{|u|\leq \log n}\left\{D_n(a,u)+W_{U(a)}(u) \right\}.
	\]	
	We will apply Lemmas 5(i) and 6(i) in \cite{durot2007} to get 
	\begin{equation}
	\label{eqn:step2_tilde_U}
	\mathcal{J}_n\1_{E_n}=\1_{E_n}\int_{J_n}\left|\tilde{\tilde{U}}_n(a)-n^{-1/6}\eta_n(a) \right|^p\frac{|U'(a)|^{1-p}}{h^{uc}(U(a))^p}\,\dd a+o_P(n^{-1/6}).
	\end{equation}
	Since $\lambda'_0$ is bounded above by a strictly negative constant, the function $\Phi$ is uniformly bounded and  $h^{uc}$ is bounded away from zero, it follows by a Taylor expansion that
	\[
	D_n(a,u)\leq -cu^2,
	\]
	for some positive constant $c$, {and all} $a\in J_n$ and $|u|\leq T_n$. 
	Moreover, since 
	\[
	\left|\frac{\partial}{\partial u}(\Lambda_0(t_n)-at_n) \right|=	\left|[t_n-U(a)]\lambda'_0(\zeta)\frac{\partial}{\partial u}(t_n) \right| 
	\]
	for some $|\zeta-U(a)|\leq |t_n-U(a)|$, it follows that 
	\[
	\left|\frac{\partial}{\partial u}D_n(a,u) \right|\leq K|u|,
	\]
	for some positive constant $K$ {and all} $a\in J_n$ and $|u|\leq T_n$. 
	Again we used the fact that $\Phi$ and $\lambda'_0$ are uniformly bounded and  $h^{uc}$ is bounded away from zero. {Hence, the function $D_n$ satisfies the conditions in \cite{durot2007}.}

	It remains to show that $R_n(a,u)$ satisfies the assumptions of Lemma 5(i) in \cite{durot2007}. If we show that 
	\begin{equation}
	\label{eqn:bound_R_n}
	\p\left[\left\{\sup_{|u|\leq T_n}|R_n(a,u)|>x \right\}\cap E_n\right]\lesssim x^{-q} n^{1-q/3},\quad\text{for all }x\in(0,n^{2/3}],
	\end{equation}
	then, as in Lemma 5(i) in \cite{durot2007}, it follows that 
	\[
	\E\left[\1_{E_n}\left|\tilde{U}_n-\tilde{\tilde{U}}_n\right|^r\right]\lesssim \left(n^{-1/6}/\log n\right)^r
	\]
	with $r=2(q-1)/(2q-3)$ (as in Lemma 5(i) in \cite{durot2007})
	and by Lemma 6(i) in \cite{durot2007}, $\tilde{U}_n$ can be replaced by $\tilde{\tilde{U}}_n$.
	Note that for \eqref{eqn:bound_R_n} it is sufficient to consider $x>n^{-1/3+1/q}$.
	{For such $x$'s} we have
	\[
	\begin{split}
	&\p\left[\left\{\sup_{|u|\leq T_n}|R_n(a,u)|>x \right\}\cap E_n\right]\\
	&\leq \p\left[\left\{\sup_{|u|\leq T_n}|R^3_n(a,u)|>x/2 \right\}\cap E_n\right]+\p\left[\left\{\sup_{|u|\leq T_n}|R^4_n(a,u)|>x/2 \right\}\cap E_n\right].
	\end{split}
	\]
	From Lemma \ref{le:Lambda_n-Lambda}, it follows that
	\[
	\p\left[\left\{\sup_{|u|\leq T_n}|R^4_n(a,u)|>x/2 \right\}\cap E_n\right]\lesssim x^{-q}n^{1-q/3}.
	\]
	It remains to consider $R^3_n$. Note that, since $a\in J_n $, it follows that $b<U(a)<1-b$ where we recall that $b=n^{-1/4}$. Hence
	\[
	\begin{split}
	&\p\left[\left\{\sup_{|u|\leq T_n}|R^3_n(a,u)|>x/4 \right\}\cap E_n\right]\\
	&\leq \p\left[{n^{1/3}T_n\left|\Psi_n^s(U(a))-\left[\Phi_n(U(a);\beta_0)-\Phi(U(a);\beta_0)\right]\right|>cx }\right]\\
	&=\p\left[n^{1/3}T_n\left|\int k_b(U(a)-v)\left[\Phi_n(v;\beta_0)-\Phi(v;\beta_0)-\Phi_n(U(a);\beta_0)+\Phi(U(a);\beta_0)\right]\,\dd v\right|>cx \right]\\
	&\leq\p\left[n^{1/3}T_n\sup_{|v-U(a)|\leq b}\left|\Phi_n(v;\beta_0)-\Phi(v;\beta_0)-\Phi_n(U(a);\beta_0)+\Phi(U(a);\beta_0)\right|>cx \right],
	\end{split}
	\]
	{using that the kernel $k$ integrates to one, is non-negative and is supported on $[-1,1]$.}	It follows from Lemma \ref{le:R_3}  that
	\[
	\p\left[\left\{\sup_{|u|\leq T_n}|R^3_n(a,u)|>x/4 \right\}\cap E_n\right]\lesssim x^{-q} n^{1-q/3}. 
	\]
	{Here we need $2\leq q<1/(\gamma+1/24)$ to apply the Lemma. Since $\gamma=1/(3(6q-11))$,  it is possible to choose $q>12$ (but $q<22$).}
	To conclude, from  Lemma 5(i) in \cite{durot2007}, it follows that ${\tilde{U}_n}$ can be replaced by $\tilde{\tilde{U}}_n$, i.e. we have \eqref{eqn:step2_tilde_U}. 
	
	It remains to replace $\tilde{\tilde{U}}_n(a)$ by $\tilde{V}(U(a))$ {where $\tilde V$ is taken from \eqref{eq: Vtilde}}. {Since $h^{uc}(t)=\lambda_{0}(t)\Phi(t;\beta_0)$, it follows from the assumptions on $\lambda_0$ and $\Phi$ that $h^{uc}$ is continuously differentiable. Combining this with a Taylor expansion together with \eqref{eqn:assumption_derivative}}, we have
	\[
	\sup_{|u|\leq \log n}\left|{-}D_n(a,u)-d(U(a))u^2\right|\lesssim n^{-s/3}(\log n)^3.
	\]
	{It follows from \eqref{eq: moments betan} and \eqref{eq: moments Phin} that $\1_{E_n}n^{-1/6}\eta_n$  has bounded moments of any order.
		Since $\tilde{\tilde{U}}_n(a)$ and $\tilde{V}(U(a))$ also} have bounded moments of any order, by Lemmas 5(ii) and 6(i) in \cite{durot2007} it follows {from \eqref{eqn:step2_tilde_U}} that 
	\[
	\begin{split}
	\mathcal{J}_n\1_{E_n}&=\1_{E_n}\int_{J_n}\left|\tilde{V}(U(a))-n^{-1/6}\eta_n(a) \right|^p\frac{|U'(a)|^{1-p}}{h^{uc}(U(a))^p}\,\dd a+o_P(n^{-1/6})\\
	&=\1_{E_n}\int_{b}^{1-b}\left|\tilde{V}(t)-n^{-1/6}\varsigma_n(t) \right|^p\left|\frac{\lambda'_0(t)}{h^{uc}(t)}\right|^{p}\,\dd t+o_P(n^{-1/6}),
	\end{split}
	\]
	where {using \eqref{eqn:lambda0},}
	\begin{equation}
	\label{def:eta_t}
	\varsigma_n(t)=\frac{h^{uc}(t)}{|\lambda'_0(t)|}\left\{\xi_n\lambda_0(t)-n^{1/2}(\hat\beta_n-\beta_0)A'_0(t)+n^{1/2}\Psi_n^s(t)\frac{\lambda_0(t)}{\Phi(t;\beta_0)} \right\}.
	\end{equation}
	Moreover, since $\p(E_n)\to 1$, we can remove $\1_{E_n}$ and get \eqref{eqn:step2}.
	{Note that the interval of integration has been restricted to $[b,1-b]$ because the bounds that we have of $\Psi_n^s$ are valid only on $[b,1-b]$ (see beginning of Step 2) so the change of variable is done on this restricted interval.}
	
	\textbf{Step 3.} Now we prove that ${\varsigma}_n$ can be removed from the integrand in \eqref{eqn:step2}. Let 
	\[
	\mathcal{D}_n=n^{1/6}\left\{\int_{b}^{1-b}\left|\tilde{V}(t) \right|^p\left|\frac{\lambda'_0(t)}{h^{uc}(t)}\right|^{p}\,\dd t-\int_{b}^{1-b}\left|\tilde{V}(t)-n^{-1/6}\varsigma_n(t) \right|^p\left|\frac{\lambda'_0(t)}{h^{uc}(t)}\right|^{p}\,\dd t \right\}.
	\]
	Define 
	\[
	V(t)=\argmax_{u\in\R}\left\{W_t(u)-d(t)u^2\right\}.
	\]
	Then {we can have $\tilde{V}(t)\neq V(t)$ only if $|V(t)|>\log n$} and 
	\begin{equation}
	\label{eqn:tildeV-V}
	\p\left(\tilde{V}(t)\neq V(t)\right)\leq 2\exp\left(-c^2(\log n)^3\right)
	\end{equation}
	(see {Theorem 4 in \cite{durot2002}}). Moreover,
	\begin{equation}
	\label{eqn:V-X0}
	d(t)^{2/3}V(t)=\argmax_{u\in\R}\left\{W_t\left(ud(t)^{-2/3}\right)-u^2d(t)^{-1/3}\right\}\stackrel{d}{=} X(0)
	\end{equation}
	{where $X$ is taken from \eqref{eq: X}.}
	Corollaries 3.4 and 3.3 in \cite{groeneboom1989} show that $X(0)$ has a bounded density function, so from \eqref{eqn:tildeV-V} and \eqref{eqn:V-X0}, so it follows that, for $\gamma\in(0,1/12)$,
	\[
	\p\left(|\tilde{V}(t)|\leq n^{-\gamma}\right)\leq Kn^{-\gamma},
	\]
	for some $K>0$ independent of $t$. Let $C>0$ and define
	\[
	A_n(t)=\left\{|\tilde{V}(t)|>n^{-\gamma}\right\},\qquad B_n=\left\{n^{-1/24}\sup_{t\in[0,1]}|{\varsigma}_n(t)|\leq C\right\}
	\]
	Then $\p(B_n)\to 1$ because $\sup_{t\in[0,1]}|{\varsigma}_n(t)|=O_P(1)$. Moreover, {it follows from \eqref{eq: moments betan} and \eqref{eq: moments Phin} that $\varsigma_n$ has bounded moments of all orders. Since $\tilde{V}$ also} has bounded moments of all the orders, by an expansion of $x\mapsto x^p$ around $\tilde{V}(t)$ we obtain
	\[
	\begin{split}
	\mathcal{D}_n&=\1_{B_n}pn^{1/6}\int_{b}^{1-b}\left\{\left|\tilde{V}(t) \right|-\left|\tilde{V}(t)-n^{-1/6}\varsigma_n(t) \right|\right\}\left|\tilde{V}(t)\right|^{p-1}   \left|\frac{\lambda'_0(t)}{h^{uc}(t)}\right|^{p}\1_{A_n(t)}\,\dd t+o_P(1)\\
	&=p\int_{b}^{1-b} \varsigma_n(t)\tilde{V}(t)\left|\tilde{V}(t) \right|^{p-2}   \left|\frac{\lambda'_0(t)}{h^{uc}(t)}\right|^{p}\,\dd t+o_P(1).
	\end{split}
	\]
	Using the definition of $\eta_n$ in \eqref{def:eta_t} we have
	\begin{equation}
	\label{eqn:integral_eta}
	\begin{split}
	&\int_{b}^{1-b} \varsigma_n(t)\tilde{V}(t)\left|\tilde{V}(t) \right|^{p-2}   \left|\frac{\lambda'_0(t)}{h^{uc}(t)}\right|^{p}\,\dd t\\
	&=\xi_n\int_{0}^1 \frac{\lambda_0(t)h^{uc}(t)}{|\lambda'_0(t)|}\tilde{V}(t)\left|\tilde{V}(t) \right|^{p-2}   \left|\frac{\lambda'_0(t)}{h^{uc}(t)}\right|^{p}\,\dd t\\
	&\quad-n^{1/2}(\hat\beta_n-\beta_0)\int_{0}^1 \frac{A'_0(t)h^{uc}(t)}{|\lambda'_0(t)|}\tilde{V}(t)\left|\tilde{V}(t) \right|^{p-2}   \left|\frac{\lambda'_0(t)}{h^{uc}(t)}\right|^{p}\,\dd t\\
	&\quad+n^{1/2}\int_{b}^{1-b}\Psi_n^s(t) \tilde{V}(t)\left|\tilde{V}(t) \right|^{p-2} \frac{\lambda_0(t)^2}{|\lambda'_0(t)|}  \left|\frac{\lambda'_0(t)}{h^{uc}(t)}\right|^{p}\,\dd t+o_P(1),
	\end{split}
	\end{equation}
	{using \eqref{eqn:lambda0}.}
	Since $\tilde{V}(t)$ has a symmetric distribution, 
	\[
	\E\left[ \left(\int_{0}^1 \frac{\lambda_0(t)h^{uc}(t)}{|\lambda'_0(t)|}\tilde{V}(t)\left|\tilde{V}(t) \right|^{p-2}   \left|\frac{\lambda'_0(t)}{h^{uc}(t)}\right|^{p}\,\dd t\right)^2\right]=\mathrm{Var}\left(\int_{0}^1 \frac{\lambda_0(t)h^{uc}(t)}{|\lambda'_0(t)|}\tilde{V}(t)\left|\tilde{V}(t) \right|^{p-2}   \left|\frac{\lambda'_0(t)}{h^{uc}(t)}\right|^{p}\,\dd t\right)
	\]
	and arguing as in Step 5 in \cite{durot2007}, we obtain
	\[
	\int_{0}^1 \frac{\lambda_0(t)h^{uc}(t)}{|\lambda'_0(t)|}\tilde{V}(t)\left|\tilde{V}(t) \right|^{p-2}   \left|\frac{\lambda'_0(t)}{h^{uc}(t)}\right|^{p}\,\dd t=o_P(1).
	\]
	In the same way also the second term in the right hand side of \eqref{eqn:integral_eta} converge to zero in probability. 
	It remains to deal with the third term. Let $a_i=(b+ib^2)\wedge(1-b)$ for $i=0,\dots,M$ where $M=\lfloor (1-2b)b^{-2}\rfloor+1$.  Then
	\[
	\begin{split}
	&n^{1/2}\int_{b}^{1-b}\Psi_n^s(t) \tilde{V}(t)\left|\tilde{V}(t) \right|^{p-2} \frac{\lambda_0(t)^2}{|\lambda'_0(t)|}  \left|\frac{\lambda'_0(t)}{h^{uc}(t)}\right|^{p}\,\dd t\\
	&=n^{1/2}\sum_{i=0}^{M-1} \int_{a_i}^{a_{i+1}}\Psi_n^s(t) \tilde{V}(t)\left|\tilde{V}(t) \right|^{p-2} \frac{\lambda_0(t)^2}{|\lambda'_0(t)|}  \left|\frac{\lambda'_0(t)}{h^{uc}(t)}\right|^{p}\,\dd t.
	\end{split}
	\]
	For each $i\in\{0,\dots,M-1\}$ we can write
	\[
	\begin{split}
	&\int_{a_i}^{a_{i+1}}\Psi_n^s(t) \tilde{V}(t)\left|\tilde{V}(t) \right|^{p-2} \frac{\lambda_0(t)^2}{|\lambda'_0(t)|}  \left|\frac{\lambda'_0(t)}{h^{uc}(t)}\right|^{p}\,\dd t\\
	&=\Psi_n^s(a_i)\int_{a_i}^{a_{i+1}} \tilde{V}(t)\left|\tilde{V}(t) \right|^{p-2} \frac{\lambda_0(t)^2}{|\lambda'_0(t)|}  \left|\frac{\lambda'_0(t)}{h^{uc}(t)}\right|^{p}\,\dd t\\
	&\qquad +\int_{a_i}^{a_{i+1}}\left[ \Psi_n^s(t)-\Psi_n^s(a_i)\right] \tilde{V}(t)\left|\tilde{V}(t) \right|^{p-2} \frac{\lambda_0(t)^2}{|\lambda'_0(t)|}  \left|\frac{\lambda'_0(t)}{h^{uc}(t)}\right|^{p}\,\dd t.
	\end{split}
	\]
	Moreover, {since $\int k'(u)\, \dd u=0$ and $k$ is supported on $[-1,1]$ and is twice differentiable, it follows from the Taylor expansion that there exist {$\xi^1_{v,t}$ and $\xi^2_{v,t}$} such that} 
	\[
	\begin{split}
	&\sup_{t\in[a_i,a_{i+1}]}\left|\Psi_n^s(t)-\Psi_n^s(a_i)\right|\\
	&= \sup_{t\in[a_i,a_{i+1}]}\frac{1}{b}\left|\int\left(k\left(\frac{t-v}{b}\right)-k\left(\frac{a_i-v}{b}\right)\right)\left[\Phi_n(v;\beta_0)-\Phi(v;\beta_0)\right]\,\dd v\right|\\
	&\leq \left|\Phi_n(a_i;\beta_0)-\Phi(a_i;\beta_0)\right|\sup_{t\in[a_i,a_{i+1}]}\frac{1}{b}\left|\int\left(k\left(\frac{t-v}{b}\right)-k\left(\frac{a_i-v}{b}\right)\right)\,\dd v\right|\\
	&\quad+ \sup_{t\in[a_i,a_{i+1}]}\frac{1}{b}\int\left|k\left(\frac{t-v}{b}\right)-k\left(\frac{a_i-v}{b}\right)\right|\left|(\Phi_n-\Phi)(v;\beta_0)-(\Phi_n-\Phi)(a_i;\beta_0)\right|\,\dd v\\
	&\leq 
	\left|\Phi_n(a_i;\beta_0)-\Phi(a_i;\beta_0)\right|\sup_{t\in[a_i,a_{i+1}]}\left(\frac{t-a_i}{b}\right)^2{\frac{b^{-1}}{2}}\int_{a_i-b}^{{a_{i+1}}+b}\left|k''\left({\xi^1_{v,t}}\right)\right|\,\dd v\\
	&\quad+\sup_{{v}\in{[a_i-b,a_{i+1}+b]}}\left|(\Phi_n-\Phi)(v;\beta_0)-(\Phi_n-\Phi)(a_i;\beta_0)\right|\sup_{t\in[a_i,a_{i+1}]}{\frac{t-a_i}{b^2}}\int_{a_i-b}^{{a_{i+1}}+b}\left|k'\left({\xi^2_{v,t}}\right)\right|\,\dd v.
	\end{split}
	\]
	{Since $|a_i-a_{i+1}|=b^2$, it follows that
		\[
		\begin{split}
		&\sup_{t\in[a_i,a_{i+1}]}\left|\Psi_n^s(t)-\Psi_n^s(a_i)\right|\\
		&\leq 
		\left|\Phi_n(a_i;\beta_0)-\Phi(a_i;\beta_0)\right|O(b^2)+\sup_{v\in{[a_i-b,a_{i+1}+b]}}\left|(\Phi_n-\Phi)(v;\beta_0)-(\Phi_n-\Phi)(a_i;\beta_0)\right|O(b).
		\end{split}
		\]
	}
	As in Lemma \ref{le:R_3} it can be shown {using \eqref{eq: moments Phin}} that {for $q\geq 2$, 
		\[
		\begin{split}
		\E\left[n^{q/2}\sup_{v\in[a_i{-b},a_{i+1}{+b}]}\left|(\Phi_n-\Phi)(v;\beta_0)-(\Phi_n-\Phi)(a_i;\beta_0)\right|^q\right]&\lesssim b^{{q/2}}+n^{-q/2+1}{b}\\
		&\lesssim  b^{{q/2}}.
		\end{split}
		\]
	}
	Moreover the constant in the inequality $\lesssim$ is not dependent on $i$.
	Hence, it follows that {
		\begin{equation}
		\label{eqn:diff_psi}
		\E^{1/2}\left[\sup_{t\in[a_i,a_{i+1}]}\left|\Psi_n^s(t)-\Psi_n^s(a_i)\right|^2\right]=O(b^{3/2} n^{-1/2})
		\end{equation}
		uniformly over $i$}
	{Using the Cauchy-Schwartz inequality we obtain
		\[
		\begin{split}
		&n^{1/2}\E\left[\sum_{i=0}^{M-1} \left|\int_{a_i}^{a_{i+1}}\left[ \Psi_n^s(t)-\Psi_n^s(a_i)\right] \tilde{V}(t)\left|\tilde{V}(t) \right|^{p-2} \frac{\lambda_0(t)^2}{|\lambda'_0(t)|}  \left|\frac{\lambda'_0(t)}{h^{uc}(t)}\right|^{p}\,\dd t\right|\right]\\
		&\leq O(b^{3/2}) \int_{a_0}^{a_{M}}\E^{1/2}\left[\left|\tilde{V}(t) \right|^{2(p-1)}\right] \frac{\lambda_0(t)^2}{|\lambda'_0(t)|}  \left|\frac{\lambda'_0(t)}{h^{uc}(t)}\right|^{p}\,\dd t\\
		&=o(1).
		\end{split}
		\]
	}
	Now we consider
	\[
	\begin{split}
	&n^{1/2}\sum_{i=0}^{M-1}\Psi_n^s(a_i)
	\int_{a_i}^{a_{i+1}} \tilde{V}(t)\left|\tilde{V}(t) \right|^{p-2} \frac{\lambda_0(t)^2}{|\lambda'_0(t)|}  \left|\frac{\lambda'_0(t)}{h^{uc}(t)}\right|^{p}\,\dd t\\
	&=n^{1/2}\Psi_n^s(a_0)
	\int_{a_0}^{a_{M}} \tilde{V}(t)\left|\tilde{V}(t) \right|^{p-2} \frac{\lambda_0(t)^2}{|\lambda'_0(t)|}  \left|\frac{\lambda'_0(t)}{h^{uc}(t)}\right|^{p}\,\dd t\\
	&\quad+n^{1/2}\sum_{i=1}^{M-1}\left[\Psi_n^s(a_i)-\Psi_n^s(a_{i-1})\right]
	\int_{a_i}^{a_{M}} \tilde{V}(t)\left|\tilde{V}(t) \right|^{p-2} \frac{\lambda_0(t)^2}{|\lambda'_0(t)|}  \left|\frac{\lambda'_0(t)}{h^{uc}(t)}\right|^{p}\,\dd t.
	\end{split}
	\]
	Using the symmetry of the distribution of $\tilde{V}(t)$ and and the fact that $\tilde{V}(t)$ is independent of $\tilde{V}(s)$ for $|t-s|>2n^{-1/3}\log n/\inf_t h^{uc}(t)$, it can be shown that 
	{
		\[
		\E^{1/2}\left[\int_{a_i}^{a_j}\tilde{V}(t)\left|\tilde{V}(t) \right|^{p-2} \frac{\lambda_0(t)^2}{|\lambda'_0(t)|}  \left|\frac{\lambda'_0(t)}{h^{uc}(t)}\right|^{p}\,\dd t\right]^2\lesssim(a_j-a_i)^{1/2} n^{-1/6}\left(\log n\right)^{1/2}
		\]
	}
	Moreover, from \eqref{eqn:diff_psi}, we have that, for all $i=1,\dots,M-1$,
	\[
	\E^{1/2}\left[n^{1/2}\left|\Psi_n^s(a_i)-\Psi_n^s(a_{i-1})\right|\right]^2\lesssim b^{3/2}
	\]
	uniformly in $i$.  {Since \eqref{eq: moments Phin} ensures that $\E\left[n^{1/2}\Psi_n^s(a_0)\right]^2\lesssim 1$,} it follows that 
	\[
	\begin{split}
	&\E\left\vert n^{1/2}\sum_{i=0}^{M-1}\Psi_n^s(a_i)
	\int_{a_i}^{a_{i+1}} \tilde{V}(t)\left|\tilde{V}(t) \right|^{p-2} \frac{\lambda_0(t)^2}{|\lambda'_0(t)|}  \left|\frac{\lambda'_0(t)}{h^{uc}(t)}\right|^{p}\,\dd t\right\vert \\
	&\lesssim n^{-1/6}\left( \log n\right)^{1/2}+b^{3/2}\sum_{i=1}^{M-1}(M-i)^{1/2}bn^{-1/6}(\log n)^{1/2}\\
	&\lesssim o(1)+b^{3/2}\sum_{j=1}^{M-1}j^{1/2}bn^{-1/6}(\log n)^{1/2}\\
	&\lesssim o(1)+b^{5/2}n^{-1/6}(\log n)^{1/2}M^{3/2}=o(1),
	\end{split}
	\]
	
	because $b=n^{-1/4}$ and $M\approx b^{-2}$.
	To conclude we have
	\[
	\mathcal{J}_n=\int_{0}^1\left|\tilde{V}(t) \right|^p\left|\frac{\lambda'_0(t)}{h^{uc}(t)}\right|^{p}\,\dd t+o_P(n^{-1/6}).
	\]
	\textbf{Step 4.} Now this is the same as $\mathcal{J}_n$ at the end of Step 3 of the proof of Theorem 2 in \cite{durot2007}. Here $L=H^{uc}$, $g=U$, $d(t)$  differs by a factor $\Phi(t;\beta_0)$ and $D_n(a,u)$ differs by a factor $\Phi(U(a);\beta_0)$. 
	The rest of the proof is exactly as in \cite{durot2007}. Here, using $h^{uc}(t)=\lambda_0(t)\Phi(t;\beta_0)$,  we obtain
	\[
	\begin{split}
	m_p&=\E\left[\left|X(0)\right|^p\right]\int_0^1 d(t)^{-2p/3}\left|\frac{\lambda'_0(t)}{h^{uc}(t)}\right|^{p}\,\text{d} t\\
	&=\E\left[\left|X(0)\right|^p\right]\int_0^1\left|\frac{4\lambda_0(t)\lambda'_0(t)}{\Phi(t;\beta_0)}\right|^{p/3}\,\text{d}t.
	\end{split}
	\]
	and
	\[
	\sigma^2_p=8k_p\int_0^1\left|\frac{4\lambda_0(t)\lambda'_0(t)}{\Phi(t;\beta_0)}\right|^{2(p-1)/3}\frac{\lambda_0(t)}{\Phi(t;\beta_0)}\,\text{d}t.
	\]
\end{proof}
\subsection{Proof of Theorem \ref{theo:CLT2}}
\label{sec:proof_CLT2}
We first need to establish similar results to those in Lemma \ref{le:inv_tail_prob1} and Theorem \ref{theo:bound_expectation_L_P}, where $\lambda
_0$ is replaced by $\lambda_{\hat\theta}$ and $U$ by  $U_{\hat\theta}$ defined as
the inverse of $\lambda_{\hat{\theta}}$, 
\begin{equation}
\label{def:U_theta}
U_{\hat\theta}(a)=\begin{cases}
\epsilon & \text{ if }a>\lambda_{\hat{\theta}}(\epsilon)\\
\lambda_{\hat{\theta}}^{-1}(a) & \text{ if }a\in[\lambda_{\hat{\theta}}(M),\lambda_{\hat{\theta}}(\epsilon)]\\
M & \text{ if } a<\lambda_{\hat{\theta}}(M).
\end{cases}
\end{equation}
From \eqref{eqn:theta_hat}, {using that $f\in\mathcal{C}^\infty(\R^d\times[\epsilon,M])$} it follows that 
\[
\begin{split}
&\sqrt{n}\sup_{t\in[\epsilon,M]}\left(\lambda_{\hat{\theta}}(t)-\lambda_0(t)\right)=O_P(1),\quad\quad\sqrt{n}\sup_{t\in[\epsilon,M]}\left(\lambda'_{\hat{\theta}}(t)-\lambda'_0(t)\right)=O_P(1),\\
&\sqrt{n}\sup_{a\in(0,\infty)}\left(U_{\hat\theta}(a)-U(a)\right)=O_P(1),\quad\quad\sqrt{n}\sup_{a\in(0,\infty)}\left(U'_{\hat\theta}(a)-U'(a)\right)=O_P(1).
\end{split}
\]
Thus, we can assume that on the event $E_n$ defined in \eqref{def:E_n}, we have
\[
\left|\hat\theta_n-\theta\right|\lesssim n^{-1/2+\alpha},
\]
{where $\alpha\in(0,1/6)$ is a fixed number,}
and
\begin{equation}
\label{eqn:par_est-true}
\begin{split}
&\sup_{t\in[\epsilon,M]}\left
|\lambda_{\hat{\theta}}(t)-\lambda_0(t)\right|\lesssim n^{-1/2+\alpha},\qquad \sup_{t\in[\epsilon,M]}\left
|\lambda'_{\hat{\theta}}(t)-\lambda'_0(t)\right|\lesssim n^{-1/2+\alpha}\\
&\sup_{a\in(0,\infty)}\left|U_{\hat\theta}(a)-U(a)\right|\lesssim n^{-1/2+\alpha},\qquad\sup_{a\in(0,\infty)}\left|U'_{\hat\theta}(a)-U'(a)\right|\lesssim n^{-1/2+\alpha}.
\end{split}
\end{equation}
From Lemma \ref{le:inv_tail_prob1} and \eqref{eqn:par_est-true}, {where $\alpha<1/6$,} it follows that
\begin{equation}
\label{eqn:inv2}
\begin{split}
&\p\left[\left\{\left|\hat{U}_n(a)-U_{\hat{\theta}}(a)\right|>x\right\}\cap E_n\right]\\
&\leq \p\left[\left\{\left|\hat{U}_n(a)-U(a)\right|>x/2\right\}\cap E_n\right]+\p\left[\left\{\left|U_{\hat{\theta}}(a)-U(a)\right|>x/2\right\}\cap E_n\right]\\
&\lesssim\exp(-Knx^3).
\end{split}
\end{equation}
Similarly, Theorem \ref{theo:bound_expectation_L_P} and \eqref{eqn:par_est-true} imply that for  $t\in[\epsilon+n^{-1/3},M-n^{-1/3}] $ it holds
\begin{equation}
\label{eqn:bound_expectation_L_p1}
\begin{split}
\E\left[1_{E_n}\left|\hat\lambda_n(t)-\lambda_{\hat{\theta}}(t)\right|^p\right]
&\lesssim \E\left[1_{E_n}\left|\hat\lambda_n(t)-\lambda_0(t)\right|^p\right]+\E\left[1_{E_n}\left|\lambda_0(t)-\lambda_{\hat{\theta}}(t)\right|^p\right]\\
&\lesssim n^{-p/3}+n^{p(-1/2+\alpha)}\lesssim n^{-p/3}
\end{split}
\end{equation}
and, for $t\in[\epsilon+n^{-1},\epsilon+n^{-1/3}]\cup[M-n^{-1/3},M-n^{-1}]$, 
\begin{equation}
\label{eqn:bound_expectation_L_p2}
\E\left[1_{E_n}\left|\hat\lambda_n(t)-\lambda_{\hat{\theta}}(t)\right|^p\right]\lesssim \left[n{((t-\epsilon)\wedge(M-t))}\right]^{-p/2}.
\end{equation}
Afterwards we show that, if we consider the conditional probability $\p_\theta$ given $\hat\theta_n$, 
{Lemma} \ref{le:Lambda_n-Lambda} holds also if $t$ is replaced by a sequence depending on $n$. In this case, we apply this Lemma with $t=U_{\hat{\theta}}(a)$ and using Lemmas \ref{le:R_1_2} and \ref{le:R_2_2} in \cite{DM_supp}, we obtain that the remainder term $r_n$ satisfies
\[
\p_\theta\left[\left\{\sup_{|t_n-U_{\hat{\theta}}(a)|\leq cn^{-1/3}T_n}|{r}_n(U_{\hat{\theta}}(a),t_n)|>x \right\}\cap E_n\right]\leq K x^{-q}n^{1-q}, 
\]
for some $K>0$ and all $x\geq 0$. The rest of the proof continues as for Theorem \ref{theo:CLT}, replacing $U(a)$ by $U_{\hat\theta}(a)$ (see \cite{DM_supp}).

\subsection{Proof of Theorem \ref{theo:CLT*}}
Let $\p_n^*$ be the empirical measure of $(T_1^*,\Delta_1^*,Z_1),\dots,(T_n^*,\Delta_n^*,Z_n)$ and let $P_n^*$ be the conditional probability measure of $(T^*,\Delta^*,Z^*)$ given the data $(T_1,\Delta_1,Z_1),\dots,$ $(T_n,\Delta_n,Z_n)$.
The bootstrap versions of $\Phi_n$ and $\Phi$ defined in~\eqref{eq:def Phin} and \eqref{eq:def Phi} are
\begin{equation}
\label{eq:def Phin*}
\Phi^*_n(x;\beta)=\int \1_{\{t\geq x\}} \mathrm{e}^{\beta'z}\,\mathrm{d}\p^*_n(t,\delta,z),\quad\text{and}\quad \Phi^*(x;\beta)=\int \1_{ \{t\geq x\}}\,\mathrm{e}^{\beta'z}\,\mathrm{d}P_n^*(t,\delta,z).
\end{equation}
To prove Theorem \ref{theo:CLT*}, we mimic the proof of Theorems~\ref{theo:CLT} and~\ref{theo:CLT2}. The bootstrap version of  \eqref{eqn:Breslow} and \eqref{eqn:lambda0}
is given by the following Lemma.
\begin{lemma}
	\label{le:Breslow*}
	The following equivalent representations  hold
	\[	
	\Lambda_{\hat\theta}(x)=\int \delta\1_{\{t\leq x\}}\frac{1}{\Phi^*(t;\tilde\beta_n)}\,\mathrm{d}P_n^*(t,\delta,z),\quad\quad\frac{ \mathrm{d}H^{uc,*}(t)}{\Phi^*(t;\tilde\beta_n)}=\lambda_{\hat\theta}(t)\mathrm{d}t,
	\]
	where $H^{uc,*}$ is the sub-distribution function of the observed event times in the bootstrap sample.
\end{lemma}
Moreover, the bootstrap versions of Lemmas 3 and 4 in \cite{LopuhaaNane2013} hold (see Lemma \ref{le:Phi*} in \cite{DM_supp}) and in particular, we have 
\[
\sqrt{n}\sup_{x\in\R}\left|\Phi_n^*(x;\tilde\beta_n)-\Phi^*(x;\tilde\beta_n)\right|=O_{P^*}(1),
\]
which corresponds to \eqref{eqn:Phi}. Furthermore, with the same argument as in \cite{linear_breslow},  it can be shown that the linear representation of the Breslow estimator holds also for the bootstrap version, i.e. we have
\begin{equation}
\label{eqn:Breslow_linear*}
\begin{split}
\Lambda_n^*(x)-\Lambda_{\hat\theta}(x)&=\int \left(\frac{\delta\1_{\{t\leq x\}}}{\Phi^*(t;\tilde\beta_n)}-e^{\tilde\beta_nz}\int_0^{t\wedge x}\frac{\lambda_{\hat\theta}(v)}{\Phi^*(v;\tilde\beta_n)}\,\dd v \right)\,\dd (\p_n^*-P_n^*)(t,\delta,z)\\
&\quad+(\hat\beta_n^*-\tilde\beta_n)A_0^*(x)+R_n^*(x),
\end{split}
\end{equation}
where
\begin{equation}
\label{def:A_0*}
A_0^*(x)=\int_0^x\frac{D^{(1),*}(u;\tilde\beta_n)}{\Phi^*(u;\tilde\beta_n)}\lambda_{\hat\theta}(u)\,\dd u,\qquad D^{(1),*}(x;\beta)=\frac{\partial \Phi^*(x;\beta)}{\partial \beta}=\int \1_{\{u\geq x\}}ze^{\beta' z}\mathrm P^*_n(u,\delta,z),
\end{equation}
and $\sup_{x\in[0,M]}|R_n^*(x)|=o_P^*(n^{-1+\alpha}).$ Let 
\begin{equation}
\label{def:D_n^*}
D^{(1),*}_n(x;\beta)=\frac{\partial\Phi_n^*(x;\beta)}{\partial\beta}=\int\1_{\{u\geq x\}}ze^{\beta'z}\p_n^*(u,\delta,z),5
\end{equation}
Then we define the bootstrap version of $E_n$ as $E_n^*=E_{n,1}^{*}\cap E_{n,2}^{*}\cap E_{n,3}^{*}\cap E_{n,4}^{*}$ where
\[
E_{n,1}^*=\left\{\sup_{x\in\R}|D_{n}^{(1),*}(x;\beta^*_n)|\leq\xi_1 \text{ for }|\beta^*_n-\tilde\beta_n|\to 0\right\},\qquad E_{n,2}^*
=
\left\{n^{1/2-\alpha}|\hat{\beta}_n^*-\tilde\beta_n|<\xi_2\right\}, 
\]
\[
E_{n,3}^*
=
\left\{n^{1/2-\alpha}\sup_{x\in \mathbb{R}}\left|\Phi_n^*(x;\tilde\beta_n)-\Phi^*(x;\tilde\beta_n)\right|\leq \xi_3\right\},\qquad E_{n,4}^*=\left\{\sup_{x\in[0,1]}|R_{n}^*(x)|\leq \xi_4n^{-1+\alpha} \right\}.
\]
This event is such that  $\1_{E_n^*}=1+o_p^*(1)$,
meaning that for all $\epsilon>0$, 
\[
\limsup_{n\to\infty}
P_n^*(|\1_{E_n^*}-1|>\epsilon)=0,
\qquad
\p-\text{almost surely}.
\]
Moreover, in $E_n^*$, we have that for sufficiently large $n$, 
\[
\sup_{x\in\R}|\Phi_n^*(x;\hat{\beta}_n^*)-\Phi^*(x;\tilde\beta_n)|\lesssim n^{-1/2+\alpha},\qquad\inf_{x\in[0,M]}\Phi_n^*(x;\hat{\beta}_n^*)\geq c>0
\]
\begin{equation}
\label{eqn:bound_inv_Phi_n*}
\sup_{s\in[0,M]}\left|\frac{1}{\Phi_n^*(s;\hat{\beta}_n^*)}-\frac{1}{\Phi^*(s;\tilde\beta_n)} \right|\lesssim n^{-1/2+\alpha}.
\end{equation}
The main problems one faces when following the proof of Theorem \ref{theo:CLT} for the bootstrap version are the dependence on $n$ of the functions related to the true distribution of the data and their non-differentiability. The dependence on $n$ is actually not a problem because these functions are uniformly bounded since they are consistent estimators of some bounded smooth function. On the other hand, to overcome non-differentiability we consider approximations of these functions by smooth versions of them. For example, let $H^*$ and $H^{uc,*}$ be the distribution of $T^*$ and the sub-distribution function of the observed event times in the bootstrap sample. We have
\[
H^*(t)=P^*_n\left(T\leq t\right)=\frac{1}{n}\sum_{i=1}^n P^*_n\left(T\leq t\,|\,Z=Z_i\right)=1-\left[1-\hat{G}_n(t)\right]\frac{1}{n}\sum_{i=1}^n\left[1-F_{\hat\theta}(t|Z_i)\right]
\]
and
\[
\begin{split}
H^{uc,*}(t)&=P^*_n\left(T\leq t,\,\Delta=1\right)=\frac{1}{n}\sum_{i=1}^n P^*_n\left(T\leq t,\,\Delta=1\,|\,Z=Z_i\right)\\
&=\frac{1}{n}\sum_{i=1}^n\int_0^tf_{\hat\theta}(u|Z_i)\left[1-\hat{G}_n(u)\right]\,\dd u.
\end{split}
\]
In particular, $H^*$ is not continuous and $H^{uc,*}$ is not continuously differentiable. However, we can approximate them by continuously differentiable functions.
Since, for the Kaplan-Meier estimator, it holds
\[
\sqrt{n}\sup_{t\in[0,M]}|\hat G_n(t)-G(t)|=O_P(1),
\]
we obtain
\[
\sqrt{n}\sup_{t\in[0,M]}|H^*(t)-\tilde{H}^*(t)|=O_P(1)
\]
where
\begin{equation}
\label{def:tilde_H*}
\tilde{H}^*(t)=1-\left[1-G(t)\right]\frac{1}{n}\sum_{i=1}^n\left[1-F_{\hat\theta}(t|Z_i)\right].
\end{equation}
Hence, with probability converging to one, given the data we have
\begin{equation}
\label{eqn:approx_H*}
\sup_{t\in[0,M]}|H^*(t)-\tilde{H}^*(t)|\lesssim n^{-1/2+\alpha}.
\end{equation}
Moreover, with
\[
\tilde{H}^{uc,*}(t)=\frac{1}{n}\sum_{i=1}^n\int_0^tf_{\hat\theta}(u|Z_i)\left[1-G(u)\right]\,\dd u,
\]
it holds
\[
\begin{split}
\sup_{t\in[0,M]}|\tilde{H}^{uc,*}(t)-\tilde{H}^{uc,*}(t)|\leq \sup_{t\in[0,M]} \frac{1}{n}\sum_{i=1}^n\int_0^tf_{\hat\theta}(u|Z_i)\left|\hat{G}_n(u)-G(u)\right|\,\dd u=O_P(n^{-1/2}).
\end{split}
\]
Hence, with probability converging to one, given the data we have
\begin{equation}
\label{eqn:approx_Huc*}
\sup_{t\in[0,M]}|H^{uc,*}(t)-\tilde{H}^{uc,*}(t)|\lesssim n^{-1/2+\alpha}.
\end{equation}

\section*{Supplementary Material}
Supplement to ''On the $L_p$-error of the Grenander-type estimator in the Cox model''.
\begin{itemize}
	\item Supplement~\ref{sec:clt_supp}: Auxiliary results for Section~\ref{sec:main}.
	\item Supplement~\ref{sec:parametric}: CLT under a parametric baseline distribution.
	\item Supplement~\ref{sec:supp_bootstrap}: CLT for the bootstrap version.
	\item {Supplement~\ref{sec:Cox}: Estimation in the Cox model.}
\end{itemize}
\section*{Acknowledgement}
This research has been conducted as part of the project Labex MME-DII (ANR11-LBX-0023-01)

\bibliography{shapeconstrained-estimation}	

\newpage
\setcounter{page}{1}
\setcounter{equation}{0}
\renewcommand{\theequation}{S\arabic{equation}}
\pagestyle{myheadings}
\markboth{C\'ecile Durot and Eni Musta}{On the $L_p$-error of the Grenander estimator in the Cox model}
%
%
\thispagestyle{empty}
\centerline{\Large\bf On the $L_p$-error of the Grenander-type estimator}
\smallskip
\centerline{\Large\bf in the Cox model}
\bigskip
\centerline{\Large Supplementary Material}
\bigskip
\centerline{ C\'ecile Durot$^\dag$ and Eni Musta$ ^\S$}
\medskip
\centerline{\it Universit\'e Paris Nanterre$^\dag$, Delft University of Technology$^\S$} 
\bigskip

\appendix

\section{Auxiliary results for Section~\ref{sec:main}}
\label{sec:clt_supp}
In what follows, $\p_T$ and $\E_T$ denote the conditional probability and conditional expectation given $T_1,\dots,T_n$.  
Moreover, we use the same notations $\lesssim$, $K,K_1$, \dots and $c,c_1,\dots$ as in Section \ref{sec:proofs}. {We begin with a general concentration inequality for binomial variables.
	\begin{lemma}
		\label{le:binomial}
		Let $X$ be a binomial variable with parameter $n$ and probability of success $p$. Then for all $r\geq 0$ we have
		$$\E\left[\exp(r|X-\E(X)|)\right]\leq 2\exp(np\phi(r))$$
		where $\phi(r)=\exp(r)-r-1$. Moreover, 
		\begin{equation}\notag
		\p(|X-\E(X)|\geq r)\leq 2\exp\left(-\frac{r^2}{2(np+r)}\right).
		\end{equation}
	\end{lemma}
	\begin{proof}
		Since $X$ has a Binomial distribution we have
		\[\begin{split}
		\E\left[e^{r|X-\E[X]|}\right]&\leq \E\left[e^{r(X-\E[X])}\right]+\E\left[e^{-r(X-\E[X])}\right]\\
		&\leq\exp(np\phi(r))+\exp(np\phi(-r)).
		\end{split}
		\]
		The function $x\mapsto\phi(x)-\phi(-x)$ is convex on $[0,\infty)$ with a vanishing derivative at $x=0$. As a consequence, it is a non-negative function and therefore, $\phi(r)\geq\phi(-r)$ for all $r\geq 0$. The first inequality of the lemma follows by combining this with the previous display. Now, it follows from Markov inequality and the first inequality of the lemma that for all $c\geq 0$, 
		\begin{eqnarray*}
			\p(|X-\E(X)|\geq r)&\leq& \exp(-cr)\E\left[\exp(c|X-\E(X)|)\right]\\
			&\leq& 2\exp(-cr)\exp(np\phi(c)).
		\end{eqnarray*}
		Choosing $c$ in such a way that the previous bound is as small as possible, i.e. $c=\log(1+r/(np))$, yields
		\begin{eqnarray*}
			\p(|X-\E(X)|\geq r)&\leq& 2\exp(-nph(r/(np)))
		\end{eqnarray*}
		where for all $u\geq -1$, $h(u)=(1+u)\log(1+u)-u$. The result now follows from the inequality $h(u)\geq u^2/(2+2u)$, that holds for all $u\geq 0$ thanks to the Taylor expansion, since $h(0)=h'(0)=0$.
\end{proof} }

\begin{lemma}
	\label{le:S1}
	Assume that the follow-up time $T$ has uniformly continuous distribution with bounded density $h$. Let $C=\sup_{u\in[0,1]}h(u)$. Let $t\mapsto r(t)$ be a positive bounded function on $[0,1]$ with $R=\sup_{u\in[0,1]}r(t)$  and consider
	\[
	S_n^1(t)=\frac{1}{n}\sum_{i=1}^n \left(\Delta_i-\E_T[\Delta_i]\right)\1_{\{T_i\leq t\}}r(T_i).
	\] 
	There exists a positive constant $K$ such that for each $t\in[0,1]$, $t_n\in(0,1)$, and $x>0$ it holds
	\[
	\p\left(\sup_{\substack{|s-t|<t_n\\s\in[0,1]}}\left|S_n^1(t)-S_n^1(s) \right|\geq x\right)\leq {4}\exp\left(-K\tilde\theta nx\right)
	\]
	where 
	\begin{equation}
	\label{def:theta_tilde}
	\tilde\theta=\min\left\{1,\frac{4x}{3t_nR^2Ce^{R^2/2}}\right\}.
	\end{equation}
\end{lemma}
\begin{proof}
	We {consider the case $s\in[t,t+t_n]$. The case $s\in[t-t_n,t]$ can be handled similarly.} By definition we have
	\[
	\begin{split}
	\p\left(\sup_{\substack{s\in[t,t+t_n]\\s\leq 1}}\left|S_n^1(s)-S_n^1(t) \right|\geq x\right) 
	&=\p\left(\sup_{\substack{s\in[t,t+t_n]\\s\leq 1}}\left|\sum_{i=1}^n \left(\Delta_i-\E_T[\Delta_i]\right)\1_{\{t<T_i\leq s\}}r(T_i) \right|\geq nx\right)\\
	&=\E\left[\p_T\left(\sup_{\substack{s\in[t,t+t_n]\\s\leq 1}}\left|\sum_{i=1}^n \left(\Delta_i-\E_T[\Delta_i]\right)\1_{\{t<T_i\leq s\}}r(T_i) \right|\geq nx\right)\right].\\
	\end{split}
	\]
	Take $\theta>0$. Then we can write 
	\[
	\begin{split}
	&\p\left(\sup_{\substack{s\in[t,t+t_n]\\s\leq 1}}\left|S_n^1(s)-S_n^1(t) \right|\geq x\right) \\
	&=\E\left[\p_T\left(\exp\left\{\theta\sup_{s\in[t,t+t_n]}\left|\sum_{i=1}^n \left(\Delta_i-\E_T[\Delta_i]\right)\1_{\{t<T_i\leq s\}}r(T_i) \right|\right\}\geq \exp(\theta nx) \right)\right].
	\end{split}
	\]
	Let 
	\[
	Y_i(s)=\left(\Delta_i-\E_T[\Delta_i]\right)\1_{\{t<T_i\leq s\}}r(T_i).
	\]
	Since, conditionally on {$T_1,\dots,T_n$}, $S_n^1$ is a martingale we can apply Doob inequality and obtain
	\[
	\begin{split}
	&\p\left(\sup_{\substack{s\in[t,t+t_n]\\s\leq 1}}\left|S_n^1(s)-S_n^1(t) \right|\geq x\right) 
	\leq\E\left[\sup_{\substack{s\in[t,t+t_n]\\s\leq 1}}\E_T\left[\exp\left\{\theta\left|\sum_{i=1}^n Y_i(s) \right|\right\} \right]\right]\exp(-\theta nx)  \\
	&\quad\leq \left(\E\left[\sup_{\substack{s\in[t,t+t_n]\\s\leq 1}}\E_T\left[\exp\left\{\theta\sum_{i=1}^n Y_i(s) \right\} \right]\right]+\E\left[\sup_{\substack{s\in[t,t+t_n]\\s\leq 1}}\E_T\left[\exp\left\{-\theta\sum_{i=1}^n Y_i(s) \right\} \right]\right]\right)\exp(-\theta nx). 
	\end{split}
	\]
	Since, conditionally on {$T_1,\dots,T_n$}, the $Y_i$'s are independent we have
	\[
	\E_T\left[\exp\left\{\theta\sum_{i=1}^n Y_i(s) \right\} \right]=\prod_{i=1}^{n}\E_T\left[\exp\left\{\theta Y_i(s) \right\} \right].
	\]
	Moreover, the $Y_i$'s are centered random variables with
	\[
	Y_i(s)\leq b_i(s)=\sup_{u\in [0,1]}r(u)\1_{\{t<T_i\leq s\}},\qquad Y_i(s)\geq -b_i(s).
	\]
	It follows by Hoeffding's Lemma that
	\[
	\E_T\left[e^{\theta Y_i(s)} \right]\leq \exp\left(\frac{\theta^2(2b_i(s))^2}{8}\right)\leq \exp\left(\frac{\theta^2\left(\sup_{u\in [0,1]}r(u)\right)^2\1_{\{t<T_i\leq s\}}}{2}\right).
	\]
	In the same way (considering $-Y_i(s)$ instead of $Y_i(s)$) we get
	\[
	\E_T\left[e^{-\theta Y_i(s)} \right]\leq \exp\left(\frac{\theta^2\left(\sup_{u\in [0,1]}r(u)\right)^2\1_{\{t<T_i\leq s\}}}{2}\right).
	\]
	Hence,
	\[
	\begin{split}
	\p\left(\sup_{\substack{s\in[t,t+t_n]\\s\leq 1}}\left|S_n^1(s)-S_n^1(t) \right|\geq x\right) 
	&\leq 2\E\left[\prod_{i=1}^n\exp\left(\frac{\theta^2\left(\sup_{u\in [0,1]}r(u)\right)^2\1_{\{t<T_i\leq t+t_n\}}}{2}\right) \right]\exp(-\theta nx)\\
	&=2\left(\E\left[\exp\left(\frac{\theta^2\left(\sup_{u\in [0,1]}r(u)\right)^2\1_{\{t<T\leq t+t_n\}}}{2}\right) \right]\right)^n\exp(-\theta nx)
	\end{split}
	\]
	{where $T$ is a random variable with the same distribution as the $T_i$'s.}
	Since, $\1_{\{t<T\leq t+t_n\}}$ is a Bernoulli random variable with success probability $\gamma<Ct_n$, we have, for {$\theta\leq 1$},
	\[
	\E\left[\exp\left(\frac{\theta^2\left(\sup_{u\in [0,1]}r(u)\right)^2\1_{\{t<T\leq t+t_n\}}}{2}\right) \right]\leq \exp\left(\gamma\left(e^{R^2\theta^2/2}-1\right)\right)\leq \exp\left(\frac12Ct_nR^2\theta^2e^{R^2/2}\right)
	\]
	and
	\[
	\begin{split}
	\p\left(\sup_{\substack{s\in[t,t+t_n]\\s\leq 1}}\left|S_n^1(s)-S_n^1(t) \right|\geq x\right) &\leq 2\exp\left(\frac12n\theta^2Ct_nR^2e^{R^2/2}\right)\exp(-\theta nx)\\
	&=2\exp\left(-n\theta\left\{x-\frac12\theta C t_nR^2e^{R^2/2}\right\}\right).
	\end{split}
	\]
	If we choose $\theta=\tilde\theta$ {as defined in \eqref{def:theta_tilde}}
	we get {$x-\frac12\theta Ct_nR^2e^{R^2/2}\geq x/3$ and therefore,}
	\[
	\p\left(\sup_{\substack{s\in[t,t+t_n]\\s\leq 1}}\left|S_n^1(s)-S_n^1(t) \right|\geq x\right)\leq 2\exp\left(-\frac{1}{3}\tilde\theta nx\right).
	\]
	{The same bound can be obtained for the supremum over $[t-t_n,t]\cap[0,1]$; and combining the two bounds completes the proof of Lemma \ref{le:S1} with $K=1/3$.}
\end{proof}

\begin{lemma}
	\label{le:S2}
	Assume that the follow-up time $T$ has uniformly continuous distribution with bounded density $h$ such that $\inf_{u\in[0,1]}h(u)>0$. Let $t\mapsto m(t)$ be a positive continuously differentiable function on $[0,1]$ and consider
	\[
	S_n^2(t)=\frac{1}{n}\sum_{i=1}^n \left(m(T_i)\1_{\{T_i\leq t\}}-\E\left[m(T_i)\1_{\{T_i\leq t\}}\right]\right).
	\] 
	There exist two positive constants $K_1,\,K_2$ such that for each $t\in[0,1]$, $t_n\in(0,1)$, {and $x>0$,} it holds
	\[
	\p\left(\sup_{\substack{|s-t|<t_n\\ s\in[0,1]}}\left|S_n^2(t)-S_n^2(s) \right|\geq x\right)\leq K_1\exp\left(-K_2nx\min\left\{1,xt_n^{-1}\right\}\right).
	\]
\end{lemma}
\begin{proof}
	We consider first the case $s\in[t,t+t_n]\cap[0,1]$.	We have
	\[
	S^2_n(s)-S^2_n(t)=\frac{1}{n}\sum_{i=1}^n m(T_i)\1_{\{t<T_i\leq s\}}-\E\left[m(T)\1_{\{{t<T\leq s}\}}\right]
	\]
	{where $T$ is a random variable with the same distribution as the $T_i$'s.}
	By a change of variable, we can write
	\[
	\E\left[{m(T)\1_{\{t<T\leq s\}}}\right]=\int_t^s m(u)\,\dd H(u)=\int_{H(t)}^{H(s)} m\left(H^{-1}(v)\right)\,\dd v
	\]
	{where $H$ denotes the distribution function of $T$.}
	Moreover,
	\[
	\frac{1}{n}\sum_{i=1}^n m(T_i)\1_{\{t<T_i\leq s\}}=\int_0^1 m\left(H^{-1}_n(v)\right)\1_{\{t<H^{-1}_n(v)\leq s \}}\,\dd v
	\]
	{where $H_n^{-1}$ is the empirical quantile function corresponding to $T_1,\dots,T_n$.}
	Hence, we get 
	\begin{equation}
	\label{eqn:1}
	\begin{split}
	\left|S^2_n(s)-S^2_n(t)\right|&=\left|\int_0^1 m\left(H^{-1}_n(v)\right)\1_{\{t<H^{-1}_n(v)\leq s \}}\,\dd v-\int_{H(t)}^{H(s)} m\left(H^{-1}(v)\right)\,\dd v\right|\\
	&\leq \int_{H_n(t)}^{H_n(s)}\left| m\left(H^{-1}_n(v) \right)-m\left(H^{-1}(v)\right)\right| \,\dd v\\
	&\quad+\left|\int_0^1 m\left(H^{-1}(v) \right)\left[\1_{\{H_n(t)<v\leq H_n(s) \}}-\1_{\{H(t)<v\leq H(s) \}}\right]\,\dd v\right|
	\end{split}
	\end{equation}
	and as a result 
	\begin{equation}
	\label{eqn:main}
	\begin{split}
	&\p\left(\sup_{\substack{s\in[t,t+t_n]\\ s\leq 1}}\left|S_n^2(t)-S_n^2(s) \right|\geq x\right)\\
	&\leq \p\left(\sup_{\substack{s\in[t,t+t_n]\\ s\leq 1}}\int_{H_n(t)}^{H_n(s)}\left| m\left(H^{-1}_n(v) \right)-m\left(H^{-1}(v)\right)\right| \,\dd v \geq \frac{1}{2}x\right)\\
	&\quad +\p\left(\sup_{\substack{s\in[t,t+t_n]\\ s\leq 1}}\left|\int_0^1 m\left(H^{-1}(v) \right)\left[\1_{\{H_n(t)<v\leq H_n(s) \}}-\1_{\{H(t)<v\leq H(s) \}}\right]\,\dd v\right|\geq \frac{1}{2}x\right).
	\end{split}	
	\end{equation}
	We consider first the first probability on the right hand side. 
	Since $m$ and $H$ are differentiable with bounded derivative, we have
	\begin{equation}
	\label{eqn:2}
	\begin{split}
	&\int_{H_n(t)}^{H_n(s)}\left| m\left(H^{-1}_n(v) \right)-m\left(H^{-1}(v)\right)\right| \,\dd v\\
	&\lesssim [H_n(s)-H_n(t)]\sup_{u\in[0,1]}\left|H_n^{-1}(u)-H^{-1}(u)\right|\\
	&\lesssim (s-t)\sup_{u\in[0,1]}\left|H_n^{-1}(u)-H^{-1}(u)\right|\\
	&\quad+\left|(H_n-H)(s)-(H_n-H)(t)\right|\sup_{u\in[0,1]}\left|H_n^{-1}(u)-H^{-1}(u)\right|\\	
	&\lesssim (s-t)\sup_{u\in[0,1]}\left|H_n^{-1}(u)-H^{-1}(u)\right|\\
	&\quad+\sup_{u\in[0,1]}\left|H_n(u)-H(u)\right|	\sup_{u\in[0,1]}\left|H_n^{-1}(u)-H^{-1}(u)\right|.
	\end{split}
	\end{equation}
	As a result, for some constant $c>0$, we obtain 
	\[
	\begin{split}
	&\p\left(\sup_{\substack{s\in[t,t+t_n]\\ s\leq 1}}\int_{H_n(t)}^{H_n(s)}\left| m\left(H^{-1}_n(v) \right)-m\left(H^{-1}(v)\right)\right| \,\dd v \geq \frac{1}{2}x\right)\\
	&\leq \p\left(\sup_{u\in[0,1]}|H_n^{-1}(u)-H^{-1}(u)| \geq cxt_n^{-1}\right)\\
	&\quad+\p\left(\sup_{u\in[0,1]}\left|H_n(u)-H(u)\right|	\sup_{u\in[0,1]}\left|H_n^{-1}(u)-H^{-1}(u)\right|\geq cx\right)
	\end{split}
	\]
	By Lemma A1 in \cite{banerjee2016divide} and the 
	DKW inequality, it follows that 
	\begin{equation}
	\label{eqn:prob1}
	\begin{split}
	&\p\left(\sup_{\substack{s\in[t,t+t_n]\\ s\leq 1}}\int_{H_n(t)}^{H_n(s)}\left| m\left(H^{-1}_n(v) \right)-m\left(H^{-1}(v)\right)\right| \,\dd v \geq \frac{1}{2}x\right)\\
	&\lesssim \exp\left(-c_1nx^2t_n^{-2} \right)+\p\left(\sup_{u\in[0,1]}\left|H_n^{-1}(u)-H^{-1}(u)\right|\geq cx^{1/2}\right)\\
	&\quad +\p\left(\sup_{u\in[0,1]}|H_n(u)-H(u)|\geq x^{1/2}\right)\\
	&\lesssim \exp\left(-c_1nx^2t_n^{-2} \right)+\exp\left(-c_2nx\right).
	\end{split}
	\end{equation}
	Now we deal with the second probability on the right hand side of \eqref{eqn:main}.
	We have
	\begin{equation}
	\label{eqn:4}
	\begin{split}
	&\left|\int_0^1 m\left(H^{-1}(v) \right)\left[\1_{\{H_n(t)<v\leq H_n(s) \}}-\1_{\{H(t)<v\leq H(s) \}}\right]\,\dd v\right|\\
	&\leq \left|\int_0^1 \left[m\left(H^{-1}(v) \right)-m\left(H^{-1}(H(t))\right)\right]\left[\1_{\{H_n(t)<v\leq H_n(s) \}}-\1_{\{H(t)<v\leq H(s) \}}\right]\,\dd v\right|\\
	&\quad +m(t)\left|(H_n-H)(s)-(H_n-H)(t)\right|\\
	\end{split}
	\end{equation}
	and {for some  $\zeta_v$ lying between $v$ and $H(t)$,}
	\begin{equation*}
	\begin{split}
	&\left|\int_0^1 \left[m\left(H^{-1}(v) \right)-m\left(H^{-1}(H(t))\right)\right]\left[\1_{\{H_n(t)<v\leq H_n(s) \}}-\1_{\{H(t)<v\leq H(s) \}}\right]\,\dd v\right|\\
	&=\left|\int_0^1 (v-H(t))(m\circ H^{-1})'(\zeta_v) \left[\1_{\{H_n(t)<v\leq H_n(s) \}}-\1_{\{H(t)<v\leq H(s) \}}\right]\,\dd v\right|\\
	&=\left|\int_{H_n(t)}^{H(t)} (v-H(t))(m\circ H^{-1})'(\zeta_v)\,\dd v-\int_{H_n(s)}^{H(s)} (v-H(t))(m\circ H^{-1})'(\zeta_v)\,\dd v\right|.\\
	\end{split}
	\end{equation*}
	Moreover,
	\begin{equation*}
	\left|\int_{H_n(t)}^{H(t)} (v-H(t))(m\circ H^{-1})'(\zeta_v)\,\dd v\right|\leq \frac{1}{2}\sup_{u\in[0,1]}|(m\circ H^{-1})'(u)||H_n(t)-H(t)|^2    
	\end{equation*}
	and 
	\begin{equation*}
	\begin{split}
	&\left|\int_{H_n(s)}^{H(s)} (v-H(t))(m\circ H^{-1})'(\zeta_v)\,\dd v\right|\\
	&\leq \frac{1}{2}\sup_{u\in[0,1]}|(m\circ H^{-1})'(u)|
	|H_n(s)-H(s)|\left\{|H(s)-H(t)|+|H_n(s)-H(t)| \right\} 
	\end{split}
	\end{equation*}
	Hence, we obtain 
	\begin{equation}
	\label{eqn:5}
	\begin{split}
	&\left|\int_0^1 \left[m\left(H^{-1}(v) \right)-m\left(H^{-1}(H(t))\right)\right]\left[\1_{\{H_n(t)<v\leq H_n(s) \}}-\1_{\{H(t)<v\leq H(s) \}}\right]\,\dd v\right|\\
	&\lesssim  \frac{\sup_{u\in[0,1]}|m'(u)|}{\inf_{u\in[0,1]}h(u)}\left\{\sup_{u\in[0,1]}|H_n(u)-H(u)|^2+\sup_{u\in[0,1]}h(u)(s-t)\sup_{u\in[0,1]}|H_n(u)-H(u)| \right\}.
	\end{split}
	\end{equation}
	From \eqref{eqn:4} and \eqref{eqn:5} it follows that
	\begin{equation}
	\label{eqn:6}
	\begin{split}
	&\p\left(\sup_{\substack{s\in[t,t+t_n]\\ s\leq 1}}\left|\int_0^1 m\left(H^{-1}(v) \right)\left[\1_{\{H_n(t)<v\leq H_n(s) \}}-\1_{\{H(t)<v\leq H(s) \}}\right]\,\dd v\right|\geq \frac{x}{2}\right)\\
	&\leq \p\left(\sup_{\substack{s\in[t,t+t_n]\\ s\leq 1}}\left|(H_n-H)(s)-(H_n-H)(t)\right|\geq cx\right)\\
	&\quad+ \p\left(\sup_{\substack{s\in[t,t+t_n]\\ s\leq 1}}\left\{\sup_{u\in[0,1]}|H_n(u)-H(u)|^2+\sup_{u\in[0,1]}h(u)(s-t)\sup_{u\in[0,1]}|H_n(u)-H(u)| \right\}\geq cx\right).
	\end{split}
	\end{equation}
	
	{
		Next we deal with the first probability in the right hand side of \eqref{eqn:6}. Assume first that $t_n\leq c_0x$ where $c_0=c/(2\sup_uh(u))$. 
		By monotonicity of both $H_n$ and $H$, for $s\in[t,t+t_n]$ we have
		\begin{eqnarray*}
			(H_n-H)(s)-(H_n-H)(t)&\leq& H_n(t+t_n)-H(t)+(H_n-H)(t)\\
			&\leq& (H_n-H)(t+t_n)-(H_n-H)(t)+t_n\sup_uh(u).
		\end{eqnarray*}
		Hence,
		$$|(H_n-H)(s)-(H_n-H)(t)|\leq| (H_n-H)(t+t_n)-(H_n-H)(t)|+t_n\sup_uh(u).$$
		Therefore, since $t_n\leq c_0x$ where $c_0=c/(2\sup_uh(u))$ we obtain
		\begin{eqnarray*}
			&&\p\left(\sup_{\substack{s\in[t,t+t_n]\\ s\leq 1}}\left|(H_n-H)(s)-(H_n-H)(t)\right|\geq cx\right)\\
			&&\leq \p\left(\left|(H_n-H)(t+t_n)-(H_n-H)(t)\right|\geq cx/2\right)\\
			&&\leq 2\exp\left(-\frac{nc^2x^2}{8p+4cx}\right),
		\end{eqnarray*}
		using Lemma \ref{le:binomial} for the last inequality with $p=H(t_n)-H(t)$. For $t_n\leq c_0x$ we have $p\leq cx/2$ and therefore,
		\begin{eqnarray}\label{eq:b1}
		\p\left(\sup_{\substack{s\in[t,t+t_n]\\ s\leq 1}}\left|(H_n-H)(s)-(H_n-H)(t)\right|\geq cx\right)\leq 2\exp\left(-cnx/8\right).
		\end{eqnarray}
		Now, if we assume that $t_n>c_0 x$ we have
		\begin{eqnarray}\notag
		&&\p\left(\sup_{\substack{s\in[t,t+t_n]\\ s\leq 1}}\left|(H_n-H)(s)-(H_n-H)(t)\right|\geq cx\right)\\ \notag
		&&\leq \p\left(\sup_{\substack{s\in[t,t+c_0x]\\ s\leq 1}}\left|(H_n-H)(s)-(H_n-H)(t)\right|\geq cx\right)\\ \label{eq:c0x}
		&&+ \p\left(\sup_{\substack{s\in[t+c_0x,t+t_n]\\ s\leq 1}}\left|(H_n-H)(s)-(H_n-H)(t)\right|\geq cx\right).
		\end{eqnarray}
		From what precedes, the first probability on the right hand side is bounded above by $2\exp(-ncx/8)$ and it remains to deal with the second probability.  For $t\in[0,1]$, define 
		\[
		\mathcal{M}_n(s)=\frac{H_n(s)-H_n(t)}{H(s)-H(t)},\qquad s\in[0,1]
		\]
		We can write
		\[
		\begin{split}
		&\p\left(\sup_{\substack{s\in[t+c_0x,t+t_n]\\ s\leq 1}}\left|(H_n-H)(s)-(H_n-H)(t)\right|\geq cx\right)\\
		&\leq\sum_{k:\ t_n>2^{k-1}c_0x}
		\p\left(\sup_{\substack{s\in[t+t_n/2^{k+1},t+t_n/2^k]\\ s\leq 1}}\left|(H_n-H)(s)-(H_n-H)(t)\right|\geq cx\right)\\
		&\leq\sum_{k:\ t_n>2^{k-1}c_0x}
		\p\left(\sup_{\substack{s\in[t+t_n/2^{k+1},t+t_n/2^k]\\ s\leq 1}}\left|\frac{(H_n-H)(s)-(H_n-H)(t)}{H(s)-H(t)}\right|\geq \frac{cx}{|H(t+t_n/2^k)-H(t)|}\right)\\
		&\leq\sum_{k:\ t_n>2^{k-1}c_0x}
		\p\left(\sup_{\substack{s\in[t+t_n/2^{k+1},t+t_n/2^k]\\ s\leq 1}}\left|\mathcal{M}_n(s)-1\right|\geq 2^k\tilde{c}xt_n^{-1}\right).
		\end{split}
		\] 
		By Lemma 2.2 in \cite{groeneboom-hooghiemstra-lopuhaa1999}, $\mathcal{M}_n(s)$, $s\in[t,1]$, is a reverse time martingale, conditionally on $H_n(t)$. Hence, each $r>0$, $\exp(r|\mathcal{M}_n(s)-1|)$ is a reverse time submartingale, conditionally on $H_n(t)$. It then follows from Doob's inequality that for all $r_k>0$,
		\[
		\begin{split}
		&\p\left(\sup_{\substack{s\in[t+c_0x,t+t_n]\\ s\leq 1}}\left|(H_n-H)(s)-(H_n-H)(t)\right|\geq cx\right)\\
		&\leq\sum_{k:\ t_n>2^{k-1}c_0x}
		\E\left[\p\left(\sup_{\substack{s\in[t+t_n/2^{k+1},t+t_n/2^k]\\ s\leq 1}}\exp\left(r_k\left|\mathcal{M}_n(s)-1\right|\right)\geq \exp\left(r_k2^k\tilde{c}xt_n^{-1}\right)\,\bigg| H_n(t)\right)\right]\\
		&\leq\sum_{k:\ t_n>2^{k-1}c_0x}
		\exp\left(-r_k2^k\tilde{c}xt_n^{-1}\right)\E\left[\exp\left(r_k\left|\mathcal{M}_n(t+t_n/2^{k+1})-1\right|\right)\right]\\
		&\leq\sum_{k:\ t_n>2^{k-1}c_0x}
		\exp\left(-r_k2^k\tilde{c}xt_n^{-1}\right)\E\left[\exp\left(r_kp^{-1}\left|(H_n-H)(t+t_n/2^{k+1})-(H_n-H)(t)\right|\right)\right],
		\end{split}
		\]
		where $p=H(t+t_n/2^{k+1})-H(t)$.
		Since $X=n[H_n(t+t_n/2^{k+1})-H_n(t)]$ is a Binomial distribution with parameters $n$ and  $p$, it follows from the first inequality in Lemma \ref{le:binomial} that
		\[
		\begin{split}
		&\p\left(\sup_{\substack{s\in[t+c_0x,t+t_n]\\ s\leq 1}}\left|(H_n-H)(s)-(H_n-H)(t)\right|\geq cx\right)\\
		&\leq\sum_{k:\ t_n>2^{k-1}c_0x}2\exp\left(-r_k2^k\tilde{c}xt_n^{-1}+np{\phi\left(\frac{r_k}{np}\right)}\right).
		\end{split}
		\]
		In the above sum, we choose $r_k$ for which the exponential takes the smallest possible value, that is $r_k=np\log(1+	2^k\tilde{c}xt_n^{-1})>0$. With this choice we obtain
		\[
		\begin{split}
		&
		\p\left(\sup_{\substack{s\in[t+c_0x,t+t_n]\\ s\leq 1}}\left|(H_n-H)(s)-(H_n-H)(t)\right|\geq cx\right)\\
		&\leq\sum_{k:\ t_n>2^{k-1}c_0x}2\exp\left(-nph(2^k\tilde{c}xt_n^{-1})\right)
		\end{split}
		\]
		where for all $u\geq -1$, $h(u)=(1+u)\log(1+u)-u$. Since $h(u)\geq u^2/(2+2u)$ for all $u\geq 0$ (see the proof of Lemma \ref{le:binomial}) and $p\geq t_n\inf_uh(u)/2^{k+1}$, this implies that
		\[
		\begin{split}
		&
		\p\left(\sup_{\substack{s\in[t+c_0x,t+t_n]\\ s\leq 1}}\left|(H_n-H)(s)-(H_n-H)(t)\right|\geq cx\right)\\
		&\leq\sum_{k:\ t_n>2^{k-1}c_0x}2\exp\left(-Kn\frac{x^2t_n^{-1}}{2^{-k}+xt_n^{-1}}\right)\\
		&\leq\sum_{k=0}^\infty 2\exp\left(-Kn\frac{2^kx^2t_n^{-1}}{1+2/c_0}\right)
		\end{split}
		\]
		and therefore,
		\begin{equation}\notag
		\p\left(\sup_{\substack{s\in[t+c_0x,t+t_n]\\ s\leq 1}}\left|(H_n-H)(s)-(H_n-H)(t)\right|\geq cx\right)	\leq K_1\exp\left(-K_2nx^2t_n^{-1}\right)\end{equation}
		for all $t_n>0$ and $x>0$ that satisfy $t_n>c_0 x$. Since \eqref{eq:b1} holds for all $t_n\leq c_0 x$	we conclude from \eqref{eq:c0x} that for all $t_n>0$ and $x>0$, we have
		\begin{equation}\label{eq:b2}
		\p\left(\sup_{\substack{s\in[t,t+t_n]\\ s\leq 1}}\left|(H_n-H)(s)-(H_n-H)(t)\right|\geq cx\right)	\leq K_1\exp\left(-K_2nx\min\{1,xt_n^{-1}\}\right).
		\end{equation}} 
	
	It remains to deal with the second probability in the right hand side of \eqref{eqn:6}. 
	We have
	\begin{equation}
	\label{eqn:prob2_2}
	\begin{split}
	&\p\left(\sup_{\substack{s\in[t,t+t_n]\\ s\leq 1}}\left\{{\sup_{u\in[0,1]}}|H_n(u)-H(u)|^2+{\sup_{u\in[0,1]}}h(u)(s-t)\sup_{u\in[0,1]}|H_n(u)-H(u)| \right\}\geq cx\right)\\
	&\leq \p\left(\sup_{u\in[0,1]}|H_n(u)-H(u)|^2\geq c_1x\right)+\p\left(\sup_{u\in[0,1]}|H_n(u)-H(u)|\geq c_2xt_n^{-1}\right)\\
	&\lesssim \exp\left(-cnx\min\left\{1,xt_n^{-2}\right\}\right),
	\end{split}
	\end{equation}
	{using again the DKW inequality.}
	Putting together \eqref{eqn:main}, \eqref{eqn:prob1}, \eqref{eqn:6}, \eqref{eq:b2} and \eqref{eqn:prob2_2},  we conclude that 
	\[
	\p\left(\sup_{\substack{s\in[t,t+t_n]\\ s\leq 1}}\left|S_n^2(t)-S_n^2(s) \right|\geq x\right)
	\leq K_1\exp\left(-K_2nx\min\left\{1,xt_n^{-1} \right\}\right)
	\]
	for some positive constants $K_1,\,K_2$. The same bound can be obtained similarly for the case $s\in[t-t_n,t]\cap[0,1]$.
\end{proof}

\begin{lemma}
	\label{le:R_1}
	Let $T_n=n^{\gamma}$ for some $\gamma\in (0,1/3)$.  For $t,\,s\in(0,1)$  define
	\begin{equation}
	\label{def:pi}
	\pi(u;t,s)=\1_{\left\{u\leq s\right\}}-\1_{\{u\leq t\}}
	\end{equation}
	and
	\begin{equation}
	\label{def:R_1}
	R^1_n(t,s)=\int\delta {\pi}(u;t,s)\left(\frac{1}{\Phi(u;\beta_0)}-\frac{1}{\Phi(t;\beta_0)}\right)\,\dd (\p_n-\p)(u,\delta,z).
	\end{equation}
	Suppose that $x\mapsto \Phi(x;\beta_0)$ is continuously differentiable and it is bounded from below by a strictly positive constant. {Let $c>0$.} For each $2\leq q< 2/(3\gamma)$,  there exists $K>0$ such that, for all $t\in[0,1]$ and $x\geq 0$,
	\[
	\p\left[\sup_{|t_n-t|\leq c  n^{-1/3}T_n}|R^1_n(t,t_n)|>x\right]\leq K x^{-q} n^{1-q}.
	\]
\end{lemma}
\begin{proof}
	{Fix $c>0$, $x>0$ and $t\in[0,1]$.}
	Let $\F$ be the following class of functions on ${[t,1]}\times\{0,1\}\times\R^p$,
	\[
	\F=\left\{f_{t_n}(u,\delta,z)= \delta {\pi}(u;t,t_n)\left(\frac{1}{\Phi(u;\beta_0)}-\frac{1}{\Phi(t;\beta_0)}\right)\,\bigg|\,t_n\in[t,t+cn^{-1/3}T_n]\right\}
	\]	
	with envelope function
	\[
	{F}(u,\delta,z)=\delta {\pi}\left(u;t,t+cn^{-1/3}T_n\right)\left(\frac{1}{\Phi(u;\beta_0)}-\frac{1}{\Phi(t;\beta_0)}\right).
	\]
	By Markov inequality we have
	\[
	\p\left[\sup_{t_n\in[t,t+cn^{-1/3}T_n]}|R^1_n(t,t_n)|>x\right]{\leq} n^{-q/2}x^{-q}\E\left[\sup_{f\in\F}\left|\int f(u,\delta,z)\,\dd\sqrt{n}(\p_n-\p)(u,\delta,z) \right|^q\right].
	\]
	From Theorem 2.14.5 and 2.14.2 in \cite{VW96}, it follows that 
	\begin{equation}
	\label{eqn:q_moment_emipirical}
	\begin{split}
	&\E\left[\sup_{f\in\F}\left|\int f(u,\delta,z)\,\dd\sqrt{n}(\p_n-\p)(u,\delta,z) \right|^q\right]^{1/q}\\
	&\lesssim \E\left[\sup_{f\in\F}\left|\int f(u,\delta,z)\,\dd\sqrt{n}(\p_n-\p)(u,\delta,z) \right|\right]+n^{-1/2+1/q}\Vert F\Vert_{L_q(\p)}\\
	&\lesssim J_{[]}(1,\F,L_2(\p))\Vert F\Vert_{L_2(\p)}+n^{-1/2+1/q}\Vert F\Vert_{L_q(\p)},
	\end{split}
	\end{equation}
	where
	\[
	J_{[]}(1,\F,L_2(\p))=\int_0^1\sqrt{1+\log N_{[]}\left(\epsilon\Vert F\Vert_{L_2(\p)},\F,L_2(\p) \right)}\,\dd\epsilon.
	\]
	The constants in the inequalities $\lesssim$ are universal.
	{Since $\Phi$ is bounded from below by a strictly positive constant and has a continuous derivative, we} have
	\[
	\begin{split}
	\Vert F\Vert_{L_q(\p)}^q&=\int \delta \1_{\left\{t<u\leq t+cn^{-1/3}T_n\right\}}\left(\frac{1}{\Phi(u;\beta_0)}-\frac{1}{\Phi(t;\beta_0)}\right)^q\,\dd \p(u,\delta,z)\\
	&=\int_{t}^{t+cn^{-1/3}T_n}(u-t)^q \left(\frac{\Phi'(t;\beta_0)}{\Phi(t;\beta_0)^2}\right)^q\,\dd H^{uc}(u)+o\left((n^{-1/3}T_n)^{q+1}\right)\\
	&={(q+1)^{-1}\left(\frac{\Phi'(t;\beta_0)}{\Phi(t;\beta_0)^2}\right)^q}
	(n^{-1/3}T_n)^{q+1}+o\left((n^{-1/3}T_n)^{q+1}\right),
	\end{split}
	\]
	{where the small$-o$ term is uniform in $t$}.  Hence, if $J_{[]}(1,\F,L_2(\p))$ is bounded {uniformly in $t$}, we obtain
	\[
	\begin{split}
	\E\left[\sup_{f\in\F}\left|\int f(u,\delta,z)\,\dd\sqrt{n}(\p_n-\p)(u,\delta,z) \right|^q\right]
	&\lesssim \Vert F\Vert_{L_2(\p)}^q+n^{-q/2+1}\Vert F\Vert_{L_q(\p)}^q\\
	&\lesssim (n^{-1/3}T_n)^{3q/2}+n^{-q/2+1}(n^{-1/3}T_n)^{q+1}
	\end{split}
	\]
	{where the constants in the inequalities $\lesssim$ can be chosen independently of $t$.}
	It follows that,
	\[
	\begin{split}
	\p\left[\sup_{t_n\in[t,t+cn^{-1/3}T_n]}|R^1_n(t,t_n)|>x\right]
	&\lesssim n^{-q}x^{-q}T_n^{3q/2}+n^{1-q}x^{-q}(n^{-1/3}T_n)^{q+1}\\
	&\lesssim n^{1-q}x^{-q},
	\end{split}
	\]
	{since $q\leq 2/(3\gamma).$}
	It remains to show that $J_{[]}(1,\F,L_2(\p))$ is bounded. We first compute $N_{[]}\left(\epsilon\Vert F\Vert_{L_2(\p)},\F,L_2(\p) \right)$. Divide the interval $[t,t+cn^{-1/3}T_n]$ in $M$ subintervals of length $L=cn^{-1/3}T_n/M$. For $i=1,\dots,M$, let
	\[
	l_i(u,\delta,z)=\delta\1_{\left\{t<u\leq t+(i-1)L\right\}}\left(\frac{1}{\Phi(u;\beta_0)}-\frac{1}{\Phi(t;\beta_0)}\right),
	\]
	\[
	L_i(u,\delta,z)=\delta\1_{\left\{t<u\leq t+iL\right\}}\left(\frac{1}{\Phi(u;\beta_0)}-\frac{1}{\Phi(t;\beta_0)}\right).
	\]
	Consider the brackets $[l_1,L_1],\dots,[l_M,L_M]$. Since, for each $t_n\in[t,t+cn^{-1/3}T_n]$, there exists $i$ such that $f_{t_n}\in [l_i,L_i]$, they cover all the class $\F$. Moreover, with $c_i=t+iL,$ we have
	\[
	\begin{split}
	\Vert L_i-l_i\Vert_{L_2(\p)}^2&=\int\left(L_i(u,\delta,z)-l_i(u,\delta,z)\right)^2\,\dd\p(u,\delta,z)\\
	&\lesssim\int \delta\1_{\{c_{i-1}<u\leq c_i\}}\left(\frac{1}{\Phi(u;\beta_0)}-\frac{1}{\Phi(t;\beta_0)}\right)^2\,\dd\p(u,\delta,z)\\
	&\lesssim L\left(n^{-1/3}T_n\right)^2.
	\end{split}
	\]
	Again, the constant in the inequality $\lesssim$ can be chosen independently of $t$.
	If we take 
	\[
	M=\left\lfloor\frac{KT_n^3}{n{\epsilon^2 \Vert F\Vert_{L_2(\p)}^2}}+1\right\rfloor,
	\]
	for a properly chosen constant $K>0$, then 
	\[
	\Vert L_i-l_i\Vert_{L_2(\p)}\leq \epsilon\Vert F\Vert_{L_2(\p)}.
	\]
	It follows that 
	\begin{equation}\label{eq: entropy}
	N_{[]}\left(\epsilon\Vert F\Vert_{L_2(\p)},\F,L_2(\p) \right)\leq \left\lfloor\frac{KT_n^3}{n{\epsilon^2 \Vert F\Vert_{L_2(\p)}^2}}+1\right\rfloor 
	\end{equation}
	and consequently
	\[
	\begin{split}
	J_{[]}(1,\F,L_2(\p))&\leq \int_0^{1}\sqrt{1+\log \left\lfloor\frac{KT_n^3}{n{\epsilon^2 \Vert F\Vert_{L_2(\p)}^2}}+1\right\rfloor  }\,\dd\epsilon\\
	&{\lesssim 1+\int_0^{(KT_n^3)^{1/2}/(n^{1/2}\Vert F\Vert_{L_2(\p)})}\sqrt{1+\log \left(\frac{KT_n^3}{n{\epsilon^2 \Vert F\Vert_{L_2(\p)}^2}}\right) } \,\dd\epsilon}\\
	&\lesssim 1+\frac{{(KT_n^{3})^{1/2}}}{{n^{1/2}}\Vert F\Vert_{L_2(\p)}}\int_0^{{1}}\sqrt{{1+\log(x^{-2})}}\,\dd x\\
	&\leq K'
	\end{split}
	\] 
	because $\Vert F\Vert_{L_2(\p)}=C(n^{-1/3}T_n)^{3/2}+o\left((n^{-1/3}T_n)^{3/2}\right)$.
	Similarly, it can be shown that 
	\[
	\p\left[\sup_{t_n\in[t-cn^{-1/3}T_n,t]}|R^1_n(t,t_n)|>x\right]\lesssim n^{1-q}x^{-q},
	\]
	for some universal constant in $\lesssim$.
\end{proof}
\begin{lemma}
	\label{le:R_2}
	Let $T_n=n^{\gamma}$ for some $\gamma\in (0,1/3)$. 
	For $t,\,s\in(0,1)$, define
	\begin{equation}
	\label{def:R_2}
	R^2_n(t,s)=\int e^{{\beta_0'} z}\pi(u;t,s)\int_{u}^{s}\frac{\lambda_0(v)}{\Phi(v;\beta_0)}\,\dd v\,\dd(\p_n-\p)(u,\delta,z),
	\end{equation}
	where {$\pi$} is as in \eqref{def:pi}.
	Suppose $x\mapsto \Phi(x;\beta_0)$  and $\lambda_0$ are  continuously differentiable bounded away from zero. 
	Assume that, given $z$, the follow up time $T$ has a continuous density uniformly bounded in $z$ and $t$, from above and below away from zero, {and assume furthermore that (A3) holds}. 
	For each $q\in[2,2/(3\gamma)],$ {and $c>0$} 
	there exist $K>0$ such that, for all $t\in[0,1]$ and $x\geq 0$,
	\[
	\p\left[\sup_{|t_n-t|\leq cn^{-1/3}T_n}|R^2_n(t,t_n)|>x\right]\leq Kx^{-q} n^{1-q}.
	\]
\end{lemma}
\begin{proof}
	{Fix $c>0$, $x>0$ and $t\in[0,1]$.}
	Let $\F$ be the following class of functions {$f_{t_n}$} on ${[t,t_n]}\times\{0,1\}\times\R^p$,
	\[
	\F=\left\{f_{t_n}(u,\delta,z)= e^{\beta'_0 z}\pi(u;t,t_n)\int_{u}^{t_n}\frac{\lambda_0(v)}{\Phi(v;\beta_0)}\,\dd v\,\bigg|\,t_n\in[t,t+cn^{-1/3}T_n]\right\}.
	\]	
	with envelope function
	\[
	{F}(u,\delta,z)=e^{\beta'_0 z}\pi(u;t,t+cn^{-1/3}T_n)\int_{u}^{t+cn^{-1/3}T_n}\frac{\lambda_0(v)}{\Phi(v;\beta_0)}\,\dd v.
	\]
	By Markov inequality we have
	\[
	\p\left[\sup_{t_n\in[t,t+cn^{-1/3}T_n]}|R^2_n(t,t_n)|>x\right]{\leq} n^{-q/2}x^{-q}\E\left[\sup_{f\in\F}\left|\int f(u,\delta,z)\,\dd\sqrt{n}(\p_n-\p)(u,\delta,z) \right|^q\right].
	\]
	Moreover, as in \eqref{eqn:q_moment_emipirical}, 
	\[
	\E\left[\sup_{f\in\F}\left|\int f(u,\delta,z)\,\dd\sqrt{n}(\p_n-\p)(u,\delta,z) \right|^q\right]^{1/q}\lesssim J_{[]}(1,\F,L_2(\p))\Vert F\Vert_{L_2(\p)}+n^{-1/2+1/q}\Vert F\Vert_{L_q(\p)}.
	\]
	The constants in the inequalities $\lesssim$ are universal.
	In this case, we have
	\[
	\begin{split}
	\Vert F\Vert_{L_q(\p)}^q&=\int e^{q\beta'_0z} {\pi}(u;t,t+cn^{-1/3}T_n)\left(\int_{u}^{t+cn^{-1/3}T_n}\frac{\lambda_0(v)}{\Phi(v;\beta_0)}\,\dd v \right)^q\,\dd \p(u,\delta,z)\\
	&=\int_{\R^p} e^{q\beta'_0z}g_n(z) \,\dd F_Z(z)+o\left(\int_{\R^p} e^{q\beta'_0z}g_n(z) \,\dd F_Z(z)\right)
	\end{split}
	\]
	for 
	\[
	\begin{split}
	g_n(z)&=\int_{t}^{t+cn^{-1/3}T_n}\left[t+cn^{-1/3}T_n-u\right]^q \left(\frac{\lambda_0({t})}{\Phi({t};\beta_0)}\right)^q\,\dd H(u|z).
	\end{split}
	\]
	Since $\lambda_{0},$ $\Phi$ and the density of $T$ given $z$ are uniformly bounded above and below away from zero, there exist constants $c_1>0$, $c_2>0$ {independent of $t$} such that 
	\[
	c_1(n^{-1/3}T_n)^{q+1}\leq \inf_{z} g_n(z)\leq \sup_{z}{g_n(z)}\leq c_2(n^{-1/3}T_n)^{q+1}
	\]
	Hence {under Assumption (A3),} if $J_{[]}(1,\F,L_2(\p))$ is bounded, 
	as in the previous lemma it follows that 
	\[
	\p\left[\sup_{t_n\in[t,t+cn^{-1/3}T_n]}|R^2_n(t,t_n)|>x\right]\lesssim n^{1-q}x^{-q}.
	\]
	It remains to show that $J_{[]}(1,\F,L_2(\p))$ is bounded. We first compute $N_{[]}\left(\epsilon\Vert F\Vert_{L_2(\p)},\F,L_2(\p) \right)$. Divide the interval $[t,t+cn^{-1/3}T_n]$ in $M$ subintervals of length $L=cn^{-1/3}T_n/M$. For $i=1,\dots,M$, let
	\[
	l_i(u,\delta,z)=e^{\beta'_0z}\1_{\left\{t<u\leq t+(i-1)L\right\}}\int_{u}^{t+(i-1)L}\frac{\lambda_0(v)}{\Phi(v;\beta_0)}\,\dd v,
	\]
	\[
	L_i(u,\delta,z)=e^{\beta'_0z}\1_{\left\{t<u\leq t+iL\right\}}\int_{u}^{t+iL}\frac{\lambda_0(v)}{\Phi(v;\beta_0)}\,\dd v.
	\]
	Consider the brackets $[l_1,L_1],\dots,[l_M,L_M]$. Since, for each $u\in[0,T_n]$, there exists $i$ such that $f_u\in [l_i,L_i]$, they cover all the class $\F$. Moreover, {using the assumption (A3) with $q=2$ together with the assumption that the conditional distribution of $T$ given $Z=z$ has a density that is bounded uniformly in $z$,}  we have
	\[
	\begin{split}
	\Vert L_i-l_i\Vert_{L_2(\p)}^2
	&\lesssim\int e^{2\beta'_0z}\1_{\{t<u\leq t+iL\}}\left(\int_{t+(i-1)L}^{t+iL}\frac{\lambda_0(v)}{\Phi(v;\beta_0)}\,\dd v\right)^2\,\dd\p(u,\delta,z)\\
	&\quad+\int e^{2\beta'_0z}\1_{\{t+(i-1)L<u\leq t+iL\}}\left(\int_{u}^{t+(i-1)L}\frac{\lambda_0(v)}{\Phi(v;\beta_0)}\,\dd v\right)^2\,\dd\p(u,\delta,z)\\
	&\lesssim iL^3+{((i-1)L+n^{-1/3}T_n)^2L}\lesssim L^3M^2.
	\end{split}
	\]
	If we let 
	\[
	M=\left\lfloor \frac{KT_n^3}{n\epsilon^2 \Vert F\Vert_{L_2(\p)}^2}+1\right\rfloor
	\]
	for a properly chosen constant $K>0$,	then 
	\[
	\Vert L_i-l_i\Vert_{L_2(\p)}\lesssim L^{3/2}M\leq \epsilon\Vert F\Vert_{L_2(\p)}.
	\]
	It follows that 
	\[
	N_{[]}\left(\epsilon\Vert F\Vert_{L_2(\p)},\F,L_2(\p) \right)\leq\left\lfloor \frac{KT_n^3}{n\epsilon^2 \Vert F\Vert_{L_2(\p)}^2}+1\right\rfloor.
	\]
	{Since we obtain the same entropy bound as in \eqref{eq: entropy}, the proof can be completed with the same arguments as for the proof of Lemma \ref{le:R_1}.}
\end{proof}
\begin{proof}[Proof of Step 1 in Theorem \ref{theo:CLT}]
	Let 
	\[
	I_1=\int_0^1\left[\hat{\lambda}_n(t)-\lambda_0(t)\right]_+^p\,\text{d}t\quad\text{and}\quad J_1=\int_0^1\int_0^{(\lambda_0(0)-\lambda_0(t))^p}\1_{\{\hat{\lambda}_n(t)\geq \lambda_0(t)+a^{1/p}\}}\,\dd a\,\dd t.
	\]
	Since $\hat{\lambda}_n(t)<{\lambda_0(0)}$ for $t>\hat{U}_n(\lambda_0(0))$, we have
	\[
	\begin{split}
	0&\leq (I_1-J_1)\1_{E_n}=\1_{E_n}\int_0^{\hat{U}_n(\lambda_0(0))}\int_{(\lambda_0(0)-\lambda_0(t))^p}^\infty \1_{\{\hat{\lambda}_n(t)\geq \lambda_0(t)+a^{1/p}\}}\,\dd a\,\dd t\\
	&\leq \1_{E_n} \int_0^{\hat{U}_n(\lambda_0(0))}\left[\hat{\lambda}_n(t)-\lambda_0(t)\right]_+^p\,\dd t\\
	&\leq \1_{E_n}\int_0^{n^{-1/3}\log n} \left[\hat{\lambda}_n(t)-\lambda_0(t)\right]_+^p\,\dd t+\left|\hat{\lambda}_n(0)-\lambda_0(1)\right|^p \1_{\{n^{1/3}\hat{U}_n(\lambda_0(0))>\log n \}\cap E_n}.
	\end{split}
	\]
	Let $p'\in(p-1/2,2)$be such that $1\leq p'\leq p$. From Lemma \ref{le:boundaries} and Lemma \ref{le:inv_tail_prob1} it follows that
	\[
	(I_1-J_1)\1_{E_n}\leq \1_{E_n} \left|\hat{\lambda}_n(0)-\lambda_0(1)\right|^{p-p'}\int_0^{n^{-1/3}\log n} \left[\hat{\lambda}_n(t)-\lambda_0(t)\right]_+^{p'}\,\dd t +o_P\left(n^{-p/3-1/6}\right).
	\]
	Moreover, as $p'\in[1,2)$, Theorem \ref{theo:bound_expectation_L_P} implies that 
	\[
	\E\left[\1_{E_n} \int_{n^{-1}}^{n^{-1/3}\log n} \left[\hat{\lambda}_n(t)-\lambda_0(t)\right]_+^{p'}\,\dd t\right]\lesssim n^{-(1+p')/3}\log n.
	\]
	Since $p'>p-1/2$ and $p\leq 5/2$,
	\[
	\begin{split}
	&\1_{E_n} \int_0^{n^{-1/3}\log n} \left[\hat{\lambda}_n(t)-\lambda_0(t)\right]_+^{p'}\,\dd t\\
	&=\1_{E_n} \int_0^{n^{-1}} \left[\hat{\lambda}_n(t)-\lambda_0(t)\right]_+^{p'}\,\dd t+O_P\left(n^{-(1+p')/3}\log n\right)\\
	&\leq \1_{E_n} \left|\hat{\lambda}_n(0)-\lambda_0(1)\right|^{p'} O(n^{-1})+o_P\left(n^{-p/3-1/6}\right)\\
	\end{split}
	\]
	{whence $I_1\1_{E_n}=J_1\1_{E_n}+o_P\left(n^{-p/3-1/6}\right).$}
	With a change of variable $b=\lambda_0(t)+a^{1/p}$ we get
	\[
	I_1\1_{E_n}=\1_{E_n}\int_{\lambda_0(1)}^{\lambda_0(0)}\int _{U(b)}^{\hat{U}_n(b)}p(b-\lambda_0(t))^{p-1}\1_{\{U(b)<\hat{U}_n(b) \}}\,\text{d}t\,\dd b+o_P\left(n^{-p/3-1/6}\right).
	\]
	By a Taylor expansion and \eqref{eqn:assumption_derivative} we have
	\[
	\left|[b-\lambda_0(t)]^{p-1}-\left[(U(b)-t)\lambda'_0\circ U(b)\right]^{p-1}\right|\lesssim (t-U(b))^{p-1+s}
	\] 
	for all $b\in(\lambda_0(1),{\lambda_0(0)})$ and $t\in(U(b),1)$. Integrating \eqref{eqn:inv} shows that, for all $q\geq 1$, there exists $K>0$ such that 
	\[
	\E\left[\1_{E_n}\left(n^{1/3}|\hat{U}_n(a)-U(a)|\right)^{q} \right]\leq K, \qquad\text{for all }a\in\R.
	\]
	It follows that 
	\[
	\begin{split}
	I_1\1_{E_n}&=\1_{E_n}\int_{\lambda_0(1)}^{\lambda_0(0)}\int _{U(b)}^{\hat{U}_n(b)}p(t-U(b))^{p-1}|{\lambda_0}'(U(b))|^{p-1}\1_{\{U(b)<\hat{U}_n(b) \}}\,\text{d}t\,\dd b\\
	&\quad+O_P\left(n^{-(p+s)/3}\right)+o_P\left(n^{-p/3-1/6}\right)\\
	&=\1_{E_n}\int_{\lambda_0(1)}^{\lambda_0(0)}{(\hat{U}_n(b)-U(b))^{p}|\lambda_0}'(U(b))|^{p-1}\1_{\{U(b)<\hat{U}_n(b) \}}\,\dd b+o_P\left(n^{-p/3-1/6}\right)\\
	&={\1_{E_n}\int_{\lambda_0(1)}^{\lambda_0(0)}(\hat{U}_n(b)-U(b))_+^{p}|U'(b))|^{p-1}\dd b+o_P\left(n^{-p/3-1/6}\right)}.
	\end{split}
	\]
	In the same way we can {prove that}
	\[
	{\1_{E_n}}\int_0^1\left[\lambda_0(t)-\hat{\lambda}_n(t)\right]_+^p\,\text{d}t={\1_{E_n}\int_{\lambda_0(1)}^{\lambda_0(0)}(U(b)-\hat{U}_n(b))_+^{p}|U'(b))|^{p-1}\dd b+o_P\left(n^{-p/3-1/6}\right)}
	\] 	
	{and the result follows.}
\end{proof}

\begin{lemma}
	\label{le:R_3}
	Let $T_n=n^{\gamma}$ for some $\gamma\in (0,1/3)$  and $b=n^{-1/4}$.  	
	Assume that, given $z$, the follow up time $T$ has a continuous density uniformly bounded from above and from below away from zero. 
	For   each $2\leq q<1/(\gamma+1/24)$,  there {exists} $K>0$ such that, for $x\geq 0$  and all $a\in\R$, 
	\[
	\p\left[n^{1/3}T_n\sup_{|v-U(a)|\leq b}\left|(\Phi_n -\Phi)(v;\beta_0)-(\Phi_n-\Phi )(U(a);\beta_0)\right|>x\right]\leq K x^{-q}n^{1-q/3}.
	\]
	
\end{lemma}
\begin{proof}
	Let $\F$ be the following class of functions {$f_v$} on ${[U(a),v]}\times\{0,1\}\times\R^p$,
	\[
	\F=\left\{f_v(t,\delta,z)=  \1_{\left\{U(a)<t\leq v\right\}}e^{\beta'_0z}\,\bigg|\,v\in\left[U(a),U(a)+b\right]\right\}
	\]	
	with envelope function
	\[
	F(t,\delta,z)=\1_{\left\{U(a)<t\leq U(a)+b\right\}}e^{\beta'_0z}.
	\]
	By Markov inequality we have
	\[
	\begin{split}
	&\p\left[n^{1/3}T_n\sup_{v\in\left[U(a),U(a)+b\right]}\left|(\Phi_n -\Phi)(v;\beta_0)-(\Phi_n-\Phi )(U(a);\beta_0)\right|>x \right]\\
	&\lesssim n^{-q/6}x^{-q}T_n^q\E\left[\sup_{f\in\F}\left|\int f(t,\delta,z)\,\dd\sqrt{n}(\p_n-\p)(t,\delta,z) \right|^q\right].
	\end{split}
	\]
	As in \eqref{eqn:q_moment_emipirical}, it follows that 
	\[
	\begin{split}
	\E\left[\sup_{f\in\F}\left|\int f(t,\delta,z)\,\dd\sqrt{n}(\p_n-\p)(t,\delta,z) \right|^q\right]^{1/q}\lesssim J_{[]}(1,\F,L_2(\p))\Vert F\Vert_{L_2(\p)}+n^{-1/2+1/q}\Vert F\Vert_{L_q(\p)}
	\end{split}
	\]
	for a universal constant in $\lesssim$.
	For the envelope function, we obtain
	\[
	\begin{split}
	\Vert F\Vert_{L_q(\p)}^q&=\int \1_{\left\{U(a)<t\leq U(a)+b\right\}}e^{q\beta'_0z}\,\dd \p(t,\delta,z)\\
	&=\int_{\R^p}e^{q\beta_0z}\int_{U(a)}^{U(a)+n^{-1/3}}\,\dd H(t|z)\,\dd F_Z(z)\,\in[c_1b,c_2b],
	\end{split}
	\]
	for some $c_1,c_2>0$ independent of $a$ because the density of $T$ given $z$ is uniformly bounded from above and below away from zero. Hence, if $J_{[]}(1,\F,L_2(\p))$ is bounded, we obtain
	\[
	\begin{split}
	&\E\left[\sup_{f\in\F}\left|\int f(t,\delta,z)\,\dd\sqrt{n}(\p_n-\p)(t,\delta,z) \right|^q\right]\\
	&\lesssim \Vert F\Vert_{L_2(\p)}^q+n^{-q/2+1}\Vert F\Vert_{L_q(\p)}^q\\
	&{\lesssim}  b^{q/2}+n^{-q/2+1}b
	\end{split}
	\]
	and it follows that,
	\[
	\begin{split}
	&\p\left[n^{1/3}T_n\sup_{v\in\left[U(a),U(a)+b\right]}\left|(\Phi_n -\Phi)(v;\beta_0)-(\Phi_n-\Phi )(U(a);\beta_0)\right|>x \right]\\
	&\lesssim n^{-7q/24}x^{-q}T_n^{q}+n^{-2q/3+3/4}x^{-q}T_n^{q}\\
	&\lesssim n^{1-q/3}x^{-q},
	\end{split}
	\]
	{since $q\leq 1/(\gamma+1/24)$.} Again, the constant in $\lesssim$ is independent of $a$.
	It remains to show that $J_{[]}(1,\F,L_2(\p))$ is bounded. We first compute $N_{[]}\left(\epsilon\Vert F\Vert_{L_2(\p)},\F,L_2(\p) \right)$. Divide the interval $\left[U(a),U(a)+b\right]$ in $M$ subintervals of length $L=b/M$. For $i=1,\dots,M$, define
	\[
	l_i(t,\delta,z)= \1_{\left\{U(a)<t\leq U(a)+(i-1)L\right\}}e^{\beta'_0z},\qquad L_i(t,\delta,z)= \1_{\left\{U(a)<t\leq U(a)+iL\right\}}e^{\beta'_0z}.
	\]
	Consider the brackets $[l_1,L_1],\dots,[l_M,L_M]$. Since, for each $u\in{\left[U(a),U(a)+b\right]}$, there exists $i$ such that $f_u\in [l_i,L_i]$, they cover all the class $\F$. Moreover, we have
	\[
	\begin{split}
	\Vert L_i-l_i\Vert_{L_2(\p)}^2
	&\lesssim\int \1_{\{U(a)+(i-1)L<t\leq U(a)+iL\}}e^{2\beta'_0z}\,\dd\p(t,\delta,z)\lesssim L.
	\end{split}
	\]
	Since the density of $T$ given $z$ is uniformly bounded from above, the constant in the inequality $\lesssim$ is independent  of $t$.
	If we take 
	\[
	M=\left\lfloor\frac{Kb}{{\epsilon^2 \Vert F\Vert_{L_2(\p)}^2}}+1\right\rfloor
	\]
	for some properly chosen constant $K$,	then
	\[
	\Vert L_i-l_i\Vert_{L_2(\p)}\leq \epsilon\Vert F\Vert_{L_2(\p)}.
	\]
	It follows that 
	\[
	N_{[]}\left(\epsilon\Vert F\Vert_{L_2(\p)},\F,L_2(\p) \right)\leq \left\lfloor\frac{Kb}{{\epsilon^2 \Vert F\Vert_{L_2(\p)}^2}}+1\right\rfloor 
	\]
	and consequently
	\[
	\begin{split}
	J_{[]}(1,\F,L_2(\p))&\leq \int_0^{1}\sqrt{1+\log \left\lfloor\frac{Kb}{{\epsilon^2 \Vert F\Vert_{L_2(\p)}^2}}+1\right\rfloor  }\,\dd\epsilon\\
	&{\lesssim 1+\int_0^{(Kb)^{1/2}/\Vert F\Vert_{L_2(\p)}}\sqrt{1+\log \left(\frac{Kb}{\epsilon \Vert F\Vert_{L_2(\p)}}\right)  }\,\dd\epsilon}\\
	&\lesssim 1+\frac{{(Kb)^{1/2}}}{\Vert F\Vert_{L_2(\p)}}\int_0^{{1}}\sqrt{{1+\log (x^{-2})}}\,\dd x\\
	&\leq K'
	\end{split}
	\] 
	because $\Vert F\Vert_{L_2(\p)}=cb^{1/2}+o(b^{1/2})$.
	Similarly, it can be shown that 
	\[
	\p\left[n^{1/3}T_n\sup_{v\in\left[U(a)-b,U(a)\right]}\left|(\Phi_n -\Phi)(v;\beta_0)-(\Phi_n-\Phi )(U(a);\beta_0)\right|>x \right]\lesssim n^{1-q/3}x^{-q}.
	\]
\end{proof}

\section{CLT under a parametric baseline distribution}
\label{sec:parametric}
In what follows, {without loss of generality we assume that $|\hat\theta_n-\theta|\leq n^{-1/2+\alpha}$ and \eqref{eqn:par_est-true} hold on $E_n$, where $\alpha\in(0,1/6)$ is fixed. $\p_\theta$ and $\E_\theta$ denote the conditional probability and expectation, given $\hat\theta_n$. Moreover, $U_{\hat\theta}$ is defined in the same manner as $\hat U_n$ in \eqref{def:U_n} but with $\Lambda_n$ replaced by $\Lambda_{\hat\theta}$, the cumulative baseline hasard corresponding to $\lambda_{\hat\theta}$.}
\begin{lemma}
	\label{le:R_1_2}
	Let $T_n=n^{\gamma}$ for some $\gamma\in (0,1/3)$. {Let} $R_n^1$ be as in  \eqref{def:R_1}.
	Suppose that {$x\mapsto \Phi(x;\beta_0)$ is  continuously differentiable and it is  bounded from below by a strictly positive constant. Let $c>0$. For each $2\leq q<2/(3\gamma)$, there exist $K>0$ such that, for all $a\in\R$ and $x\geq 0$,}
	\[
	\p_\theta\left[\left\{\sup_{|t_n-U_{\hat\theta}(a)|\leq cn^{-1/3}T_n}|R^1_n(U_{\hat\theta}{(a)},t_n)|>x\right\}\cap E_n\right]\leq Kx^{-q} n^{1-q}.
	\]
\end{lemma}
\begin{proof}
	Let {$\pi$} be as in \eqref{def:pi} 	and let $\F$ be the  class of functions on $[U(a)-n^{-1/2+\alpha},1]\times\{0,1\}\times\R^p$ of the form
	\[
	f_{t_n,s}(u,\delta,z)= \delta {\pi}\left(u;U(a)+s,t_n\right)\left(\frac{1}{\Phi(u;\beta_0)}-\frac{1}{\Phi(U(a)+s;\beta_0)}\right)
	\]
	for $t_n\in[U(a)+s,U(a)+s+cn^{-1/3}T_n]$ and $|s|\leq n^{-1/2+\alpha}$. 
	The  envelope function is 
	\[
	F_n(u,\delta,z)=\delta \1_{\left\{U(a)-n^{-1/2+\alpha}<u\leq U(a)+n^{-1/2+\alpha}+n^{-1/3}T_n\right\}}\left(\frac{1}{\Phi(u;\beta_0)}-\frac{1}{\Phi(U(a)-n^{-1/2+\alpha};\beta_0)}\right).
	\]
	Then, {using that \eqref{eqn:par_est-true} holds on $E_n$,}
	\[
	\begin{split}
	&\p_\theta\left[\left\{\sup_{|t_n-U_{\hat\theta}(a)|\leq cn^{-1/3}T_n}|R^1_n(U_{\hat\theta}{(a)},t_n)|>x\right\}\cap E_n\right]\\
	&\leq \p\left[\sup_{f\in\F}\left|\int f(u,\delta,z)\,\dd\sqrt{n}(\p_n-\p)(u,\delta,z)\right|>\sqrt{n}x\right].
	\end{split}
	\]
	The rest of the proof remains the same as in Lemma \ref{le:R_1} because 
	\[
	\begin{split}
	\Vert F\Vert_{L_q(\p)}^q&=\int_{U(a)-n^{-1/2+\alpha}}^{U(a)+n^{-1/2+\alpha}+n^{-1/3}T_n}(u-U(a)+n^{-1/2+\alpha})^q \left(\frac{\Phi'(U(a)-n^{-1/2+\alpha};\beta_0)}{\Phi(U(a)-n^{-1/2+\alpha};\beta_0)^2}\right)^q\,\dd H^{uc}(u)\\
	&\quad+o\left((n^{-1/3}T_n)^{q+1}\right)\\
	&={(q+1)^{-1}\left(\frac{\Phi'(U(a)-n^{-1/2+\alpha};\beta_0)}{\Phi(U(a)-n^{-1/2+\alpha};\beta_0)^2}\right)^q}(n^{-1/3}T_n)^{q+1}+o\left((n^{-1/3}T_n)^{q+1}\right),
	\end{split}
	\]
	{where the small-$o$ term is uniform in $t$.} It remains to show that $J_{[]}(1,\F,L_2(\p))$ is bounded {uniformly in $t$}. We first compute $N_{[]}\left(\epsilon\Vert F\Vert_{L_2(\p)},\F,L_2(\p) \right)$. Divide the interval $[0,cn^{-1/3}T_n]$ in $M$ subintervals of length $L_1=cn^{-1/3}T_n/M$ and divide the interval $[-n^{-1/2+\alpha},n^{-1/2+\alpha}]$ in $M$ subintervals of length $L_2=n^{-1/2+\alpha}/M$. For $i,j=1,\dots,M$,  let $d_j=-n^{-1/2+\alpha}+jL_2$ and $c_{i,j}=U(a)+d_j+iL_1.$ Define 
	\[
	l_{i,j}(u,\delta,z)=\delta\1_{\left\{U(a)+d_{j}<u\leq c_{i-1,j-1}\right\}}\left(\frac{1}{\Phi(u;\beta_0)}-\frac{1}{\Phi(U(a)+d_{j-1};\beta_0)}\right),
	\]
	\[
	L_{i,j}(u,\delta,z)=\delta\1_{\left\{U(a)+d_{j-1}<t\leq c_{i,j}\right\}}\left(\frac{1}{\Phi(u;\beta_0)}-\frac{1}{\Phi(U(a)+d_j;\beta_0)}\right).
	\]
	The brackets $[l_{i,j},L_{i,j}]$, for $i,j=1,\dots,M$, cover all the class $\F$. Moreover, we have
	\[
	\begin{split}
	&\Vert L_{i,j}-l_{i,j}\Vert_{L_2(\p)}^2\\
	&\lesssim\int \delta\left[\1_{\{U(a)+d_{j-1}<u\leq U(a)+d_{j}\}}+\1_{\{c_{i-1,j-1}<u\leq c_{i,j}}\right]\left(\frac{1}{\Phi(u;\beta_0)}-\frac{1}{\Phi(U(a)+d_j;\beta_0)}\right)^2\,\dd\p(u,\delta,z)\\
	&\quad+\int \delta\1_{\{U(a)+d_{j}<u\leq c_{i-1,j-1}\}}\left(\frac{1}{\Phi(U(a)+d_{j-1};\beta_0)}-\frac{1}{\Phi(U(a)+d_j;\beta_0)}\right)^2\,\dd\p(u,\delta,z)\\
	&\lesssim L_1\left(n^{-1/3}T_n\right)^2+L_1^2n^{-1/3}T_n\lesssim L_1\left(n^{-1/3}T_n\right)^2,
	\end{split}
	\]
	{using that $\alpha<1/6$.}
	If we take 
	\[
	M=\left\lfloor\frac{KT_n^3}{n{\epsilon^2 \Vert F\Vert_{L_2(\p)}^2}}+1\right\rfloor,
	\]
	for a properly chosen constant $K>0$, we get
	\[
	\Vert L_{i,j}-l_{i,j}\Vert_{L_2(\p)}\leq \epsilon\Vert F\Vert_{L_2(\p)}.
	\]
	It follows that 
	\[
	N_{[]}\left(\epsilon\Vert F\Vert_{L_2(\p)},\F,L_2(\p) \right)\leq {\left\lfloor\frac{KT_n^3}{n\epsilon^2 \Vert F\Vert_{L_2(\p)}^2}+1\right\rfloor}
	\]
	and consequently, as in Lemma \ref{le:R_1}, 
	\[
	J_{[]}(1,\F,L_2(\p))\leq \int_0^{1}\sqrt{1+{\log} \left\lfloor\frac{KT_n^3}{n{\epsilon^2 \Vert F\Vert_{L_2(\p)}^2}}+1\right\rfloor  }\,\dd\epsilon\leq K
	\] 
\end{proof}
\begin{lemma}
	\label{le:R_2_2}
	Let $T_n=n^{\gamma}$ for some $\gamma\in (0,1/3)$. Let $R_n^2$ be as in \eqref{def:R_2}.	Suppose $x\mapsto \Phi(x;\beta_0)$ and $\lambda_0$  are continuously differentiable bounded away from zero. 
	Assume that, given $z$, the follow up time $T$ has a continuous density uniformly bounded in $z$ and $t$, from above and below away from zero, and assume furthermore that (A3) holds.
	For each $q\in[2,2/(3\gamma))$, and $c>0$, there exist $K>0$ such that, for all {$a\in\R$} and $x\geq 0$,
	\[
	\p_\theta\left[\left\{\sup_{|t_n-U_{\hat\theta}(a)|\leq cn^{-1/3}T_n}|R^2_n(U_{\hat\theta}{(a)},t_n)>x\right\}\cap E_n\right]\leq Kx^{-q} n^{1-q}.
	\]
\end{lemma}
\begin{proof}
	Let ${\pi}$ be as in \eqref{def:pi} and let $\F$ be the class of functions on ${[U(a)+s,t_n]}\times\{0,1\}\times\R^p$ of the form 
	\[
	f_{t_n,s}(u,\delta,z)= e^{\beta_0' z}\pi(u;U(a)+s,t_n)\int_{t}^{t_n}\frac{\lambda_0(v)}{\Phi(v;\beta_0)}\,\dd v,
	\]
	where $t_n\in[U(a)+s,U(a)+s+cn^{-1/3}T_n]$ and $|s|\leq n^{-1/2+\alpha}$. The envelope function is 
	\[
	F_n(u,\delta,z)=e^{\beta_0' z}\1_{\left\{U(a)-n^{-1/2+\alpha}<u\leq U(a)+n^{-1/2+\alpha}+cn^{-1/3}T_n\right\}}\int_{t}^{U(a)+n^{-1/2+\alpha}+cn^{-1/3}T_n}\frac{\lambda_0(v)}{\Phi(v;\beta_0)}\,\dd v
	\]
	Then,  {using that \eqref{eqn:par_est-true} holds on $E_n$,}
	\[
	\begin{split}
	&\p_\theta\left[\left\{\sup_{|t_n-U_{\hat\theta}(a)|\leq cn^{-1/3}T_n}|R^2_n(U_{\hat\theta}{(a)},t_n)>x\right\}\cap E_n\right]\\
	&\leq\p\left[\sup_{f\in\F}\left|\int f(t,\delta,z)\,\dd\sqrt{n}(\p_n-\p)(t,\delta,z) \right|>n^{1/2}x\right].
	\end{split}
	\]
	The rest of the proof remains the same as in Lemma \ref{le:R_2} because we have
	\[
	\begin{split}
	\Vert F_n\Vert_{L_q(\p)}^q&=\int e^{q\beta_0'z}\1_{\left\{U(a)-n^{-1/2+\alpha}<u\leq U(a)+n^{-1/2+\alpha}+cn^{-1/3}T_n\right\}}\\
	&\qquad \left(\int_{u}^{U(a)+n^{-1/2+\alpha}+nc^{-1/3}T_n}\frac{\lambda_0(v)}{\Phi(v;\beta_0)}\,\dd v \right)^q\,\dd \p(u,\delta,z)\\
	&=\int_{\R^p} e^{q\beta_0'z}g_n(z) \,\dd F_Z(z)+o\left(\int_{\R^p} e^{q\beta_0'z}g_n(z) \,\dd F_Z(z)\right)
	\end{split}
	\]
	for 
	\[
	\begin{split}
	g_n(z)&=\int_{U(a)-n^{-1/2+\alpha}}^{U(a)+n^{-1/2+\alpha}+cn^{-1/3}T_n}\left[U(a)+n^{-1/2+\alpha}+cn^{-1/3}T_n-u\right]^q\left(\frac{\lambda_0(u)}{\Phi(u;\beta_0)}\right)^q\,\dd H(u|z).
	\end{split}
	\]
	{Since $\lambda_0$, $\Phi$ and the density of $T$ given $z$ are uniformly bounded above and below away from zero, there exist constants $c_1>0$ and $c_2>0$ independent of $t$ such that
		$$c_1(n^{-1/3}T_n)^{q+1}\leq \inf_zg_n(z)\leq\sup_zg_n(z)\leq c_2(n^{-1/3}T_n)^{q+1}.$$}
	It remains to show that $J_{[]}(1,\F,L_2(\p))$ is bounded. We first compute $N_{[]}\left(\epsilon\Vert F\Vert_{L_2(\p)},\F,L_2(\p) \right)$. Divide the interval $[0,cn^{-1/3}T_n]$ in $M$ subintervals of length $L_1=cn^{-1/3}T_n/M$ and divide the interval $[-n^{-1/2+\alpha},n^{-1/2+\alpha}]$ in $M$ subintervals of length $L_2=n^{-1/2+\alpha}/M$. For $i,j=1,\dots,M$,  let $d_j=-n^{-1/2+\alpha}+jL_2$ and $c_{i,j}=U(a)+d_j+iL_1.$ Define
	\[
	l_i(u,\delta,z)=e^{\beta_0'z}\1_{\left\{U(a)+d_{j}<t\leq c_{i-1,j-1}\right\}}\int_{u}^{U(a)+d_{j-1}+(i-1)L}\frac{\lambda_0(v)}{\Phi(v;\beta_0)}\,\dd v,
	\]
	\[
	L_i(u,\delta,z)=e^{\beta_0'z}\1_{\left\{U(a)+d_{j-1}<t\leq c_{i,j}\right\}}\int_{u}^{U(a)+d_j+iL}\frac{\lambda_0(v)}{\Phi(v;\beta_0)}\,\dd v.
	\]
	The brackets $[l_{i,j},L_{i,j}]$ cover all the class $\F$. Moreover, we have
	\[
	\begin{split}
	&\Vert L_{i,j}-l_{i,j}\Vert_{L_2(\p)}^2\\
	&\lesssim\int e^{2\beta_0'z}\left[\1_{\{U(a)+d_{j-1}<u\leq U(a)+d_{j}\}}+\1_{\{c_{i-1,j-1}<u\leq c_{i,j}}\right]\left(\int_{u}^{c_{i-1,j-1}}\frac{\lambda_0(v)}{\Phi(v;\beta_0)}\,\dd v\right)^2\,\dd\p(u,\delta,z)\\
	&\quad+\int e^{2\beta_0'z}\1_{\{U(a)+d_{j}<u\leq c_{i,j}\}}\left(\int_{c_{i-1,j-1}}^{c_{i,j}}\frac{\lambda_0(v)}{\Phi(v;\beta_0)}\,\dd v\right)^2\,\dd\p(u,\delta,z)\\
	&\lesssim L_1(n^{-1/3}T_n)^2+L_1^2n^{-1/3}T_n\lesssim L_1(n^{-1/3}T_n)^2.
	\end{split}
	\]
	As in the proof of Lemma \ref{le:R_1_2}, we get that $J_{[]}(1,\F,L_2(\p))$  is bounded.
\end{proof}
\begin{lemma}
	\label{le:R_3_2}
	Let $T_n=n^{\gamma}$ for some $\gamma\in (0,1/3)$ and $b={n^{-1/4}}$.
	Assume that, given $z$, the follow up time $T$ has a continuous density {uniformly bounded from above and below away from zero. 
		For $q\in[2,1/(\gamma+1/24))$, there exists $K>0$ such that, for $x\geq 0$ and all $a\in\R$,}
	\[
	\p_\theta\left[\left\{n^{1/3}T_n\sup_{|v-U_{\hat{\theta}}(a)|\leq b}\left|(\Phi_n -\Phi)(v;\beta_0)-(\Phi_n-\Phi )(U_{\hat{\theta}}(a);\beta_0)\right|>x\right\}\cap E_n\right]\leq K x^{-q}n^{1-q/3}.
	\] 
\end{lemma}
\begin{proof}
	Let $\F$ be the following class of functions on $[0,1]\times\{0,1\}\times\R^p$,
	\[
	\F=\left\{f_{v,s}(t,\delta,z)=  \1_{\left\{U(a)+s<t\leq U(a)+s+v\right\}}e^{\beta_0'z}\,\bigg|\,v\in[0,b],\,|s|\leq n^{-1/2+\alpha}\right\}.
	\]	
	with envelope function
	\[
	F(t,\delta,z)=\1_{\left\{U(a)-n^{-1/2+\alpha}<t\leq U(a)+n^{-1/2+\alpha}+b\right\}}e^{\beta_0'z}.
	\]
	Then,  {using that \eqref{eqn:par_est-true} holds on $E_n$,}
	\[
	\begin{split}
	&\p_\theta\left[\left\{n^{1/3}T_n\sup_{v\in[U_{\hat{\theta}}(a),U_{\hat{\theta}}(a)+b]}\left|(\Phi_n -\Phi)(v;\beta_0)-(\Phi_n-\Phi )(U_{\hat{\theta}}(a);\beta_0)\right|>x\right\}\cap E_n\right]\\
	&\leq \p\left[\sup_{f\in\mathcal{F}}\left|\int f(t,\delta,z)\,\dd\sqrt{n}(\p_n-\p)(t,\delta,z)\right|>cn^{1/6}T_n^{-1}x\right].
	\end{split}	
	\]
	The rest of the proof remains the same as in Lemma \ref{le:R_3} because
	\[
	\Vert F\Vert_{L_q(\p)}^q=\int_{\R^p}e^{q\beta_0' z}\int_{U(a)-n^{-1/2+\alpha}}^{U(a)+n^{-1/2+\alpha}+b}\,\dd H(t|z)\,\dd F_Z(z)\in[c_1b,c_2b]
	\]
	for some $c_1,c_2>0$ independent of $a$ because the density of $T$ given $z$ is uniformly bounded from above and below away from zero.	It remains to show that $J_{[]}(1,\F,L_2(\p))$ is bounded. We first compute $N_{[]}\left(\epsilon\Vert F\Vert_{L_2(\p)},\F,L_2(\p) \right)$. Divide the interval $[0,b]$ in $M$ subintervals of length $L_1=b/M$ and divide the interval $[-n^{-1/2+\alpha},n^{-1/2+\alpha}]$ in $M$ subintervals of length $L_2=n^{-1/2+\alpha}/M$. For $i,j=1,\dots,M$,  let
	\[
	l_{i,j}(t,\delta,z)= \1_{\left\{U(a)-n^{-1/2+\alpha}+jL_2<t\leq U(a)-n^{-1/2+\alpha}+(j-1)L_2+(i-1)L_1\right\}}e^{\beta_0'z},
	\]
	\[
	L_{i,j}(t,\delta,z)= \1_{\left\{U(a)-n^{-1/2+\alpha}+(j-1)L_2<t\leq U(a)-n^{-1/2+\alpha}+jL_2+iL_1\right\}}e^{\beta_0'z}.
	\]
	The brackets $[l_{i,j},L_{i,j}]$ for $i,j=1,\dots,M$, cover all the class $\F$. Moreover, we have
	\[
	\begin{split}
	\Vert L_{i,j}-l_{i,j}\Vert_{L_2(\p)}^2&\lesssim\int \1_{\{U(a)-n^{-1/2+\alpha}+(j-1)L_2<t\leq U(a)-n^{-1/2+\alpha}+jL_2\}}e^{2\beta_0'z}\,\dd\p(t,\delta,z)\\
	&\quad+\int \1_{\{U(a)-n^{-1/2+\alpha}+(j-1)L_2+(i-1)L_1<t\leq U(a)-n^{-1/2+\alpha}+jL_2+iL_1\}}e^{2\beta_0'z}\,\dd\p(t,\delta,z)\\
	&\lesssim L_1.
	\end{split}
	\]
	If we take 
	\[
	M=\left\lfloor\frac{Kb}{{\epsilon^2 \Vert F\Vert_{L_2(\p)}^2}}+1\right\rfloor,
	\]
	for a properly chosen constant $K>0$, then 
	\[
	\Vert L_{i,j}-l_{i,j}\Vert_{L_2(\p)}\leq \epsilon\Vert F\Vert_{L_2(\p)}.
	\]
	It follows that 
	\[
	N_{[]}\left(\epsilon\Vert F\Vert_{L_2(\p)},\F,L_2(\p) \right)\leq \left\lfloor\frac{Kb}{{\epsilon^2 \Vert F\Vert_{L_2(\p)}^2}}+1\right\rfloor
	\]
	and consequently, as in Lemma \ref{le:R_3}, we have 
	\[
	\begin{split}
	J_{[]}(1,\F,L_2(\p))&\leq \int_0^{1}\sqrt{1+{\log} \left\lfloor\frac{Kb}{{\epsilon^2 \Vert F\Vert_{L_2(\p)}^2}}+1\right\rfloor  }\,\dd\epsilon\leq K'.
	\end{split}
	\] 
\end{proof}

\begin{proof}[Proof of Theorem \ref{theo:CLT2}]
	Let 
	\[
	\mathcal{J}_n=n^{p/3}\int_\epsilon^M\left|\hat{\lambda}_n(t)-\lambda_{\hat{\theta}}(t)\right|^p\,\text{d}t.
	\]
	\textbf{Step 1.} 
	We first show that $\mathcal{J}_n\1_{E_n}=\tilde{\mathcal{J}}_n\1_{E_n}+o_P(n^{-1/6})$ where 
	\[
	\tilde{\mathcal{J}}_n=n^{p/3}\int_{\lambda_{\hat{\theta}}(M)}^{\lambda_{\hat{\theta}}(\epsilon)}\left|\hat{U}_n(a)-U_{\hat\theta}(a)\right|^p\left|U'_{\hat\theta}(a)\right|^{1-p}\,\text{d}a.
	\]
	Let 
	\[
	I_1=\int_\epsilon^M\left[\hat{\lambda}_n(t)-\lambda_{\hat{\theta}}(t)\right]_+^p\,\text{d}t\quad\text{and}\quad J_1=\int_\epsilon^M\int_0^{(\lambda_{\hat{\theta}}(\epsilon)-\lambda_{\hat{\theta}}(t))^p}\1_{\{\hat{\lambda}_n(t)\geq \lambda_{\hat{\theta}}(t)+a^{1/p}\}}\,\dd a\,\dd t.
	\]
	Since $\hat{\lambda}_n(t)<\lambda_{\hat{\theta}}(t)$ for $t>\hat{U}_n(\lambda_{\hat{\theta}}{(t)})$, we have
	\[
	\begin{split}
	0&\leq (I_1-J_1)\1_{E_n}=\1_{E_n}\int_\epsilon^{\hat{U}_n(\lambda_{\hat{\theta}}(\epsilon))}\int_{(\lambda_{\hat{\theta}}(\epsilon)-\lambda_{\hat{\theta}}(t))^p}^\infty \1_{\{\hat{\lambda}_n(t)\geq \lambda_{\hat{\theta}}(t)+a^{1/p}\}}\,\dd a\,\dd t\\
	&\leq \1_{E_n} \int_\epsilon^{\hat{U}_n(\lambda_{\hat{\theta}}(\epsilon))}\left[\hat{\lambda}_n(t)-\lambda_{\hat{\theta}}(t)\right]_+^p\,\dd t\\
	&\leq \1_{E_n}\int_\epsilon^{\epsilon+n^{-1/3}\log n} \left[\hat{\lambda}_n(t)-\lambda_{\hat{\theta}}(t)\right]_+^p\,\dd t+\left|\hat{\lambda}_n(\epsilon)-\lambda_{\hat{\theta}}(M)\right|^p \1_{\left\{n^{1/3}\left(\hat{U}_n(\lambda_{\hat{\theta}}(\epsilon))-\epsilon\right)>\log n \right\}\cap E_n}.
	\end{split}
	\]
	{Since \eqref{eqn:par_est-true} holds on $E_n$ and $\lambda_0$ is bounded, we also have
		$$\sup_{t\in[\epsilon,M]}\vert\lambda_{\hat\theta}(t)\vert\lesssim1$$
		on that event.
		Combining this with} Lemma \ref{le:boundaries}, \eqref{eqn:inv2} and \eqref{eqn:bound_expectation_L_p2}, it can be shown exactly as in step 1, proof of Theorem \ref{theo:CLT}, that
	\[
	(I_1-J_1)\1_{\E_n}=o_P\left(n^{-p/3-1/6}\right).
	\]
	With a change of variable $b=\lambda_{\hat{\theta}}(t)+a^{1/p}$ we get
	\[
	I_1\1_{E_n}=\1_{E_n}\int_{\lambda_{\hat{\theta}}(M)}^{\lambda_{\hat{\theta}}(\epsilon)}\int _{U_{\hat{\theta}}(b)}^{\hat{U}_n(b)}p(b-\lambda_{\hat{\theta}}(t))^{p-1}\1_{\{U_{\hat{\theta}}(b)<\hat{U}_n(b) \}}\,\text{d}t\,\dd b+o_P\left(n^{-p/3-1/6}\right).
	\]
	Note that {the function $\frac{\partial^2 f(\vartheta,t)}{\partial t^2}$ is continuous with respect to both $\vartheta$ and $t$. Hence it is uniformly bounded on $\mathcal{B}_{\theta,1}\times[\epsilon,M]$ where $\mathcal{B}_{\theta,1}$ is the ball in $\R^d$ centered in the true parameter $\theta$ with radius one. Since, on $E_n$ we have $|\hat{\theta}_n-\theta|\lesssim n^{-1/2+\alpha}$, it follows that  $\hat\theta_n\in\mathcal{B}_{\theta,1} $ so in particular, $\lambda''_{\hat\theta}(t)=\frac{\partial^2 f(\hat\theta_n,t)}{\partial t^2}$ is uniformly bounded with respect to $t$ and $n$ on $E_n$. This implies that} $\lambda'_{\hat\theta}$ satisfies \eqref{eqn:assumption_derivative}. 
	Then, by a Taylor expansion  we have
	\[
	\left|[b-\lambda_{\hat{\theta}}(t)]^{p-1}-\left[(U_{\hat{\theta}}(b)-t)\lambda'_{\hat{\theta}}\circ U_{\hat{\theta}}(b)\right]^{p-1}\right|\lesssim (t-U_{\hat{\theta}}(b))^{p-1+s}
	\] 
	for all $b\in(\lambda_{\hat{\theta}}(M),\lambda_{\hat{\theta}}(\epsilon))$ and $t\in(U_{\hat{\theta}}(b),M)$. 
	As in step 1, proof of Theorem \ref{theo:CLT}, it follows that
	\[
	I_1\1_{E_n}=\1_{E_n}\int_{\lambda_{\hat{\theta}}(M)}^{\lambda_{\hat{\theta}}(\epsilon)}\int _{U_{\hat{\theta}}(b)}^{\hat{U}_n(b)}p(t-U_{\hat{\theta}}(b))^{p-1}|\lambda'_{\hat{\theta}}(U_{\hat{\theta}}(b))|^{p-1}\1_{\{U_{\hat{\theta}}(b)<\hat{U}_n(b) \}}\,\text{d}t\,\dd b+o_P\left(n^{-p/3-1/6}\right).
	\]
	In the same way we can deal with 
	\[
	I_2=\int_\epsilon^M\left[\lambda_{\hat{\theta}}(t)-\hat{\lambda}_n(t)\right]_+^p\,\text{d}t.
	\]
	\textbf{Step 2.} The notation is the same as in step 2 of the proof of Theorem \ref{theo:CLT}. 
	We argue conditionally on $\hat\theta_n$ {and assume that on $E_n$,  $|\hat\theta_n-\theta|\leq n^{-1/2+\alpha}$ and  \eqref{eqn:par_est-true} holds where $\alpha\in(0,1/6)$ is fixed.} We show that 
	\begin{equation}
	\label{eqn:step2_2}
	\tilde{\mathcal{J}}_n\1_{E_n}=\int_{\epsilon+b}^{M-b}\left|\tilde{V}(t)-n^{-1/6}\eta_n(t) \right|^p\left|\frac{\lambda'_0(t)}{h^{uc}(t)}\right|^p\,\dd t+o_{P_\theta}(n^{-1/6}),
	\end{equation}
	for some $\eta_n$ such that $\sup_{t\in[0,1]}|\eta_n(t)|=O_{P_\theta}(1)$. 
	For every {$a\in[\lambda_{\hat\theta}(M)+b,\lambda_{\hat\theta}(\epsilon)-b]$}, let $\tilde{a}=\lambda_0(U_{\hat\theta}(a))$ and define
	\[
	\begin{split}
	a^\xi=\tilde{a}-n^{-1/2}\xi_n\frac{h^{uc}(U_{\hat\theta}(a))}{\Phi(U_{\hat\theta}(a);\beta_0)}+(\hat\beta_n-\beta_0)A'_0(U_{\hat\theta}(a))-\Psi_n^s(U_{\hat\theta}(a))\frac{\lambda_0(U_{\hat\theta}(a))}{\Phi(U_{\hat\theta}(a);\beta_0)}.
	\end{split}
	\]
	{Since on $E_n$ we have $\sup_{t\in[\epsilon,M]}|\lambda_{\hat\theta}(t)-\lambda_0(t)|\lesssim n^{-1/2+\alpha}$ with $\alpha<1/6$ and $b=n^{-1/4}$, it follows that \[
		a\in[\lambda_{0}(M)-n^{-1/2+\alpha}+b,\lambda_{0}(\epsilon)+n^{-1/2+\alpha}-b]\subseteq[\lambda_{0}(M),\lambda_{0}(\epsilon)]
		\]
		so $a$ is in the range of $\lambda_{0}$ and $a=\lambda_0(U(a))$. It follows that $|\tilde{a}-a|\lesssim n^{-1/2+\alpha}$ {on $E_n$} and therefore, 
	}
	we again have $\frac{\dd a^\xi}{\dd a}=1+o(1)$. {Note that $h^{uc}$ is differentiable since $h^{uc}(t)=\lambda_{0}(t)\Phi(t;\beta_{0})$.} 
	
	We can assume that on the event $E_n$ it holds
	\begin{equation}
	\label{eqn:E_n_add2}
	\sup_{a\in\R}\left|a-a^\xi\right|\leq cn^{-1/2+\alpha}.
	\end{equation}
	Note that \eqref{eqn:inv2} hold also if we replace $\p$ by $\p_\theta$ because $\hat{U}_n$ is independent of $\hat{\theta}_n$. {Note that if suffices to prove  \eqref{eqn:inv2} for $x>n^{-1/3}$ because otherwise the exponential bound is trivial and that on $E_n$, we have $\sup_a|U_{\hat\theta}(a)-U(a)|\lesssim n^{-1/2+\alpha}$ (see \eqref{eqn:par_est-true}) where $\alpha<1/6$. Hence, the second probability on the  second line of \eqref{eqn:inv2}, which is the probability of a non random event under $\p_\theta$, is equal to zero.} Then, from Lemma 6(i) in \cite{durot2007} and the change of variable $a\to a^{\xi}$ we obtain
	\[
	\begin{split}
	\tilde{\mathcal{J}}_n\1_{E_n}
	&=\1_{E_n}n^{p/3}\int_{J_n}\left|\frac{H^{uc}(\hat{U}_n(a^\xi))-H^{uc}(U_{\hat\theta}(a^\xi))}{h^{uc}(U_{\hat\theta}(a))}\right|^p\left|U'_{\hat\theta}(a)\right|^{1-p}\,\text{d}a+o_{P_\theta}(n^{-1/6}),
	\end{split}
	\]
	with  
	\[
	J_n=\left[\lambda_{\hat\theta}(M)+n^{-1/6}/\log n,\lambda_{\hat\theta}(\epsilon)-n^{-1/6}/\log n\right] .
	\]
	Let ${t}_n=(H^{uc})^{-1}(H^{uc}(U_{\hat\theta}(a))+n^{-1/3}u)$. 
	By definition of $\hat{U}_n$ and  properties of the $\argmax$ function it follows  that
	\[
	n^{1/3}\left(H^{uc}(\hat{U}_n(a^\xi))-H^{uc}(U_{\hat\theta}(a))\right)=\argmax_{u\in I_n(a)}\left\{D_n(a,u)+S_n(a,u) \right\},
	\]
	where
	\[
	I_n(a)=\left[-n^{1/3}\left(H^{uc}(U_{\hat\theta}(a))-H^{uc}(\epsilon)\right),n^{1/3}\left(H^{uc}(M)-H^{uc}(U_{\hat\theta}(a))\right)\right],
	\]
	\[
	D_n(a,u)=n^{2/3}\Phi(U_{\hat\theta}(a);\beta_0)\left\{\left(\Lambda_0(t_n)-\tilde{a}t_n\right)-\left(\Lambda_0(U_{\hat\theta}(a))-\tilde{a}U_{\hat\theta}(a)\right) \right\}
	\]
	and
	\begin{equation}
	\label{def:S_n2}
	S_n(a,u)=n^{2/3}\Phi(U_{\hat\theta}(a);\beta_0)\left\{ (\tilde{a}-a^\xi)\left[t_n-U_{\hat\theta}(a)\right]+(\Lambda_n-\Lambda_0)\left(t_n\right)-(\Lambda_n-\Lambda_0)\left(U_{\hat\theta}(a)\right)\right\}.
	\end{equation}
	Next, we follow the proof of Theorem \ref{theo:CLT} with $U(a)$ replaced by $U_{\hat\theta}(a)$. From Lemmas \ref{le:R_1_2} and \ref{le:R_2_2}, it follows that Lemma \ref{le:Lambda_n-Lambda} holds also if even $t$ is allowed to depend on $n$ (here $t=U_{\hat\theta}(a)$). Hence, we can write
	\[
	S_n(a,u)=W_{U_{\hat\theta}(a)}(u)+R^3_n(a,u)+R^4_n(a,u)
	\]
	where
	\[
	R^3_n(a,u)=n^{2/3}\left\{\Psi_n^s\left(U_{\hat\theta}(a)\right)-\left[\Phi_n\left(U_{\hat\theta}(a);\beta_0\right)-\Phi\left(U_{\hat\theta}(a);\beta_0\right)\right]\right\}\lambda_0\left(U_{\hat\theta}(a)\right)\left[t_n-U_{\hat\theta}(a)\right]
	\] 	
	and
	\[
	R^4_n(a,u)=n^{2/3}\Phi\left(U_{\hat\theta}(a);\beta_0\right)r_n\left(U_{\hat\theta}(a),t_n\right)
	\]
	with $r_n$ as in Lemma \ref{le:Lambda_n-Lambda}. Let $R_n(a,u)=R^3_n(a,u)+R^4_n(a,u)$. We use Lemma 5 in \cite{durot2007} to show that $R_n$ is negligible.
	First we  localize. Let $q>12$,  $T_n=n^{1/(3(6q-11))}$ and define
	\[
	\tilde{U}_n=\argmax_{|u|\leq T_n}\left\{D_n(a,u)+W_{U_{\hat\theta}(a)}(u)+R_n(a,u) \right\}.
	\]
	As in step 2, proof of Theorem \ref{theo:CLT}, since $h^{uc}$ and $U'_{\hat\theta}$ are bounded, using \eqref{eqn:inv2} with $\p$ replaced by $\p_\theta$, we obtain 
	\[
	\p_\theta\left(\left\{\left|H^{uc}(\hat{U}_n(a^\xi))-H^{uc}(U_{\hat\theta}(a))\right|>x \right\}\cap E_n\right)\leq K_1\exp(-K_2nx^3)
	\]
	and
	\[
	\sup_{a\in\R}\E_\theta\left[\1_{E_n}n^{q'/3}\left|H^{uc}(\hat{U}_n(a^\xi))-H^{uc}(U_{\hat\theta}(a))\right|^{q'}\right]\leq K,
	\]
	for every $q'\geq 1$. It follows that 
	\[
	\tilde{\mathcal{J}}_n\1_{E_n}=\1_{E_n}\int_{J_n}\left|\tilde{U}_n(a)-n^{-1/6}\eta_n(a) \right|^p\frac{|U'_{\hat\theta}(a)|^{1-p}}{h^{uc}(U_{\hat\theta}(a))^p}\,\dd a+o_{P_\theta}(n^{-1/6})
	\]
	where
	\[
	\eta_n(a)=n^{1/2}(a-a^\xi)|U'_{\hat\theta}(a)|h^{uc}(U_{\hat\theta}(a))=O_{P_\theta}(1).
	\]
	Now we approximate $\tilde{U}_n$ by 
	\[
	\tilde{\tilde{U}}_n(a)=\argmax_{|u|\leq \log n}\left\{D_n(a,u)+W_{U_{\hat\theta}(a)}(u) \right\}.
	\]	
	By Taylor's expansion,
	\[
	\left|\frac{\partial}{\partial u}D_n(a,u) \right|\leq K|u|,\quad\text{and}\quad D_n(a,u)\leq -cu^2,
	\]
	for some positive constants $K$, $c$, $a\in J_n$ and $|u|\leq T_n$. Moreover, given $\hat\theta_n$, $W_{U_{\hat\theta}}$ is a standard Brownian motion.  It remains to show that
	\begin{equation}
	\label{eqn:bound_R_n2}
	\p_\theta\left[\left\{\sup_{|u|\leq T_n}|R_n(a,u)|>x \right\}\cap E_n\right]\lesssim x^{-q} n^{1-q/3},\quad\text{for all }x\in(0,n^{2/3}],
	\end{equation}
	then, as in Lemma 5(i) in \cite{durot2007}, it follows that 
	\[
	\E_\theta\left[\1_{E_n}\left|\tilde{U}_n-\tilde{\tilde{U}}_n\right|^r\right]\lesssim \left(n^{-1/6}/\log n\right)^r
	\]
	and by Lemma 6(i) in \cite{durot2007}, $\tilde{U}_n$ can be replaced by $\tilde{\tilde{U}}_n$.
	Again, for \eqref{eqn:bound_R_n2} it is sufficient to consider $x>n^{-1/3+1/q}$.
	From Lemma \ref{le:Lambda_n-Lambda} (with both $t$ and $t_n$ depending on $n$) we have 
	\[
	\p_\theta\left[\left\{\sup_{|u|\leq T_n}|R^4_n(a,u)|>x/2\right\}\cap E_n\right]\lesssim x^{-q} n^{1-q/3}.
	\]
	It remains to consider $R^3_n$. Note that, since $a\in J_n $, it follows that $b<U_{\hat\theta}(a)<1-b$. From Lemma \ref{le:R_3_2} it follows  that
	\[
	\begin{split}
	&\p_\theta\left[\left\{\sup_{|u|\leq T_n}|R^3_n(a,u)|>x/4 \right\}\cap E_n\right]\\
	&\leq\p_\theta\left[n^{1/3}T_n\sup_{|v-U_{\hat\theta}(a)|\leq b}\left|\Phi_n(v;\beta_0)-\Phi(v;\beta_0)-\Phi_n(U_{\hat\theta}(a);\beta_0)+\Phi(U_{\hat\theta}(a);\beta_0)\right|>cx \right]\\
	&\lesssim x^{-q} n^{1-q/3}.
	\end{split}
	\]
	To conclude, from  Lemma 5(i) in \cite{durot2007}, it follows that $\tilde{U}$ can be replaced by $\tilde{\tilde{U}}_n$.
	
	It remains to replace $\tilde{\tilde{U}}_n(a)$ by $\tilde{V}(U_{\hat\theta}(a))$. By a Taylor expansion and \eqref{eqn:assumption_derivative} we have
	\[
	\sup_{|u|\leq \log n}\left|D_n(a,u)-d(U_{\hat\theta}(a))u^2\right|\lesssim n^{-s/3}(\log n)^3
	\]
	{on $E_n$.}
	Since $\tilde{\tilde{U}}_n(a)$, $\tilde{V}(U_{\hat\theta}(a))$ and $\1_{E_n}n^{-1/6}\eta_n$ have bounded moments of any order, by Lemmas 5(ii) and 6(i) in \cite{durot2007} it follows that 
	\[
	\begin{split}
	\tilde{\mathcal{J}}_n\1_{E_n}&=\1_{E_n}\int_{J_n}\left|\tilde{V}(U_{\hat\theta}(a))-n^{-1/6}\eta_n(a) \right|^p\frac{|U'_{\hat\theta}(a)|^{1-p}}{h^{uc}(U_{\hat\theta}(a))^p}\,\dd a+o_{P_\theta}(n^{-1/6})\\
	&=\1_{E_n}\int_{\epsilon+b}^{M-b}\left|\tilde{V}(t)-n^{-1/6}\eta_n(t) \right|^p\left|\frac{\lambda'_{\hat\theta}(t)}{h^{uc}(t)}\right|^{p}\,\dd t+o_{P_\theta}(n^{-1/6}),
	\end{split}
	\]
	where
	\begin{equation}
	\label{def:eta_t2}
	\eta_n(t)=\frac{h^{uc}(t)}{|\lambda'_0(t)|}\left\{\xi_n\lambda_0(t)-n^{1/2}(\hat\beta_n-\beta_0)A'_0(t)+n^{1/2}\Psi_n^s(t)\frac{\lambda_0(t)}{\Phi(t;\beta_0)} -n^{1/2}\left(\lambda_{\hat\theta}(t)-\lambda_0(t)\right)\right\}.
	\end{equation}
	Then, \eqref{eqn:step2_2} follows from \eqref{eqn:par_est-true} and the fact that $\p_\theta
	(E_n)\to 1$.
	
	\textbf{Step 3.} Now we prove that $\eta_n$ can be removed from the integrand in \eqref{eqn:step2_2}. Let 
	\[
	\mathcal{D}_n=n^{1/6}\left\{\int_{\epsilon+b}^{M-b}\left|\tilde{V}(t) \right|^p\left|\frac{\lambda'_0(t)}{h^{uc}(t)}\right|^{p}\,\dd t-\int_{\epsilon+b}^{M-b}\left|\tilde{V}(t)-n^{-1/6}\eta_n(t) \right|^p\left|\frac{\lambda'_0(t)}{h^{uc}(t)}\right|^{p}\,\dd t \right\}.
	\]
	As in the proof of Theorem \ref{theo:CLT}, we obtain
	\[
	\mathcal{D}_n=p\int_{\epsilon+b}^{M-b} \eta_n(t)\tilde{V}(t)\left|\tilde{V}(t) \right|^{p-2}   \left|\frac{\lambda'_0(t)}{h^{uc}(t)}\right|^{p}\,\dd t+o_{P_\theta}(1).
	\]
	Using the definition of $\eta_n$ in \eqref{def:eta_t} we have
	\begin{equation}
	\label{eqn:integral_eta2}
	\begin{split}
	&\int_{\epsilon+b}^{M-b} \eta_n(t)\tilde{V}(t)\left|\tilde{V}(t) \right|^{p-2}   \left|\frac{\lambda'_0(t)}{h^{uc}(t)}\right|^{p}\,\dd t\\
	&=\xi_n\int_{\epsilon}^M \frac{\lambda_0(t)h^{uc}(t)}{|\lambda'_0(t)|}\tilde{V}(t)\left|\tilde{V}(t) \right|^{p-2}   \left|\frac{\lambda'_0(t)}{h^{uc}(t)}\right|^{p}\,\dd t\\
	&\quad-n^{1/2}(\hat\beta_n-\beta_0)\int_{\epsilon}^M \frac{A'_0(t)h^{uc}(t)}{|\lambda'_0(t)|}\tilde{V}(t)\left|\tilde{V}(t) \right|^{p-2}   \left|\frac{\lambda'_0(t)}{h^{uc}(t)}\right|^{p}\,\dd t\\
	&\quad+n^{1/2}\int_{\epsilon+b}^{M-b}\Psi_n^s(t) \tilde{V}(t)\left|\tilde{V}(t) \right|^{p-2} \frac{\lambda_0(t)^2}{|\lambda'_0(t)|}  \left|\frac{\lambda'_0(t)}{h^{uc}(t)}\right|^{p}\,\dd t\\
	&\quad-n^{1/2}\int_{\epsilon+b}^{M-b}\left(\lambda_{\hat{\theta}}(t)-\lambda_0(t)\right) \tilde{V}(t)\left|\tilde{V}(t) \right|^{p-2} \frac{h^{uc}(t)}{|\lambda'_0(t)|}  \left|\frac{\lambda'_0(t)}{h^{uc}(t)}\right|^{p}\,\dd t+o_{P_\theta}(1).
	\end{split}
	\end{equation}
	As in the proof of Theorem \ref{theo:CLT}, the first three terms on the right hand side of \eqref{eqn:integral_eta2} converge to zero. It remains to deal with the fourth term. Let $a_i=(\epsilon+b+iL)\wedge(M-b)$ for $i=0,\dots,M$ where $M=\lfloor (M-\epsilon-2b)/L\rfloor+1$ and $L\to 0$, $L>n^{-1/2}$. Then we can deal with this term in the same way as we did for the third term in the proof of Theorem \ref{theo:CLT}, replacing $\Psi_n^s(t)$ and $b^2$ by $\lambda_{\hat\theta}(t)-\lambda_0(t)$ and $L$ respectively. In this case, instead of \eqref{eqn:diff_psi}, we use
	\begin{equation}
	\label{eqn:diff_psi2}
	\begin{split}
	&\sup_{t\in[a_i,a_{i+1}]}\left|\left(\lambda_{\hat{\theta}}(t)-\lambda_0(t)\right)- \left(\lambda_{\hat{\theta}}(a_i)-\lambda_0(a_i)\right)\right|\\
	&= \sup_{t\in[a_i,a_{i+1}]}\left|(\hat\theta_n-\theta)\left(\frac{\partial}{\partial \theta}f(\theta,a_i)-\frac{\partial}{\partial \theta}f(\theta,t)\right)-(\hat\theta_n-\theta)^2 \left(\frac{\partial^2}{\partial \theta^2}f(\theta,a_i)\bigg|_{\theta^*_1}-\frac{\partial^2}{\partial \theta^2}f(\theta,t)\bigg|_{\theta^*_2}\right)\right|\\
	&=O_P(Ln^{-1/2}).
	\end{split}
	\end{equation}
	To conclude we have
	\[
	\tilde{\mathcal{J}}_n=\int_{\epsilon}^M\left|\tilde{V}(t) \right|^p\left|\frac{\lambda'_0(t)}{h^{uc}(t)}\right|^{p}\,\dd t+o_{P_\theta}(n^{-1/6}).
	\]
	\textbf{Step 4.} Now this is the same as $\mathcal{J}_n$ at the end of Step 3 of the proof of Theorem \ref{theo:CLT}. Hence, we have that, given {$\hat\theta_n$} 
	\[
	n^{1/6}\left\{\tilde{\mathcal{J}}_n\1_{E_n}-m_p\right\}\xrightarrow{d}N(0,\sigma^2_p).
	\]
	As a result, we obtain
	\[
	n^{1/6}\left\{\tilde{\mathcal{J}}_n\1_{E_n}-m_p\right\}\xrightarrow{d}N(0,\sigma^2_p)
	\]
	and the statement of the theorem follows from $\mathcal{J}_n\1_{E_n}=\tilde{\mathcal{J}}_n\1_{E_n}+o_P(n^{-1/6})$ and  the fact that $\p(E_n)\to 1$.
\end{proof}

\section{CLT for the bootstrap version}
\label{sec:supp_bootstrap}
Throughout this section we assume that the assumptions of {Theorem \ref{theo:CLT*}} hold and we do not state them again for the separate lemmas.
\begin{proof}[Proof of Lemma \ref{le:Breslow*}.]
	By definition and the form of the likelihood in the Cox model (see for example Section 2.1 in \cite{LopuhaaNane2013}), we have
	\begin{eqnarray}
	\label{eqn:expression_Breslow*}\notag
	&&\int \delta\1_{\{t\leq x\}}\frac{1}{\Phi^*(t;\tilde\beta_n)}\,\mathrm{d}P_n^*(t,\delta,z)\\
	&\notag=&\frac{1}{n}\sum_{i=1}^n\int \delta\1_{\{t\leq x\}}\frac{1}{\Phi^*(t;\tilde\beta_n)}\,\mathrm{d}P_n^*(t,\delta|z=Z_i)\\
	&\notag=&\frac{1}{n}\sum_{i=1}^n\int \1_{\{t\leq x\}}\frac{1}{\Phi^*(t;\tilde\beta_n)}f_{\hat\theta}(t|Z_i)[1-\hat G_n(t)]\,\mathrm{d}t\\
	&=&\frac{1}{n}\sum_{i=1}^n\int\1_{\{t\leq x\}}\frac{1}{\Phi^*(t;\tilde\beta_n)}\lambda_{\hat\theta}(t)e^{\hat\beta'_nZ_i}[1- F_{\hat\theta}(t|Z_i)][1-\hat G_n(t)]\,\mathrm{d}t,
	\end{eqnarray}
	where $P_n^*(t,\delta|z=Z_i)$ is the distribution of $(T^*,\Delta^*)$ given that $Z^*=Z_i$.
	Since we fix the covariates when generating bootstrap samples, $Z^*$ can only take values $Z_1,\dots, Z_n$, each with probability $1/n$.
	Note that here $f_{\hat\theta}$ and $F_{\hat\theta}$ are the  density and the cdf of the event times in the bootstrap sample.
	Moreover, 
	\begin{eqnarray}
	\label{eqn:expression_Phi*}\notag
	\Phi^*(t;\tilde\beta_n)&=&\int  \1_{\{u\geq t\}}e^{\tilde\beta'_nz}\,\mathrm{d}P_n^*(u,\delta,z)\\
	&\notag=&\frac{1}{n}\sum_{i=1}^ne^{\tilde\beta'_nZ_i}\int  \1_{\{u\geq t\}}\,\mathrm{d}P_n^*(u,\delta|z=Z_i)\\
	&\notag=&\frac{1}{n}\sum_{i=1}^ne^{\tilde\beta'_nZ_i}P_n^*(T^*\geq t|Z^*=Z_i)\\
	&=&\frac{1}{n}\sum_{i=1}^ne^{\tilde\beta'_nZ_i}[1- F_{\hat\theta}(t|Z_i)][1-\hat G_n(t)].
	\end{eqnarray}
	If we replace the expression for $\Phi^*(t;\tilde\beta_n)$ in the wright hand side of equation \eqref{eqn:expression_Breslow*} we obtain 
	\[
	\begin{split}
	&\int \delta\1_{\{t\leq x\}}\frac{1}{\Phi^*(t;\tilde\beta_n)}\,\mathrm{d}P_n^*(t,\delta,z)\\
	&=\frac{1}{n}\sum_{i=1}^n\int\1_{\{t\leq x\}}\frac{1}{\frac{1}{n}\sum_{j=1}^ne^{\tilde\beta'_nZ_j}[1- F_{\hat\theta}(t|Z_j)][1-\hat G_n(t)]}\lambda_{\hat\theta}(t)e^{\hat\beta'_nZ_i}[1- F_{\hat\theta}(t|Z_i)][1-\hat G_n(t)]\,\mathrm{d}t\\
	&=\int_0^x\frac{\frac{1}{n}\sum_{i=1}^ne^{\hat\beta'_nZ_i}[1- F_{\hat\theta}(t|Z_i)]}{\frac{1}{n}\sum_{j=1}^ne^{\tilde\beta'_nZ_j}[1- F_{\hat\theta}(t|Z_j)]}\lambda_{\hat\theta}(t)\,\mathrm{d}t\\
	&=\int_0^x\lambda_{\hat\theta}(t)\,\mathrm{d}t=\Lambda_{\hat\theta}(x).
	\end{split}
	\]
	{Hence we obtain the statement of the Lemma.}
\end{proof}
\begin{lemma}
	\label{le:Phi*}
	{Under the assumptions of Theorem \ref{theo:CLT*},} for any $\epsilon>0$ and $M<\tau_H$, given the data $(T_1,\Delta_1,Z_1),\dots,(T_n,\Delta_n,Z_n)$, we have
	\begin{itemize}
		\item[(i)] 
		\[
		0<\liminf_{n\to\infty}\inf_{x\leq M}\inf_{|\beta-\tilde\beta_n|\leq \epsilon}|\Phi^*(x;\beta)|\leq\limsup_{n\to\infty}\sup_{x\leq M}\sup_{|\beta-\tilde\beta_n|\leq \epsilon}|\Phi^*(x;\beta)|<\infty.
		\]
		\item[(ii)] For any sequence $\beta_n^*$ such that $|\beta_n^*-\tilde\beta_n|$ converges to zero almost surely (with respect to the conditional probability given the data),
		\[
		0<\liminf_{n\to\infty}\inf_{x\leq M}|\Phi^*_n(x;\beta_n^*)|\leq\limsup_{n\to\infty}\sup_{x\leq M}|\Phi^*_n(x;\beta_n^*)|<\infty,
		\]
		with probability one (given the data).
		\item[(iii)] For $D^{(1),*}$ defined in \eqref{def:A_0*}, it holds
		\[
		\limsup_{n\to\infty}\sup_{x\in\R}\sup_{|\beta-\tilde\beta_n|\leq \epsilon}\left|D^{(1),*}(x;\beta)\right|<\infty.
		\]
		\item[(iv)] For any sequence $\beta_n^*$ such that $|\beta_n^*-\tilde\beta_n|$ converges to zero almost surely, (with respect to the conditional probability given the data), and $D^{(1),*}_n$ defined in \eqref{def:D_n^*} 
		\[
		\limsup_{n\to\infty}\sup_{x\in\R}\left|D^{(1),*}_n(x;\beta^*_n)\right|<\infty,
		\]
		with probability one (given the data).
		\item[(v)] \[
		\sqrt{n}\sup_{x\in\R}\left|\Phi_n^*(x;\hat\beta_n^*)-\Phi^*(x;\tilde\beta_n)\right|=O_{P^*}(1).
		\]
	\end{itemize}
\end{lemma}
\begin{proof}
	Note that 
	\[
	\Phi^*(x;\beta)=\frac{1}{n}\sum_{1}^n e^{\beta'Z_i}P^*_n(T\geq x|Z=Z_i)
	\]
	and the $Z_i$s are known (given the data).  Then we have 
	\[
	\begin{split}
	\inf_{x\leq M}\inf_{|\beta-\tilde\beta_n|\leq \epsilon}\Phi^*(x;\beta)&>\inf_{|\beta-\tilde\beta_n|\leq \epsilon}\frac{1}{n}\sum_{1}^n e^{\beta'Z_i}P^*_n(T\geq M|Z=Z_i)\\
	&=[1-\hat{G}_n(M)]\inf_{|\beta-\tilde\beta_n|\leq \epsilon}\frac{1}{n}\sum_{1}^n e^{\beta'Z_i}[1-F_{\hat\theta}(M|Z_i)]\\
	&=[1-G(M)]\inf_{|\beta-\tilde\beta_n|\leq \epsilon}\frac{1}{n}\sum_{1}^n e^{\beta'Z_i}[1-F(M|Z_i)]+o_P(1).
	\end{split}
	\]
	Moreover, since 
	\[
	\begin{split}
	&[1-G(M)]\inf_{|\beta-\tilde\beta_n|\leq \epsilon}\frac{1}{n}\sum_{1}^n e^{\beta'Z_i}[1-F(M|Z_i)]\\
	&\to [1-G(M)]\inf_{|\beta-\tilde\beta_n|\leq \epsilon}\E\left[ e^{\beta'Z}[1-F(M|Z)]\right]\\
	&=\inf_{|\beta-\beta_0|\leq 2\epsilon}\Phi(M;\beta)>0, 
	\end{split}
	\]
	the first inequality of (i) follows immediately. For the second inequality we have
	\[
	\sup_{x\leq M}\sup_{|\beta-\tilde\beta_n|\leq \epsilon}|\Phi^*(x;\beta)|\leq \sup_{|\beta-\tilde\beta_n|\leq \epsilon}\frac{1}{n}\sum_{1}^n e^{\beta'Z_i}\to \sup_{|\beta-\beta_0|\leq 2\epsilon}\E[Z]<\infty. 
	\]
	Similarly, statements (ii), (iii), (iv) can be proved following the reasoning in Lemma 3 in \cite{LopuhaaNane2013}. Next, we consider (v). By the triangular inequality we write
	\begin{equation}
	\label{eqn:Phi_triangular*}
	\sqrt{n}\sup_{x\in\R}\left|\Phi_n^*(x;\hat\beta_n^*)-\Phi^*(x;\tilde\beta_n)\right|\leq \sqrt{n}\sup_{x\in\R}\left|\Phi_n^*(x;\hat\beta_n^*)-\Phi^*_n(x;\tilde\beta_n)\right|+\sqrt{n}\sup_{x\in\R}\left|\Phi_n^*(x;\tilde\beta_n)-\Phi^*(x;\tilde\beta_n)\right|.
	\end{equation}
	The Taylor expansion gives us
	\[
	\left|\Phi_n^*(x;\hat\beta_n^*)-\Phi^*_n(x;\tilde\beta_n)\right|\leq |\hat\beta_n^*-\tilde\beta_n|\sup_{x\in[0,M]}|D_n^{(1),*}(\beta^*,x)|=O_{P^*}(n^{-1/2}).
	\]
	For the second term in the right hand side of~\eqref{eqn:Phi_triangular*}, note that the class of functions
	\[
	h_n(t,z;x)=\1_{\{t\geq x\}}e^{\hat\beta'_nz},\qquad x\in[0,M],
	\]
	is uniformly Donsker over the probabilities $P_n^*$ (satisfies Theorem 2.8.4 in \cite{VW96}). Hence 
	\[
	\sqrt{n}\sup_{x\in\R}\left|\Phi^*_n(x;\tilde\beta_n)-\Phi^*(x;\tilde\beta_n)\right|=O_{P^*}(1).
	\]	
\end{proof}
In what follows, $\p_T^*$ and $\E^*_T$ denote the conditional probability and conditional expectation given the data and $T_1^*,\dots,T_n^*$.
\begin{lemma}
	\label{le:S1*}
	Let $t\mapsto r(t)$ be a positive bounded function on $[\epsilon,M]$ with $R=\sup_{u\in[0,1]}r(t)$  and consider
	\[
	S_n^{1,*}(t)=\frac{1}{n}\sum_{i=1}^n \left(\Delta_i^*-\E_T^*[\Delta_i^*]\right)\1_{\{T_i^*\leq t\}}r(T_i^*).
	\] 
	{Under the assumptions of Theorem \ref{theo:CLT*},} there exists a positive constant $K$ such that for each $t\in[\epsilon,M]$, $t_n\in(0,M)$, and $x>0$ it holds
	\[
	P^*_n\left(\sup_{\substack{|s-t|<t_n\\s\in[\epsilon,M]}}\left|S_n^{1,*}(t)-S_n^{1,*}(s) \right|\geq x\right)\leq 4\exp\left(-K\tilde\theta nx\right)
	\]
	where
	\begin{equation}
	\label{def:theta_tilde*}
	\tilde\theta=\min\left\{1,\frac{4x}{3C(t_n+n^{-1/2+\alpha})e^{R^2/2}}\right\},
	\end{equation}
	for an appropriately chosen $C$. 
\end{lemma}
\begin{proof}
	It can be proved arguing as in Lemma \ref{le:S1}. The only change is that now $\1_{\{t<T^*\leq t+t_n\}}$ is a Bernoulli random variable with success probability 
	\[
	\begin{split}
	\gamma&=H^*(t+t_n)-H^*(t)\\
	&\leq\tilde{H}^*(t+t_n)-\tilde{H}^*(t)+2\sup_{t\in[\epsilon,M]}|\tilde{H}^*(t)-H^*(t)|\\
	&\lesssim t_n+n^{-1/2+\alpha},
	\end{split}
	\]
	where $H^*$ is the distribution function of $T^*$ and $\tilde{H}^*$ is defined in \eqref{def:tilde_H*}. For the last inequality we used 
	\begin{equation}
	\label{eqn:tilde_h*}
	\begin{split}
	\frac{\dd }{\dd t}\tilde{H}^*(t)&=g(t)\frac{1}{n}\sum_{i=1}^n\left[1-F_{\hat\theta}(t|Z_i)\right]+[1-G(t)]\frac{1}{n}\sum_{i=1}^n f_{\hat{\theta}}(t|Z_i)\\
	&\to g(t)\E\left[1-F(t|Z)\right]+[1-G(t)]\E\left[ f(t|Z)\right]\\
	&=\int_{\R^p}\left\{g(t)\left[1-F(t|z)\right]+[1-G(t)]f(t|z)\right\}\,\dd F_Z(z)=h(t),
	\end{split}
	\end{equation}
	i.e. $\frac{\dd }{\dd t}\tilde{H}^*(t)$ is uniformly bounded. As a result, $Ct_n$ in Lemma \ref{le:S1} will here be replaced by $C(t_n+n^{-1/2+\alpha})$. Note that the constant $C$ is not the same.
\end{proof}
\begin{lemma}
	\label{le:S2*}
	Let $t\mapsto m(t)$ be a positive continuously differentiable function on $[\epsilon,M]$ and consider
	\[
	S_n^{2,*}(t)=\frac{1}{n}\sum_{i=1}^n \left(m(T_i^*)\1_{\{T_i^*\leq t\}}-\E^*\left[m(T_i^*)\1_{\{T_i^*\leq t\}}\right]\right).
	\] 
	{Under the assumptions of Theorem \ref{theo:CLT*},} there exist two positive constants $K_1,\,K_2$ such that for each $t\in[\epsilon,M]$, $t_n\in(0,M)$ and $x>n^{-1}\log n$   it holds {on $E_n$}
	\[
	P^*_n\left(\sup_{\substack{|s-t|<t_n\\ s\in[\epsilon,M]}}\left|S_n^{2,*}(t)-S_n^{2,*}(s) \right|\geq x\right)\leq K_1\exp\left(-K_2nx\min\left\{1,xt_n^{-1}\right\}\right).
	\]
\end{lemma}
\begin{proof}
	We write
	\[
	S_n^{2,*}(t)-S_n^{2,*}(s)=\int_t^s m(v)\,\dd \left[H^*_n(v)-H^*(v)\right].
	\]
	Hence, {using integration by parts in generalized form with right continuous functions of bounded variation and the fact that the derivative of $m$ is bounded} we get 
	\[
	\left|S_n^{2,*}(t)-S_n^{2,*}(s) \right|\lesssim \left[H^*(u)-H^*_n(u)\right]\bigg|_t^s+t_n\sup_{u\in[\epsilon,M]}\left|H^*(u)-H^*_n(u)\right|
	\]
	and as a result
	{\small	\begin{equation*}
		\begin{split}
		P^*_n\left(\sup_{\substack{s\in[t,t+t_n]\\ s\leq M}}\left|S_n^{2,*}(t)-S_n^{2,*}(s) \right|\geq x\right)
		&\leq P^*_n\left(\sup_{u\in[\epsilon,M]}\left|H^*(u)-H^*_n(u)\right| \geq cxt_n^{-1}\right)\\
		&\quad +P^*_n\left(\sup_{\substack{s\in[t,t+t_n]\\ s\leq M}}\left|H^*(s)-H^*_n(s)-H^*(t)+H^*_n(t)\right|\geq cx\right).
		\end{split}	
		\end{equation*}}
	By the DKW inequality {the first probability on the right hand side satisfies}
	\[
	P^*_n\left(\sup_{u\in[\epsilon,M]}\left|H^*(u)-H^*_n(u)\right| \geq cxt_n^{-1}\right)\lesssim \exp\left(-cnx^2t_n^{-2}\right)\lesssim \exp\left(-cnx^2t_n^{-1}\right)
	\]
	From \eqref{eqn:tilde_h*} we have that  $C=\sup_u\frac{\dd}{\dd u}\tilde{H}^*(u)<\infty$ and is independent of $n$.  Assume first that $t_n\leq c_0x$ where $c_0=c/(2C)$.
	By monotonicity of both $H_n^*$ and $H^*$, for $s\in[t,t+t_n]$ we have
	\begin{eqnarray*}
		(H_n^*-H^*)(s)-(H_n^*-H^*)(t)&\leq& H_n^*(t+t_n)-H^*(t)-(H_n^*-H^*)(t)\\
		&\leq& (H_n^*-H^*)(t+t_n)-(H_n^*-H^*)(t)+H^*(s)-H^*(t)
	\end{eqnarray*} 		
	and
	\[
	\begin{split}
	H^*(s)-H^*(t)&\leq\tilde{H}^*(s)-\tilde{H}^*(t)+|(\tilde{H}^*-H^*)(s)-(\tilde{H}^*-H^*)(t)|\\
	&\leq C (s-t)+|(\tilde{H}^*-H^*)(s)-(\tilde{H}^*-H^*)(t)|.
	\end{split}
	\]
	Moreover, using the Brownian bridge approximation of $G_n-G$ we obtain
	\[
	\begin{split}
	&\sup_{\substack{s\in[t,t+t_n]\\ s\leq M}}|(\tilde{H}^*-H^*)(s)-(\tilde{H}^*-H^*)(t)|\\
	&\leq \sup_{\substack{s\in[t,t+t_n]\\ s\leq M}}|(G_n-G)(s)-(G_n-G)(t)|+ct_n\sup_{u\in[\epsilon,M]}|G_n(u)-G(u)|\\
	&=O_P(n^{-1/2}\sqrt{t_n}).
	\end{split}
	\]
	Therefore we can assume that on $E_n$ we have 
	\[
	\sup_{\substack{s\in[t,t+t_n]\\ s\leq M}}|(\tilde{H}^*-H^*)(s)-(\tilde{H}^*-H^*)(t)|\leq \frac{c}{4\sqrt{c_0}}\sqrt{t_n}n^{-1/2}(\log n)^{1/2}.
	\]
	Since $x\geq n^{-1}(\log n)$ and $t_n\leq c_0x$
	\[
	|(H_n^*-H^*)(s)-(H_n^*-H^*)(t)|\leq| (H_n^*-H^*)(t+t_n)-(H_n^*-H^*)(t)|
	+\frac34cx.
	\]
	Therefore, since $t_n\leq c_0x$ where $c_0=c/(2C)$ we obtain
	\begin{eqnarray*}
		&&P^*_n\left(\sup_{\substack{s\in[t,t+t_n]\\ s\leq M}}\left|(H_n^*-H^*)(s)-(H_n^*-H^*)(t)\right|\geq cx\right)\\
		&&\leq  P_n^*\left(\left|(H_n^*-H^*)(t+t_n)-(H_n^*-H^*)(t)\right|\geq cx/4\right).
	\end{eqnarray*}
	Using Lemma \ref{le:binomial} with $p=H^*(t_n)-H^*(t)\leq cx/2$ we have
	\[
	P_n^*\left(\left|(H_n^*-H^*)(t+t_n)-(H_n^*-H^*)(t)\right|\geq cx/4\right)
	\leq 2\exp\left(-\frac{nc^2x^2}{32p+16cx}\right)\leq 2\exp\left(-Knx\right).
	\] 
	It remains to consider the case $t_n>c_0 x$. We use the same type of argument as in the proof of Lemma \ref{le:S2}. Here we have 
	\[
	\mathcal{M}_n^*(s)=\frac{H_n^*(s)-H_n^*(t)}{H^*(s)-H^*(t)},\qquad s\in[t,1]
	\]
	which, given the data,  is a reverse time martingale conditionally on $H^*_n(t)$. Let $c_1=\inf_t \frac{\dd \tilde{H}^*(t)}{\dd t}$.
	We can again assume that, for all $k$ such that $ t_n>2^{k}c_0x$, 
	\[
	|(\tilde{H}^*-H^*)(t+t_n/2^k)-(\tilde{H}^*-H^*)(t)|\leq c_1 \sqrt{\frac{c_0}{4}}\sqrt{\frac{t_n}{2^k}}n^{-1/2}(\log n)^{1/2}.
	\]
	Then, we have 
	{\small\[
		\begin{split}
		&P_n^*\left(\sup_{\substack{s\in[t+c_0x,t+t_n]\\ s\leq M}}\left|(H_n^*-H^*)(s)-(H_n^*-H^*)(t)\right|\geq cx\right)\\
		&\leq \sum_{k:\, t_n>2^{k}c_0x}
		P^*_n\left(\sup_{\substack{s\in[t+t_n/2^{k+1},t+t_n/2^k]\\ s\leq M}}\left|\mathcal{M}_n^*(s)-1\right|\geq \tilde{c}x\left(\frac{t_n}{2^k}+c_1\sqrt{\frac{c_0}{4}}\sqrt{\frac{t_n}{2^k}}n^{-1/2}(\log n)^{1/2}\right)^{-1}\right)\\ 
		&\leq\sum_{k:\, t_n>2^{k}c_0x}
		\exp\left(-r_k\tilde{c}x\left(\frac{t_n}{2^k}+c_1\sqrt{\frac{c_0}{4}}\sqrt{\frac{t_n}{2^k}}n^{-1/2}(\log n)^{1/2}\right)^{-1}\right)\E\left[\exp\left(\frac{r_k}{np}\left|X-E[X]\right|\right)\right],
		\end{split}
		\]}
	where $p=H^*(t+t_n/2^{k+1})-H^*(t)$ and $X=n[H_n^*(t+t_n/2^{k+1})-H_n^*(t)]$ is a Binomial distribution with parameters $n$ and  $p$. It then follows as in the proof of Lemma \ref{le:S2} that, 
	\[
	\begin{split}
	&
	P_n^*\left(\sup_{\substack{s\in[t+c_0x,t+t_n]\\ s\leq M}}\left|(H_n^*-H^*)(s)-(H_n^*-H^*)(t)\right|\geq cx\right)\\
	&\leq\sum_{k=1}^{\infty} 2\exp\left(-Knpx^2\left(\frac{t_n}{2^k}+c_1\sqrt{\frac{c_0}{4}}\sqrt{\frac{t_n}{2^k}}n^{-1/2}(\log n)^{1/2}\right)^{-2}\right).
	\end{split}
	\]
	Note, since $x\geq n^{-1}\log n$ and $t\geq 2^kc_0x$,  we have 
	\[
	p\geq c_1\frac{t_n}{2^{k+1}}-c_1\sqrt{\frac{c_0}{4}}\sqrt{\frac{t_n}{2^{k+1}}}n^{-1/2}(\log n)^{1/2}\geq \frac{c_1}{2}\left(1-\frac{1}{\sqrt{2}}\right)\frac{t_n}{2^k}
	\]
	and
	\[
	\frac{t_n}{2^k}+c_1\sqrt{\frac{c_0}{4}}\sqrt{\frac{t_n}{2^k}}n^{-1/2}(\log n)^{1/2}\leq \left(1+\frac{c_1}{2}\right)\frac{t_n}{2^k}
	\]
	As a result we conclude that 
	\[
	P_n^*\left(\sup_{\substack{s\in[t+c_0x,t+t_n]\\ s\leq M}}\left|(H_n^*-H^*)(s)-(H_n^*-H^*)(t)\right|\geq cx\right)\leq 2\exp\left(-K_1nx^2t_n^{-1}\right).
	\]
	The case $t-t_n<s<t$ can be treated similarly.
\end{proof}
\begin{lemma}
	\label{lemma:approx_Phi*}
	{Under the assumptions of Theorem \ref{theo:CLT*},} it holds
	\[
	\sup_{x\in[\epsilon,M]}\left|\Phi^*(x;\tilde\beta_n)-\Phi(x;\beta_0)\right|=O_P(n^{-1/2}).
	\]
	In particular, with probability converging to one, given the data it holds
	\[
	\sup_{x\in[\epsilon,M]}\left|\Phi^*(x;\tilde\beta_n)-\Phi(x;\beta_0)\right|\lesssim n^{-1/2+\alpha}.
	\]
\end{lemma}
\begin{proof}
	By definition we have
	\[
	\begin{split}
	\Phi^*(x;\tilde\beta_n)
	&=\frac{1}{n}\sum_{i=1}^n e^{\tilde\beta'_nZ_i}P_n^*(T^*\geq x|Z^*=Z_i)\\
	&=\frac{1}{n}\sum_{i=1}^n e^{\tilde\beta'_nZ_i}P_n^*(X^*\geq x|Z^*=Z_i)P_n^*(C^*\geq x)\\
	&=[1-\hat{G}_n(x)]\frac{1}{n}\sum_{i=1}^n e^{\tilde\beta'_nZ_i}\exp\left[-\Lambda_{\theta}(x)e^{\tilde\beta'_nZ_i}\right].
	\end{split}
	\]
	Since, for the Kaplan-Meier estimator, it holds
	\[
	\sqrt{n}\sup_{x\in[0,M]}|\hat G_n(x)-G(x)|=O_P(1),
	\]
	we obtain 
	\[
	\begin{split}
	&\sup_{x\in[\epsilon,M]}\left|\Phi^*(x;\tilde\beta_n)-[1-G(x)]\frac{1}{n}\sum_{i=1}^n e^{\tilde\beta'_nZ_i}\exp\left[-\Lambda^s_n(x)e^{\tilde\beta'_nZ_i}\right]\right|\\
	&\leq O_P(n^{-1/2}) \limsup_{n\to\infty}\frac{1}{n}\sum_{i+1}^n e^{\tilde\beta'_nZ_i}\\
	&\leq O_P(n^{-1/2}) \sup_{|\beta-\beta_0|\leq\epsilon}\E\left[e^{\beta'Z}\right]\\
	&=O_P(n^{-1/2}).
	\end{split}
	\]
	Similarly, using $|\tilde\beta_n-\beta_0|=O_P(n^{-1/2})$ and $\sup_{t\in[\epsilon,M]}|\Lambda_{\hat\theta}(t)-\Lambda_0(t)|=O_P(n^{-1/2})$, we can replace $\tilde\beta_n$ and $\Lambda_{\hat\theta}$ by $\beta_0$ and $\Lambda_0$. Then, the statement follows from 
	\[
	\begin{split}
	&\frac{1}{n}\sum_{i=1}^n e^{\beta'_0Z_i}\left[1-F(x|Z_i)\right]\left[1-G(x)\right]=	\frac{1}{n}\sum_{i=1}^ne^{\beta'_0Z_i}\p\left(T>x|Z_i\right)\\
	&=\E\left[e^{\beta'_0Z}\p\left(T>x|Z\right)\right]+O_P(n^{-1/2})=\Phi(x;\beta_0)+O_P(n^{-1/2}).
	\end{split}
	\]
\end{proof}
The inverse process $\hat U_n^*$ is defined by
\begin{equation}
\label{def:U_n*}
\hat{U}_n^*(a)
=
\argmax_{x\in[\epsilon,M]} \left\{\Lambda_n^*(x)-ax\right\}.
\end{equation}
and it satisfies the switching relation, $\hat{\lambda}_n^*(t)\geq a$ if and only if $\hat{U}_n^*(a)\geq t$, for $t\in(\epsilon,M]$. 
Moreover, let $U_{\hat\theta}$ be defined as in \eqref{def:U_theta}.
\begin{lemma}
	\label{le:inv_tail_prob1*}
	{Under the assumptions of Theorem \ref{theo:CLT*},} there exist  constants $K_1,\, K_2>0$ such that, 	for every $a\geq 0$ and $x>0$, with probability converging to one, given the data 
	\begin{equation}
	\label{eqn:inv*}
	P_n^*\left(
	\left\{|\hat{U}_n^*(a)-U_{\hat\theta}(a)|\geq x\right\}
	\cap E_n^*
	\right)
	\leq
	K_1\exp\left(-K_2nx^3\right).
	\end{equation}
\end{lemma}
\begin{proof}
	We follow the proof of Lemma \ref{le:inv_tail_prob1}. Note that, since $\sup_{t\in[\epsilon,M]}|\lambda'_{\hat\theta}(t)-\lambda'_0(t)|\to 0$, we have $\liminf_{n\to\infty}\inf_{t\in[\epsilon,M]}|\lambda'_{\hat\theta}(t)|=c>0$.
	As a result, by Lemma \ref{le:Breslow*}, this time we have  
	\[
	\begin{split}
	&
	P_n^*\left(\left\{
	\hat{U}_n^*(a)\geq U_{\hat\theta}(a)+x \right\}\cap E_n^*\right)\\
	&\leq\sum_{k=0}^\infty
	P^*_n\left(\left\{
	\sup_{y\in I_k}
	\big|S_n^*(y)-S_n^*(U_{\hat\theta}(a))\big|
	\geq \frac{c}{2}\,x^2\,2^{2k}\right\}\cap E_n^*\right)\\
	&\qquad+\sum_{k=0}^\infty
	P^*_n\left(\left\{
	\sup_{y\in I_k}
	|R_n^*(y)-R_n^*(U_{\hat\theta}(a))|
	\geq \frac{c}{2}\,x^2\,2^{2k}\right\}\cap E_n^*\right)
	\end{split}
	\] 
	where
	\[
	S_n^*(t)=\int \frac{\delta\1_{\{u\leq t\}}}{\Phi^*(u;\tilde\beta_n)}\,\dd (\p_n^*-P^*_n)(u,\delta,z),
	\]
	and
	\[
	R_n^*(t)=\int \delta\1_{\{u\leq t\}}\left(\frac{1}{\Phi_n^*(u;\hat\beta_n^*)}-\frac{1}{\Phi^*(u;\tilde\beta_n)}\right)\,\dd \p_n^*(u,\delta,z).
	\]
	The process $S_n^*$ can be written as
	\[
	\begin{split}
	S_n^*(t)&=\frac{1}{n}\sum_{i=1}^n\left(\Delta_i^*-\E^*_T[\Delta_i^*]\right)\frac{1}{\Phi(T_i^*;\beta_0)}\1_{\{T_i^*\leq t\}}\\
	&\quad+\frac{1}{n}\sum_{i=1}^n\left(\E^*_T\left[\Delta_i^*\right]\frac{1}{\Phi(T_i^*;\beta_0)}\1_{\{T_i^*\leq t\}}-\E^*\left[\frac{\Delta_i^*}{\Phi(T_i^*;\beta_0)}\1_{\{T_i^*\leq t\}}\right]\right)\\
	&\quad+\int \delta\1_{\{u\leq t\}}\left(\frac{1}{\Phi^*(u;\hat\beta_n)}-\frac{1}{\Phi(u;\beta_0)}\right)\,\dd (\p_n^*-P^*_n)(u,\delta,z)\\
	&=S_n^{1,*}(t)+S^{2,*}_n(t)+S_n^{3,*}(t),
	\end{split}
	\]
	where $S^{1,*}_n,\,S^{2,*}_n$ are defined as in Lemmas \ref{le:S1*}, \ref{le:S2*} with 
	\[
	r(t)=\frac{1}{\Phi(t;\beta_0)}<\frac{1}{\Phi(M;\beta_0)}\quad\text{and}\quad	m(t)=\frac{\E^*\left[\Delta^*\,|\,T^*=t\right]}{\Phi(t;\beta_0)}<\frac{1}{\Phi(M;\beta_0)}.
	\]
	It follows by Lemmas \ref{le:S1*}, \ref{le:S2*} that  there exist some positive constants $K_1,\, K_2$ such that
	\begin{equation*}
	\begin{split}
	&P_n^*\left(
	\sup_{y\in I_k}
	\big|S^{1,*}_n(y)-S^{1,*}_n(U_{\hat\theta}(a))\big|
	\geq cx^2\,2^{2k}\right)+P^*_n\left(
	\sup_{y\in I_k}
	\big|S_n^{2,*}(y)-S_n^{2,*}(U_{\hat\theta}(a))\big|
	\geq cx^2\,2^{2k}\right)\\
	&\leq K_1\exp\left(-K_2nx^32^{3k}\right).
	\end{split}
	\end{equation*}
	Note that when we applying Lemmas \ref{le:S1*}, \ref{le:S2*} we have 
	$t_n=x2^{k+1}$ and $x$ replaced by $x^22^{2k}>n^{-1}\log n$ since it is sufficient to consider $x>n^{-1/3}.$ On the other hand, for $S^{3,*}_n$ we have 
	\[
	\begin{split}
	\sup_{y\in I_k}
	|S^{3,*}_n(y)-S^{3,*}_n(U_{\hat\theta}(a))|&\leq \sup_{y\in I_k}
	\frac{1}{n}\sum_{i=1}^n \Delta_i^*\1_{\{U_{\hat\theta}(a)<T_i^*\leq y\}}\left|\frac{1}{\Phi^*(T_i^*;\tilde\beta_n)}-\frac{1}{\Phi(T_i^*;\beta_0)}\right|\\
	& +\sup_{y\in I_k}\E^*\left[\Delta^*\1_{\{U_{\hat\theta}(a)<T^*\leq y\}}\left|\frac{1}{\Phi^*(T^*;\tilde\beta_n)}-\frac{1}{\Phi(T^*;\beta_0)}\right|\right].
	\end{split}
	\]
	By Lemma \ref{lemma:approx_Phi*}, we have that, with probability converging to one, given the data 
	\[
	\begin{split}
	&\sup_{y\in I_k}\E^*\left[\Delta^*\1_{\{U_{\hat\theta}(a)<T^*\leq y\}}\left|\frac{1}{\Phi^*(T^*;\tilde\beta_n)}-\frac{1}{\Phi(T^*;\beta_0)}\right|\right]\\
	&\leq\sup_{t\in[\epsilon,M]}\left|\frac{1}{\Phi^*(t;\tilde\beta_n)}-\frac{1}{\Phi(t;\beta_0)}\right|\sup_{y\in I_k}\E^*\left[\Delta^*\1_{\{U_{\hat\theta}(a)<T^*\leq y\}}\right]\\
	&\lesssim n^{-1/2+\alpha}\sup_{y\in I_k}\left[H^{uc,*}(y)-H^{uc,*}(U_{\hat\theta}(a))\right]\\
	&\lesssim n^{-1/2+\alpha}\left(x2^{k+1}+n^{-1/2+\alpha}\right).
	\end{split}
	\]
	Since $n^{-1/2+\alpha}x2^{k+1}<x^22^{2k}$, it follows that 
	\[
	\begin{split}
	&P^*_n\left(\left\{
	\sup_{y\in I_k}
	|S^{3,*}_n(y)-S^{3,*}_n(U_{\hat\theta}(a))|
	\geq cx^2\,2^{2k}\right\}\cap E_n^*\right)\\
	&\leq P_n^*\left(\left\{
	\sup_{y\in I_k}
	\sum_{i=1}^n \Delta_i^*\1_{\{U_{\hat\theta}(a)<T_i^*\leq y\}}\left|\frac{1}{\Phi^*(T_i^*;\tilde\beta_n)}-\frac{1}{\Phi(T_i^*;\beta_0)}\right|
	\geq cnx^2\,2^{2k}\right\}\cap E_n^*\right)\\
	&\leq P^*_n\left(\left\{
	\sum_{i=1}^n \Delta_i^*\1_{\{U_{\hat\theta}(a)<T_i^*\leq U_{\hat\theta}(a)+x2^{k+1}\}}
	\geq cn^{3/2-\alpha}\,x^2\,2^{2k}\right\}\cap E_n^*\right)\\
	&\leq P^*_n\left(\left\{
	\left|\sum_{i=1}^n \left(\Delta_i^*\1_{\{U_{\hat\theta}(a)<T_i^*\leq U_{\hat\theta}(a)+x2^{k+1}\}}-\E^*\left[\Delta^*\1_{\{U_{\hat\theta}(a)<T^*\leq U_{\hat\theta}(a)+x2^{k+1}\}}\right]\right)\right|
	\geq cn^{3/2-\alpha}\,x^2\,2^{2k}\right\}\cap E_n^*\right).
	\end{split}
	\]
	Note that $n\E^*\left[\Delta^*\1_{\{U_{\hat\theta}(a)<T^*\leq U_{\hat\theta}(a)+x2^{k+1}\}}\right]=O(nx2^{k+1})<O(n^{3/2-\alpha}\,x^2\,2^{2k})$ because $x>n^{-1/3}$ and $\alpha$ can be chosen small enough. From  Bernstein inequality we obtain
	\[
	\begin{split}
	&P^*_n\left(\left\{
	\sup_{y\in I_k}
	\left|S^{3,*}_n(y)-S^{3,*}_n(U_{\hat\theta}(a))\right|
	\geq c\,x^2\,2^{2k}\right\}\cap E_n^*\right)\\
	&\leq 2\exp\left\{-c\frac{n^{3-2\alpha}\,x^4\,2^{4k}}{n^{3/2-\alpha}\,x^2\,2^{2k}+nx2^{k+1}} \right\}\\
	&\leq2\exp\left\{-c\,2^{2k} n^{3/2-\alpha}\,x^2\right\}\\
	&\leq 2\exp\left\{-c\,2^{2k}  nx^3\right\}.
	\end{split}
	\]
	As a result we conclude that 
	\[
	\sum_{k=0}^\infty
	P^*_n\left(\left\{
	\sup_{y\in I_k}
	\big|S_n^*(y)-S_n^*(U_{\hat\theta}(a))\big|
	\geq c\,x^2\,2^{2k}\right\}\cap E_n^*\right)\leq K_1\exp\left(-K_2nx^3\right).
	\]
	In the same way, using \eqref{eqn:bound_inv_Phi_n*} we can deal with $R^*_n$.
\end{proof}
\begin{lemma}
	\label{le:inv_tail_prob2*}
	{Under the assumptions of Theorem \ref{theo:CLT*},} there exist positive constants $K_1,\, K_2$ such that, 	for every  {$x>0$} and $a\notin \lambda_{\hat\theta}([\epsilon,M])$ such that $x\left|\lambda_{\hat\theta}(U_{\hat\theta}(a))-a\right|\geq n^{-1}\log n$, with probability converging to one, given the data 
	\begin{equation*}
	P^*_n\left(
	\left\{|\hat{U}_n^*(a)-U_{\hat\theta}(a)|\geq x\right\}
	\cap E_n^*
	\right)
	\leq K_1\exp\left\{-K_2nx\left|\lambda_{\hat\theta}(U_{\hat\theta}(a))-a\right| \min\left(1,\left|\lambda_{\hat\theta}(U_{\hat\theta}(a))-a\right|\right)\right\}.
	\end{equation*}
\end{lemma}
\begin{proof}
	We argue as in the proof of Lemma \ref{le:inv_tail_prob2}. Here we have 
	\[
	\Lambda_{\hat\theta}(y)-\Lambda_{\hat\theta}(U_{\hat\theta}(a))\leq(y-U_{\hat\theta}(a))\lambda_{\hat\theta}(U_{\hat\theta}(a))
	\]	
	and as a result, we obtain
	\[
	\begin{split}
	&
	P^*_n\left(\left\{
	\hat{U}_n^*(a)\geq U_{\hat\theta}(a)+x \right\}\cap E_n^*\right)\\
	&\sum_{k=0}^\infty
	P^*_n\left(\left\{
	\sup_{y\in I_k}
	\big|S_n^*(y)-S_n^*(U_{\hat\theta}(a))\big|
	\geq cx\,2^{k}\left|\lambda_{\hat\theta}(U_{\hat\theta}(a))-a\right|\right\}\cap E_n^*\right)\\
	&\quad+\sum_{k=0}^\infty
	P^*_n\left(\left\{
	\sup_{y\in I_k}
	|R^*_n(y)-R^*_n(U(a))|
	\geq cx\,2^{k}\left|\lambda_{\hat\theta}(U_{\hat\theta}(a))-a\right|\right\}\cap E_n^*\right),
	\end{split}
	\]
	where $S_n^*$ and $R_n^*$ are as in Lemma \ref{le:inv_tail_prob1*}.
	From Lemmas \ref{le:S1*}, \ref{le:S2*}, it follows that, for $i=1,2$,
	\[
	\begin{split}
	&\sum_{k=0}^\infty P^*_n\left(
	\sup_{y\in I_k}
	\big|S_n^{i,*}(y)-S_n^{i,*}(U_{\hat\theta}(a))\big|
	\geq cx\,2^{k}\left|\lambda_{\hat\theta}(U_{\hat\theta}(a))-a\right|\right)\\
	&\leq K_1\exp\Big(-K_2nx\left|\lambda_{\hat\theta}(U_{\hat\theta}(a))-a\right|\min\left\{1,\left|\lambda_{\hat\theta}(U_{\hat\theta}(a))-a\right| \right\}\Big).
	\end{split}
	\]
	Using \eqref{eqn:bound_inv_Phi_n*} and the approximation of $H^{uc,*}$ by $\tilde{H}^{uc,*}$, it can be shown as in  Lemma \ref{le:inv_tail_prob2} that
	{\small	\[
		\sum_{k=0}^\infty
		P^*_n\left(\left\{
		\sup_{y\in I_k}
		|R_n^*(y)-R_n^*(U_{\hat\theta}(a))|
		\geq cx\,2^{k}\left|\lambda_{\hat\theta}(U_{\hat\theta}(a))-a\right|\right\}\cap E_n^*\right)\leq K_1 \exp\left\{-K_2 nx\left|\lambda_{\hat\theta}(U_{\hat\theta}(a))-a\right|\right\}.
		\]}
	In the same we deal with $S^{3,*}_n$ (see also the proof of Lemma \ref{le:inv_tail_prob1*}) but, instead of \eqref{eqn:bound_inv_Phi_n*},  we use 
	\[
	\sup_{t\in[\epsilon,M]}\left|\frac{1}{\Phi^*(t;\hat\beta_n)}-\frac{1}{\Phi(t;\beta_0)}\right|\lesssim n^{-1/2+\alpha}
	\]
\end{proof}
\begin{lemma}
	\label{le:bound_expectation_L_P*}
	Let $p\geq 1$. {Under the assumptions of Theorem \ref{theo:CLT*},} for each $\alpha_0\in(0,1)$, there exists $K>0$ such that, with probability converging to one, given the data,
	\[
	\E^*\left[\1_{E_n^*}|\hat\lambda_n^*(t)-\lambda_{\hat\theta}(t)|^p\right]\leq Kn^{-p/3},
	\]
	for all $t\in [\epsilon+n^{-1/3},M-n^{-1/3}]$,
	\[
	\E^*\left[\1_{E_n^*}|\hat\lambda_n^*(t)-\lambda_{\hat\theta}(t)|^p\right]\leq K\left[n((t-\epsilon)\wedge(M-t))\right]^{-p/2},
	\]
	for all $t\in [\epsilon+n^{-1/2}\sqrt{\log n},\epsilon+n^{-1/3}]\cup[M-n^{-1/3},M-n^{-1/2}\sqrt{\log n}]$, and
	\[
	\E^*\left[\1_{E_n^*}|\hat\lambda_n^*(t)-\lambda_{\hat\theta}(t)|^p\right]\leq K\left[n((t-\epsilon)^{1+\alpha_0}\wedge(M-t)^{1+\alpha_0})\right]^{-p/2},
	\]
	for all $t\in [\epsilon+n^{-1}\log n,\epsilon+n^{-1/2}\sqrt{\log n}]\cup[M-n^{-1/2}\sqrt{\log n},M-n^{-1}\log n]$. 	
\end{lemma}
\begin{proof}As in the proof of Theorem \ref{theo:bound_expectation_L_P} we obtain
	\begin{equation}
	\label{eqn:I_1*}
	\begin{split}
	I_1&\leq \int_0^{\lambda_{\hat\theta}(\epsilon)-\lambda_{\hat\theta}(t)} P^*_n\left[\left\{\hat{U}_n^*(\lambda_{\hat\theta}(t)+x)>t \right\} \cap E_n^*\right]\,px^{p-1}\,\dd x\\
	&\quad+\int_{\lambda_{\hat\theta}(\epsilon)-\lambda_{\hat\theta}(t)}^{\infty}P^*_n\left[\left\{\hat{U}_n^*(\lambda_{\hat\theta}(t)+x)>t \right\} \cap E_n^*\right]px^{p-1}\,\dd x.
	\end{split}
	\end{equation}
	Since 
	\[
	U'^*(a)=\frac{1}{\lambda'_{\hat\theta}(\lambda_{\hat{\theta}}^{-1}(a))}\to\frac{1}{\lambda'_0(\lambda_0^{-1}(a))},
	\]
	we have $\liminf_{n\to\infty}\inf_{a\in(\lambda_{\hat\theta}(M),\lambda_{\hat\theta}(\epsilon))}|U'^*(a)|>c>0$.
	As a result, for  $x<\lambda_{\hat\theta}(\epsilon)-\lambda_{\hat\theta}(t)$ it holds that $U_{\hat\theta}(\lambda_{\hat\theta}(t)+x)<t-cx$. Hence, as in the proof of Theorem \ref{theo:bound_expectation_L_P}, we obtain
	\[
	\int_0^{\lambda_{\hat\theta}(\epsilon)-\lambda_{\hat\theta}(t)} P^*_n\left[\left\{\hat{U}_n^*(\lambda_{\hat\theta}(t)+x)>t \right\} \cap E_n^*\right]\,px^{p-1}\,\dd x
	\lesssim
	n^{-p/3}.
	\]
	For $x>\lambda_{\hat\theta}(\epsilon)-\lambda_{\hat\theta}(t)$, we have $U_{\hat\theta}(\lambda_{\hat\theta}(t)+x))=\epsilon$. We consider first the case $t-\epsilon\geq n^{-1/2}\sqrt{\log n}$.  As in the proof of Theorem \ref{theo:bound_expectation_L_P} it follows that
	\[
	\begin{split}
	&\int_{\lambda_{\hat\theta}(\epsilon)-\lambda_{\hat\theta}(t)}^{\lambda_{\hat\theta}(\epsilon)-\lambda_{\hat\theta}(t)+t-\epsilon}P^*_n\left[\left\{\hat{U}_n^*(\lambda_{\hat\theta}(t)+x)>t \right\} \cap E_n^*\right]px^{p-1}\,\dd x\\
	&\lesssim \exp\left\{-K_2n(t-\epsilon)^3 \right\} \int_{0}^{\lambda_{\hat\theta}(\epsilon)-\lambda_{\hat\theta}(t)+t-\epsilon}px^{p-1}\,\dd x\\
	&\lesssim n^{-p/3}.
	\end{split}
	\]
	Here we used that $\limsup_{n\to\infty}\sup_{t\in[\epsilon,M]}|\lambda'_{\hat\theta}(t)|<c$.
	Moreover, from  Lemma \ref{le:inv_tail_prob2*}, we have
	\begin{equation}
	\label{eqn:bound_expectation1*}
	\begin{split}
	&\int_{\lambda_{\hat\theta}(\epsilon)-\lambda_{\hat\theta}(t)+t-\epsilon}^\infty P^*_n\left[\left\{\hat{U}_n^*(\lambda_{\hat\theta}(t)+x)>t \right\} \cap E_n^*\right]px^{p-1}\,\dd x\\
	&\lesssim \int_{\lambda_{\hat\theta}(\epsilon)-\lambda_{\hat\theta}(t)+t-\epsilon}^{\lambda_{\hat\theta}(\epsilon)-\lambda_{\hat\theta}(t)+M} \exp\left\{-K_2n(t-\epsilon)\left|\lambda_{\hat\theta}(\epsilon)-\lambda_{\hat\theta}(t)-x\right|^{2} \right\} px^{p-1}\,\dd x\\
	&\quad +\int_{\lambda_{\hat\theta}(\epsilon)-\lambda_{\hat\theta}(t)+M}^\infty \exp\left\{-K_2n(t-\epsilon)\left|\lambda_{\hat\theta}(\epsilon)-\lambda_{\hat\theta}(t)-x\right|\right\} px^{p-1}\,\dd x.
	\end{split}
	\end{equation}
	Here we needed $t-\epsilon\geq n^{-1/2}\sqrt{\log n}$ to apply Lemma \ref{le:inv_tail_prob2*}, i.e. we need $(t-\epsilon)\left|\lambda_{\hat\theta}(\epsilon)-\lambda_{\hat\theta}(t)-x\right|>n^{-1}\log n$, which is satisfied since $\left|\lambda_{\hat\theta}(\epsilon)-\lambda_{\hat\theta}(t)-x\right|>t-\epsilon $.
	As in the proof of Theorem \ref{theo:bound_expectation_L_P}, we have
	\[
	\begin{split}
	&\int_{\lambda_{\hat\theta}(\epsilon)-\lambda_{\hat\theta}(t)+t-\epsilon}^{\lambda_{\hat\theta}(\epsilon)-\lambda_{\hat\theta}(t)+M} \exp\left\{-K_2n(t-\epsilon)\left|\lambda_{\hat\theta}(\epsilon)-\lambda_{\hat\theta}(t)-x\right|^{2} \right\} px^{p-1}\,\dd x\\
	&=\int_{t-\epsilon}^{M} \exp\left\{-K_2n(t-\epsilon)y^{2} \right\} p\left(y+\lambda_{\hat\theta}(\epsilon)-\lambda_{\hat\theta}(t)\right)^{p-1}\,\dd y\\
	&\lesssim (n(t-\epsilon))^{-p/2}
	\end{split}
	\]
	and
	\[
	\begin{split}
	&\int_{\lambda_{\hat\theta}(\epsilon)-\lambda_{\hat\theta}(t)+M}^\infty \exp\left\{-K_2n(t-\epsilon)\left|\lambda_{\hat\theta}(\epsilon)-\lambda_{\hat\theta}(t)-x\right|\right\} px^{p-1}\,\dd x\\
	&=\int_{M}^{\infty} \exp\left\{-K_2n(t-\epsilon)y \right\} p\left(y+\lambda_{\hat\theta}(\epsilon)-\lambda_{\hat\theta}(t)\right)^{p-1}\,\dd y\\
	&\lesssim (n(t-\epsilon))^{-p}.
	\end{split}
	\]
	Since, $t>\epsilon+n^{-1}$, $(n(t-\epsilon))^{-p}\leq(n(t-\epsilon))^{-p/2}$ and, in particular,  for $t\geq\epsilon+ n^{-1/3}$ we get
	\[
	\int_{\lambda_{\hat\theta}(\epsilon)-\lambda_{\hat\theta}(t)+t-\epsilon}^\infty \exp\left\{-K_2n(t-\epsilon)\left|\lambda_{\hat\theta}(\epsilon)-\lambda_{\hat\theta}(t)-x\right|\right\} px^{p-1}\,\dd x\lesssim n^{-p/3}.
	\]
	Similarly we can deal with $I_2^*$.
	
	Now we consider  the case $n^{-1}\log n\leq t-\epsilon< n^{-1/2}\sqrt{\log n}$.  As in the proof of Theorem \ref{theo:bound_expectation_L_P} it follows that
	\[
	\begin{split}
	&\int_{\lambda_{\hat\theta}(\epsilon)-\lambda_{\hat\theta}(t)}^{\lambda_{\hat\theta}(\epsilon)-\lambda_{\hat\theta}(t)+\frac{\log n}{n(t-\epsilon)}}P^*_n\left[\left\{\hat{U}_n^*(\lambda_{\hat\theta}(t)+x)>t \right\} \cap E_n^*\right]px^{p-1}\,\dd x\\
	&\lesssim \exp\left\{-K_2n(t-\epsilon)^3 \right\} \int_{0}^{\lambda_{\hat\theta}(\epsilon)-\lambda_{\hat\theta}(t)+\frac{\log n}{n(t-\epsilon)}}px^{p-1}\,\dd x\\
	&\lesssim \frac{(\log n)^p}{(n(t-\epsilon)^3)^{\alpha_0 p/5}n^p(t-\epsilon)^{p}}\\
	&\lesssim \left(n(t-\epsilon)^{1+\alpha_0}\right)^{-p/2}.
	\end{split}
	\]
	Moreover, from  Lemma \ref{le:inv_tail_prob2*}, we have
	\begin{equation*}
	\begin{split}
	&\int_{\lambda_{\hat\theta}(\epsilon)-\lambda_{\hat\theta}(t)+\frac{\log n}{n(t-\epsilon)}}^\infty P^*_n\left[\left\{\hat{U}_n^*(\lambda_{\hat\theta}(t)+x)>t \right\} \cap E_n^*\right]px^{p-1}\,\dd x\\
	&\lesssim \int_{\lambda_{\hat\theta}(\epsilon)-\lambda_{\hat\theta}(t)+\frac{\log n}{n(t-\epsilon)}}^{\lambda_{\hat\theta}(\epsilon)-\lambda_{\hat\theta}(t)+M} \exp\left\{-K_2n(t-\epsilon)\left|\lambda_{\hat\theta}(\epsilon)-\lambda_{\hat\theta}(t)-x\right|^{2} \right\} px^{p-1}\,\dd x\\
	&\quad +\int_{\lambda_{\hat\theta}(\epsilon)-\lambda_{\hat\theta}(t)+M}^\infty \exp\left\{-K_2n(t-\epsilon)\left|\lambda_{\hat\theta}(\epsilon)-\lambda_{\hat\theta}(t)-x\right|\right\} px^{p-1}\,\dd x.
	\end{split}
	\end{equation*}
	Note that we can apply Lemma \ref{le:inv_tail_prob2*} because  $(t-\epsilon)\left|\lambda_{\hat\theta}(\epsilon)-\lambda_{\hat\theta}(t)-x\right|>n^{-1}\log n$.
	As in the proof of Theorem \ref{theo:bound_expectation_L_P}, we have
	\[
	\begin{split}
	&\int_{\lambda_{\hat\theta}(\epsilon)-\lambda_{\hat\theta}(t)+\frac{\log n}{n(t-\epsilon)}}^{\lambda_{\hat\theta}(\epsilon)-\lambda_{\hat\theta}(t)+M} \exp\left\{-K_2n(t-\epsilon)\left|\lambda_{\hat\theta}(\epsilon)-\lambda_{\hat\theta}(t)-x\right|^{2} \right\} px^{p-1}\,\dd x\\
	&=\int_{\frac{\log n}{n(t-\epsilon)}}^{M} \exp\left\{-K_2n(t-\epsilon)y^{2} \right\} p\left(y+\lambda_{\hat\theta}(\epsilon)-\lambda_{\hat\theta}(t)\right)^{p-1}\,\dd y\\
	&\lesssim (n(t-\epsilon))^{-p/2}
	\end{split}
	\]
	and
	\[
	\begin{split}
	&\int_{\lambda_{\hat\theta}(\epsilon)-\lambda_{\hat\theta}(t)+M}^\infty \exp\left\{-K_2n(t-\epsilon)\left|\lambda_{\hat\theta}(\epsilon)-\lambda_{\hat\theta}(t)-x\right|\right\} px^{p-1}\,\dd x\\
	&=\int_{M}^{\infty} \exp\left\{-K_2n(t-\epsilon)y \right\} p\left(y+\lambda_{\hat\theta}(\epsilon)-\lambda_{\hat\theta}(t)\right)^{p-1}\,\dd y\\
	&\lesssim (n(t-\epsilon))^{-p}\leq(n(t-\epsilon))^{-p/2}.
	\end{split}
	\]
	Similarly we can deal with $I_2^*$.
\end{proof}
\begin{lemma}
	\label{le:embedding*}
	Let 
	\begin{equation}
	\label{def:tilde_M_n*}
	\begin{split}
	\tilde{M}_n^*(t)&=\frac{1}{n}\sum_{i=1}^{n}\left(\Delta_i^*\1_{\{T_i^*\leq t\}}-\E^*\left[\Delta_i^*\1_{\{T_i^*\leq t\}} \right]\right)\\
	&=\frac{1}{n}\sum_{i=1}^{n}\Delta_i^*\1_{\{T_i^*\leq t\}}-H^{uc,*}(t).
	\end{split}
	\end{equation}
	There exist a Brownian bridge $B_n$ and positive constants $K,\,c_1, \,c_2$ such that 
	\[
	P^*_n\left[n\sup_{s\in[0,M]}\left|\tilde{M}_n^*(s)-n^{-1/2}B_n\left(H^{uc,*}(s)\right) \right|\geq x+c_1\log n \right]\leq K\exp\left(-c_2 x\right).
	\]
\end{lemma}
\begin{proof}
	See proof of Lemma \ref{le:embedding}.
\end{proof}
\begin{lemma}
	\label{le:boundaries*}
	$\hat{\lambda}_n^*(\epsilon)$ and $\hat{\lambda}_n^*(M)$ are stochastically bounded with respect to $P^*_n$.
\end{lemma}
\begin{proof}
	See proof of Lemma \ref{le:boundaries}. Here we use that, since for 
	\[
	h^{uc,*}(t)=\frac{1}{n}\sum_{i=1}^nf_{\hat\theta}(t|Z_i)\left[1-\hat{G}_n(t)\right]
	\]
	we have 
	\[
	n^{1/2}\sup_{t\in[\epsilon,M]}\left|h^{uc,*}(t)-[1-G(t)]\E[f_{\hat\theta}(t|Z)]\right|,
	\]
	it follows that $h^{uc,*}$ is uniformly bounded.
\end{proof}

As in Section~\ref{sec:proof_CLT2}, we first need to establish similar results to those in Lemma \ref{le:inv_tail_prob1*} and Theorem \ref{le:bound_expectation_L_P*}, where $\lambda
_{\hat\theta}$ is replaced by $\lambda^*_{\hat\theta}$ and $U_{\hat\theta}$ by  $U^*_{\hat\theta}$ defined as
the inverse of $\lambda^*_{\hat{\theta}}$, 
\begin{equation}
\label{def:U_theta*}
U^*_{\hat\theta}(a)=\begin{cases}
\epsilon & \text{ if }a>\lambda^*_{\hat{\theta}}(\epsilon)\\
(\lambda^*_{\hat{\theta}})^{-1}(a) & \text{ if }a\in[\lambda^*_{\hat{\theta}}(M),\lambda^*_{\hat{\theta}}(\epsilon)]\\
M & \text{ if } a<\lambda^*_{\hat{\theta}}(M).
\end{cases}
\end{equation}
From \eqref{eqn:theta_hat*}, it follows that
\[
\begin{split}
&\sqrt{n}\sup_{t\in[\epsilon,M]}\left(\lambda^*_{\hat{\theta}}(t)-\lambda_{\hat\theta}(t)\right)=O_{P^*}(1),\quad\quad\sqrt{n}\sup_{t\in[\epsilon,M]}\left((\lambda^*_{\hat{\theta}})'(t)-\lambda'_{\hat\theta}(t)\right)=O_{P^*}(1),\\
&\sqrt{n}\sup_{a\in(0,\infty)}\left(U^*_{\hat\theta}(a)-U_{\hat\theta}(a)\right)=O_{P^*}(1),\quad\quad\sqrt{n}\sup_{a\in(0,\infty)}\left((U^*_{\hat\theta})'(a)-U'_{\hat\theta}(a)\right)=O_{P^*}(1).
\end{split}
\]
Thus, we can assume that on the event $E_n^*$ we have
\[
\left|\hat\theta_n^*-\hat\theta_n\right|\lesssim n^{-1/2+\alpha}
\]
and
\begin{equation}
\label{eqn:par_est-true*}
\begin{split}
&\sup_{t\in[\epsilon,M]}\left
|\lambda^*_{\hat{\theta}}(t)-\lambda_{\hat\theta}(t)\right|\lesssim n^{-1/2+\alpha},\qquad \sup_{t\in[\epsilon,M]}\left
|(\lambda^*_{\hat{\theta}})'(t)-\lambda'_{\hat\theta}(t)\right|\lesssim n^{-1/2+\alpha}\\
&\sup_{a\in(0,\infty)}\left|U^*_{\hat\theta}(a)-U_{\hat\theta}(a)\right|\lesssim n^{-1/2+\alpha},\qquad\sup_{a\in(0,\infty)}\left|U'_{\hat\theta}(a)-U'(a)\right|\lesssim n^{-1/2+\alpha}.
\end{split}
\end{equation}
From Lemma \ref{le:inv_tail_prob1*} and \eqref{eqn:par_est-true*}, it follows that
\begin{equation}
\label{eqn:inv2*}
\begin{split}
&P^*_n\left[\left\{\left|\hat{U}_n^*(a)-U_{\hat{\theta}}^*(a)\right|>x\right\}\cap E_n^*\right]\\
&\leq P^*_n\left[\left\{\left|\hat{U}_n^*(a)-U_{\hat\theta}(a)\right|>x/2\right\}\cap E_n^*\right]+P^*_n\left[\left\{\left|U_{\hat{\theta}}^*(a)-U_{\hat{\theta}}(a)\right|>x/2\right\}\cap E_n^*\right]\\
&\lesssim\exp(-Knx^3).
\end{split}
\end{equation}
Similarly, Theorem \ref{le:bound_expectation_L_P*} and \eqref{eqn:par_est-true*} imply that for  $t\in[\epsilon+n^{-1/3},M-n^{-1/3}] $ it holds
\begin{equation}
\label{eqn:bound_expectation_L_p1*}
\begin{split}
\E^*\left[1_{E_n^*}\left|\hat\lambda_n^*(t)-\lambda_{\hat{\theta}}^*(t)\right|^p\right]
&\lesssim \E^*\left[1_{E_n^*}\left|\hat\lambda_n^*(t)-\lambda_{\hat\theta}(t)\right|^p\right]+\E^*\left[1_{E_n^*}\left|\lambda_{\hat\theta}(t)-\lambda_{\hat{\theta}}^*(t)\right|^p\right]\\
&\lesssim n^{-p/3}+n^{p(-1/2+\alpha)}\lesssim n^{-p/3},
\end{split}
\end{equation}
for $t\in[\epsilon+n^{-1/2}\sqrt{\log n},\epsilon+n^{-1/3}]\cup[M-n^{-1/3},M-n^{-1/2}\sqrt{\log n}]$, 
\begin{equation}
\label{eqn:bound_expectation_L_p2*}
\E^*\left[1_{E_n^*}\left|\hat\lambda_n^*(t)-\lambda_{\hat{\theta}}^*       (t)\right|^p\right]\lesssim \left[n(t-\epsilon)\wedge n(M-t)\right]^{-p/2},
\end{equation}
and for $t\in[\epsilon+n^{-1}\log n,\epsilon+n^{-1/2}\sqrt{\log n}]\cup[M-n^{-1/2}\sqrt{\log n},M-n^{-1}\log n]$, 
\begin{equation}
\label{eqn:bound_expectation_L_p3*}
\E^*\left[1_{E_n^*}\left|\hat\lambda_n^*(t)-\lambda_{\hat{\theta}}^*       (t)\right|^p\right]\lesssim \left[n(t-\epsilon)^{1+\alpha_0}\wedge n(M-t)^{1+\alpha_0}\right]^{-p/2},
\end{equation}
Afterwards we  consider the conditional probability $P^*_\theta$ and conditional expectation $\E_{\theta^*}$ given $\hat\theta_n^*$ such that $\left|\hat\theta_n^*-\hat\theta_n\right|\lesssim n^{-1/2+\alpha}$ and obtain the bootstrap version of \ref{le:Lambda_n-Lambda}.
\begin{lemma}
	\label{le:Breslow_linear*} It holds
	\[
	\sup_{t,s\in[\epsilon,M]}\left|\int e^{\tilde\beta_nz}\int_{u\wedge t}^{u\wedge s}\left[\frac{\lambda_{\hat\theta}(v)}{\Phi^*(v;\tilde\beta_n)}-\frac{\lambda_{0}(v)}{\Phi(v;\beta_0)}\right]\,\dd v\,\dd (\p_n^*-P_n^*)(u,\delta,z)\right|=O_{P^*}(n^{-1}).
	\]
\end{lemma}
\begin{proof}
	Let $\F$ be the following class of functions on $[0,M]\times\{0,1\}\times\R^p$,
	\[
	\F=\left\{f_{t,s}(u,\delta,z)= e^{\tilde\beta_nz}\int_{u\wedge t}^{u\wedge s}\left[\frac{\lambda_{\hat\theta}(v)}{\Phi^*(v;\tilde\beta_n)}-\frac{\lambda_{0}(v)}{\Phi(v;\beta_0)}\right]\,\dd v\,\bigg|\,\epsilon\leq t<s\leq M\right\}
	\]	
	with envelope function
	\[
	F(u,\delta,z)=e^{\tilde\beta_nz}\1_{\{u>\epsilon\}}\int_\epsilon ^{u\wedge M}\left|\frac{\lambda_{\hat\theta}(v)}{\Phi^*(v;\tilde\beta_n)}-\frac{\lambda_{0}(v)}{\Phi(v;\beta_0)}\right|\,\dd v.
	\]
	By Markov inequality we have
	\[
	\begin{split}
	&P^*_n\left[n\sup_{t,s\in[\epsilon,M]}\left|\int e^{\tilde\beta_nz}\int_{u\wedge t}^{u\wedge s}\left[\frac{\lambda_{\hat\theta}(v)}{\Phi^*(v;\tilde\beta_n)}-\frac{\lambda_{0}(v)}{\Phi(v;\beta_0)}\right]\,\dd v\,\dd (\p_n^*-P_n^*)(u,\delta,z)\right|>x\right]\\
	&\lesssim nx^{-2}\E^*\left[\sup_{f\in\F}\left|\int f(u,\delta,z)\,\dd\sqrt{n}(\p_n^*-P^*_n)(u,\delta,z) \right|^2\right].
	\end{split}
	\]
	Hence, it suffices to show that the expectation in the right hand side of the previous equation is of order $n^{-1}$. From Theorem 2.14.5 and 2.14.2 in \cite{VW96}, it follows that 
	\begin{equation}
	\label{eqn:expectation*}
	\begin{split}
	&\E^*\left[\sup_{f\in\F}\left|\int f(u,\delta,z)\,\dd\sqrt{n}(\p_n^*-P^*_n)(u,\delta,z) \right|^2\right]^{1/2}\\
	&\lesssim \left[1+J_{[]}(1,\F,L_2(P^*_n))\right]\Vert F\Vert_{L_2(P^*_n)}.
	\end{split}
	\end{equation}
	Using Lemma \ref{lemma:approx_Phi*} and \eqref{eqn:par_est-true}, we obtain
	\[
	\Vert F\Vert_{L_2(P^*_n)}^2=\int e^{2\tilde\beta_nz}\1_{\{u>\epsilon\}}\left(\int_\epsilon^{u\wedge M}\left|\frac{\lambda_{\hat\theta}(v)}{\Phi^*(v;\tilde\beta_n)}-\frac{\lambda_{0}(v)}{\Phi(v;\beta_0)}\right|\,\dd v\right)^2 \,\dd P^*_n(u,\delta,z)=O(n^{-1}).
	\]
	Hence, 
	it remains to show that $J_{[]}(1,\F,L_2(P^*_n))$ is bounded. We first compute $N_{[]}\left(\varepsilon\Vert F\Vert_{L_2(P^*_n)},\F,L_2(P^*_n) \right)$. Divide the interval $[\epsilon,M]$ in $N$ subintervals of length $L=(M-\epsilon)/N$. For $i\leq j$,, $i,j=1,\dots,N$, let
	\[
	l_{i,j}(u,\delta,z)=e^{\tilde\beta_nz}\1_{\{u>\epsilon+iL\}}\int_{\epsilon+iL} ^{u\wedge (\epsilon+(j-1)L)}\frac{\lambda_{\hat\theta}(v)}{\Phi^*(v;\tilde\beta_n)}\,\dd v-e^{\tilde\beta_nz}\1_{\{u>\epsilon+(i-1)L\}}\int_{\epsilon+(i-1)L} ^{u\wedge (\epsilon+jL)}\frac{\lambda_{0}(v)}{\Phi(v;\beta_0)}\,\dd v,
	\]
	\[
	L_{i,j}(u,\delta,z)=e^{\tilde\beta_nz}\1_{\{u>\epsilon+(i-1)L\}}\int_{\epsilon+(i-1)L} ^{u\wedge (\epsilon+jL)}\frac{\lambda_{\hat\theta}(v)}{\Phi^*(v;\tilde\beta_n)}\,\dd v-e^{\tilde\beta_nz}\1_{\{u>\epsilon+iL\}}\int_{\epsilon+iL} ^{u\wedge (\epsilon+(j-1)L)}\frac{\lambda_{0}(v)}{\Phi(v;\beta_0)}\,\dd v.
	\]
	The brackets $[l_{i,j},L_{i,j}]$ for $i\leq j$, with the convention that $\int_a^b f(a)\,\dd x=0$ if $a>b$, cover all the class $\F$. Moreover,  we have
	$
	\Vert L_{i,j}-l_{i,j}\Vert_{L_2(P^*_n)}^2
	\lesssim L.
	$
	Thus, if we take 
	\[
	N=\left\lfloor\frac{K}{\varepsilon^2 \Vert F\Vert_{L_2(P^*_n)}^2}+1\right\rfloor,
	\]
	for a properly chosen constant $K>0$, then 
	$\Vert L_i-l_i\Vert_{L_2(P^*_n)}\leq \varepsilon\Vert F\Vert_{L_2(P^*_n)}.$
	It follows that 
	\[
	N_{[]}\left(\varepsilon\Vert F\Vert_{L_2(P^*_n)},\F,L_2(P^*_n) \right)\lesssim \left\lfloor\frac{K}{\varepsilon^2 \Vert F\Vert_{L_2(P^*_n)}^2}+1\right\rfloor ^2
	\]
	and consequently $J_{[]}(1,\F,L_2(P^*_n))$ is bounded.
\end{proof}
\begin{lemma}
	\label{le:R_1*}
	Let $T_n=n^{\gamma}$ for some $\gamma\in (0,1/3)$. For $t,\,s\in(\epsilon,M)$ and $\pi$ as in \eqref{def:pi},  define
	\begin{equation}
	\label{def:R_1*}
	R^{1,*}_n(t,s)=\int\delta \pi(u;t,s)\left(\frac{1}{\Phi(u;\beta_0)}-\frac{1}{\Phi(t;\beta_0)}\right)\,\dd (\p_n^*-P^*_n)(u,\delta,z).
	\end{equation}
	{Under the assumptions of Theorem \ref{theo:CLT*},} for each $2\leq q<2/(3\gamma)$ , there exist $K>0$ such that, for $x\geq 0$ and $a\in\R$,
	\[
	P^*_{\theta}\left[\left\{\sup_{|t_n-U^*_{\hat\theta}|\leq c  n^{-1/3}T_n}|R^{1,*}_n(t,t_n)|>x\right\}\cap E_n^*\right]\leq K x^{-q} n^{1-q}.
	\]
\end{lemma}
\begin{proof}
	We can follow the proof of  Lemma \ref{le:R_1_2} considering exactly the same class of functions with $U(a)$ replaced by $U_{\hat{\theta}}(a)$. Hence, we just need to show that $\Vert F\Vert_{L_q(P^*_n)}^q=O\left((n^{-1/3}T_n)^{q+1}\right)$ and that $J_{[]}(1,\F,L_2(P^*_n))$ is bounded.
	We have
	\[
	\begin{split}
	\Vert F\Vert_{L_q(P^*_n)}^q&=\int_{U_{\hat\theta}(a)-n^{-\frac12+\alpha}}^{U_{\hat\theta}(a)+n^{-\frac12+\alpha}+n^{-\frac13}T_n}(u-U_{\hat\theta}(a)+n^{-\frac12+\alpha})^q \left(\frac{\Phi'(U_{\hat\theta}(a)-n^{-\frac12+\alpha};\beta_0)}{\Phi(U_{\hat\theta}(a)-n^{-\frac12+\alpha};\beta_0)^2}\right)^q\,\dd H^{uc,*}(u)\\
	&\quad+o\left((n^{-1/3}T_n)^{q+1}\right)\\
	&=C(n^{-1/3}T_n)^{q+1}+o\left((n^{-1/3}T_n)^{q+1}\right),
	\end{split}
	\]
	for some $C>0$. Here we use the fact that $h^{uc,*}$ is uniformly bounded (see proof of Lemma \ref{le:boundaries*}). Similarly, following the proof of  Lemma \ref{le:R_1_2} we can also show that $J_{[]}(1,\F,L_2(P^*_n))$ is bounded. 
\end{proof}
\begin{lemma}
	\label{le:R_2*}
	Let $T_n=n^{\gamma}$ for some $\gamma\in (0,1/3)$. For $t,\,s\in(\epsilon,M)$ and $\pi$ as in \eqref{def:pi}, define
	\begin{equation}
	\label{def:R_2*}
	R^{2,*}_n(t,s)=\int e^{\beta_0 z}\pi_n(u;t,s)\int_{u}^{s}\frac{\lambda_0(v)}{\Phi(v;\beta_0)}\,\dd v\,\dd(\p_n^*-P^*_n)(u,\delta,z).
	\end{equation}
	{Under the assumptions of Theorem \ref{theo:CLT*},} for each $2\leq q<2/(3\gamma)$, there exist $K>0$ such that, for $x\geq 0$  and $a\in\R$,
	\[
	P^*_\theta\left[\left\{\sup_{|t_n-U_{\hat\theta}^*(a)|\leq cn^{-1/3}T_n}|R^{2,*}_n(t,t_n)|>x\right\}\cap E_n^*\right]\leq Kx^{-q} n^{1-q}.
	\]
\end{lemma}
\begin{proof}
	We follow the proof of Lemma \ref{le:R_2_2} considering the same class of functions with $U(a)$ replaced by $U_{\hat\theta}$. We just need to show that $\Vert F\Vert_{L_q(P^*_n)}^q=O\left((n^{-1/3}T_n)^{q+1}\right)$ and that $J_{[]}(1,\F,L_2(P^*_n))$ is bounded. 
	In this case, we have
	\[
	\begin{split}
	\Vert F\Vert_{L_q(P^*_n)}^q&=\int e^{q\beta_0z}\1_{\left\{U_{\hat\theta}(a)-n^{-1/2+\alpha}<u\leq U_{\hat\theta}(a)+n^{-1/2+\alpha}+cn^{-1/3}T_n\right\}}\\
	&\qquad \left(\int_{u}^{U_{\hat\theta}(a)+n^{-1/2+\alpha}+nc^{-1/3}T_n}\frac{\lambda_0(v)}{\Phi(v;\beta_0)}\,\dd v \right)^q\,\dd P_n^*(u,\delta,z)\\
	&=\frac{1}{n}\sum_{i=1}^n e^{q\beta_0Z_i}\int_{U_{\hat\theta}(a)-n^{-\frac12+\alpha}}^{U_{\hat\theta}(a)+n^{-\frac12+\alpha}+cn^{-\frac13}T_n}\left(\int_{u}^{U_{\hat\theta}(a)+n^{-\frac12+\alpha}+cn^{-\frac13}T_n}\frac{\lambda_0(v)}{\Phi(v;\beta_0)}\,\dd v \right)^q\,\dd H^*(u|Z_i)\\
	&\lesssim (n^{-1/3}T_n)^{q}\frac{1}{n}\sum_{i=1}^n e^{q\beta_0Z_i}\left[H^*(U_{\hat\theta}(a)+n^{-1/2+\alpha}+cn^{-1/3}T_n|Z_i)-H^*(U_{\hat\theta}(a)+n^{-1/2+\alpha}|Z_i)\right]
	\end{split}
	\]
	Using the approximation of $H^*$ by the differentiable version $\tilde{H}^*$ as in \eqref{eqn:approx_H*}, we obtain 	$\Vert F\Vert_{L_q(P^*_n)}^q=O\left((n^{-1/3}T_n)^{q+1}\right)$. In the same way one can adjust also the proof of the boundedness of $J_{[]}(1,\F,L_2(P^*_n))$. 
\end{proof}
\begin{lemma}
	\label{le:Lambda_n-Lambda*}
	Let $M_n^*=\Lambda_n^*-\Lambda_{\hat\theta}$, $a\in\R$ and  $T_n<n^{\gamma}$, $\gamma\in(0,1/12)$. Under the assumptions of Theorem \ref{theo:CLT*}, there exists a Brownian Bridge $B_n$ and an event $E_n^*$ such that 
	\[
	\begin{split}
	M_n^*(t_n)-M_n^*(U_{\hat\theta}^*(a))&= n^{-1/2}\frac{1}{\Phi(U_{\hat\theta}^*(a);\beta_0)}\left[B_n\left(H^{uc,*}(t_n)\right)-B_n\left(H^{uc,*}(U_{\hat\theta}^*(a))\right)\right]\\
	&\quad-\left[\Phi_n^*(U_{\hat\theta}^*(a);\tilde\beta_n)-\Phi^*(U_{\hat\theta}^*(a);\tilde\beta_n)\right] \frac{\lambda_0(U_{\hat\theta}^*(a))}{\Phi(U_{\hat\theta}^*(a);\beta_0)}\left(t_n-U_{\hat\theta}^*(a)\right)\\
	&\quad+(\hat\beta_n^*-\tilde\beta_n)A'_0(U_{\hat\theta}^*(a))\left(t_n-U_{\hat\theta}^*(a) \right)+{r}_n^*(U_{\hat\theta}^*(a),t_n),
	\end{split}
	\]
	with some $r_n^*(U_{\hat\theta}^*(a),t_n)$ that satisfies 
	\[
	\p^*_\theta\left[\left\{\sup_{|t_n-U_{\hat\theta}^*(a)|\leq cn^{-1/3}T_n}|{r}_n^*(U_{\hat\theta}^*(a),t_n)|>x \right\}\cap E_n^*\right]\leq K x^{-q}n^{1-q}, 
	\]
	for some $K>0$ and all $x\geq 0$,   $a\in\R$ and  $q<2/(3\gamma)$.
\end{lemma}
\begin{proof}
	Let $\pi_(u;t,t_n)=\1_{\left\{u\leq t_n\right\}}-\1_{\left\{u\leq t\right\}}$.	Using \eqref{eqn:Breslow_linear*}, we can write
	\begin{equation}
	\label{eqn:Lambda_n-Lambda*}
	\begin{split}
	M_n^*(t_n)-M_n^*(U_{\hat\theta}^*(a))
	&=\int\frac{\delta \pi_n(u;U_{\hat\theta}^*(a),t_n)}{\Phi^*(u;\tilde\beta_n)}\,\dd (\p_n^*-P^*_n)(u,\delta,z)\\
	&\quad-\int e^{\tilde\beta_n z}\int_{u\wedge U_{\hat\theta}^*(a)}^{u\wedge t_n}\frac{\lambda_{\hat\theta}(v)}{\Phi^*(v;\tilde\beta_n)}\,\dd v\,\dd(\p_n^*-P^*_n)(u,\delta,z)\\
	&\quad+(\hat\beta_n^*-\tilde\beta_n)\left[A_0^*\left(t_n\right)-A_0^*\left(U_{\hat\theta}^*(a)\right)\right]+R_n^*(t_n)-R_n^*\left(U_{\hat\theta}^*(a)\right).
	\end{split}
	\end{equation}
	Similarly to Lemma \ref{le:Breslow_linear*}, it can be shown that 
	\[
	\sup_{t,s\in[\epsilon,M]}\left|\int\delta \pi(u;t,s)\left(\frac{1}{\Phi^*(u;\tilde\beta_n)}-\frac{1}{\Phi(u;\beta_0)}\right)\,\dd (\p_n^*-P^*_n)(u,\delta,z)\right|=O_{P^*}(n^{-1}).
	\]
	Moreover, since the left hand side does not depend on $\hat\theta^*$, the results holds also with respect to the conditional probability $P^*_\theta$.
	Hence, by adjusting the event $E^*_n$, we can assume that 
	\[
	\sup_{t,s\in[\epsilon,M]}\left|\int\delta \pi(u;t,s)\left(\frac{1}{\Phi^*(u;\tilde\beta_n)}-\frac{1}{\Phi(u;\beta_0)}\right)\,\dd (\p_n^*-P^*_n)(u,\delta,z)\right|\lesssim n^{-1+\alpha}.
	\]
	As in the proof of Lemma \ref{le:Lambda_n-Lambda}, we have
	\begin{equation*}
	\begin{split}
	&\int\frac{\delta \pi\left(u;U_{\hat\theta}^*(a),t_n\right)}{\Phi(u;\beta_0)}\,\dd (\p_n^*-P^*_n)(u,\delta,z)\\
	&=\frac{1}{\Phi\left(U_{\hat\theta}^*(a);\beta_0\right)}\int \delta \pi\left(u;U_{\hat\theta}^*(a),t_n\right)\,\dd (\p_n^*-P^*_n)(u,\delta,z)+R^{1,*}_n\left(U_{\hat\theta}^*(a),t_n\right)
	\end{split}
	\end{equation*}
	where
	\[
	R^{1,*}_n\left(U_{\hat\theta}^*(a),t_n\right)=\int\delta \pi\left(u;U_{\hat\theta}^*(a),t_n\right)\left(\frac{1}{\Phi(u;\beta_0)}-\frac{1}{\Phi\left(U_{\hat\theta}^*(a);\beta_0\right)}\right)\,\dd (\p_n^*-P^*_n)(u,\delta,z)
	\]
	and 
	\begin{equation*}
	\begin{split}
	&\int \delta \pi\left(u;U_{\hat\theta}^*(a),t_n\right)\,\dd (\p_n^*-P^*_n)(u,\delta,z)\\
	&=n^{-1/2}\left[B_n\left(H^{uc,*}(t_n)\right)-B_n\left(H^{uc,*}\left(U_{\hat\theta}^*(a)\right)\right)\right]\\
	&\quad+\left[(\tilde{M}_n^*-n^{-1/2}B_n\circ H^{uc,*})(t_n)-(\tilde{M}^*_n-n^{-1/2}B_n\circ H^{uc,*})\left(U_{\hat\theta}^*(a)\right)\right]
	\end{split}
	\end{equation*}
	where $\tilde{M}_n^*$ is defined in \eqref{def:tilde_M_n*}.
	Similarly, in the third term of \eqref{eqn:Lambda_n-Lambda*} we can replace $A_0^*$ by the twice differentiable version $A_0$ because
	\[
	\sup_{t\in[\epsilon,M]}\left|(\hat\beta_n^*-\tilde{\beta}_n)\left[(A_0^*-A_0)(t_n)-(A_0^*-A_0)(t)\right]\right|=O_{P^*}(n^{-1}).
	\]
	By Lemma \ref{le:Breslow_linear*}, also in the second term of \eqref{eqn:Lambda_n-Lambda*}, we can replace $\lambda_{\hat{\theta}}$ and $\Phi^*$ by $\lambda_0$ and $\Phi$. Then, we can proceed as in Lemma \ref{le:Lambda_n-Lambda} using Lemmas \ref{le:R_1*} and \ref{le:R_2*}.
\end{proof}
\begin{lemma}
	\label{le:R_3*}
	Let $T_n=n^{\gamma}$ for some $\gamma\in (0,1/3)$ and $b=n^{-1/4}$.  	
	{Under the assumptions of Theorem \ref{theo:CLT*},} for each $2\leq q<1/(\gamma+1/24)$, there exist $K>0$ such that, for $x\geq 0$ and all  $a\in\R$,
	\[
	P^*_\theta\left[n^{1/3}T_n\sup_{|v-U^*_{\hat\theta}(a)|\leq b}\left|(\Phi_n^* -\Phi^*)(v;\tilde\beta_n)-(\Phi_n^*-\Phi^* )(U^*_{\hat\theta}(a);\tilde\beta_n)\right|>x\right]\leq K x^{-q}n^{1-q/3}.
	\]
\end{lemma}
\begin{proof}
	We follow the proof of Lemma \ref{le:R_3_2} considering the same class of functions with $\beta_0$ and $U(a)$ replaced by $\tilde{\beta}_n$ and $U_{\hat\theta}(a)$ respectively. We just need to show that $\Vert F\Vert_{L_q(P^*_n)}^q=O(b)$ and that $J_{[]}(1,\F,L_2(P^*_n))$ is bounded. 
	For the envelope function, we obtain
	\[
	\begin{split}
	\Vert F\Vert_{L_q(P^*_n)}^q&=\int \1_{\left\{U_{\hat\theta}(a)+n^{-1/2+\alpha}<t\leq U_{\hat\theta}(a)+n^{-1/2+\alpha}+b\right\}}e^{q\tilde\beta_nz}\,\dd P^*_n(t,\delta,z)\\
	&=\frac{1}{n}\sum_{i=1}^n e^{q\tilde\beta_nZ_i}\left[H^*\left(U_{\hat\theta}(a)+n^{-1/2+\alpha}+b|Z_i\right)-H^*\left(U_{\hat\theta}(a)+n^{-1/2+\alpha}|Z_i\right)\right].
	\end{split}
	\]
	As in the proof of Lemma \ref{le:R_2*} it can be shown that $\Vert F\Vert_{L_q(P^*_n)}^q=O(b)$. In the same way e can adjust the proof for the boundedness of  
	$J_{[]}(1,\F,L_2(P^*_n))$. 
\end{proof}
\begin{proof}[Proof of Theorem \ref{theo:CLT*}]Let 
	\[
	\mathcal{J}^*_n=n^{p/3}\int_\epsilon^M\left|\hat{\lambda}_n^*(t)-\lambda_{\hat\theta}^*(t)\right|^p\,\text{d}t.
	\]
	
	\textbf{Step 1.} As in the proof of Theorem \ref{theo:CLT} and Theorem \ref{theo:CLT2}, it can be shown that  $\mathcal{J}_n^*\1_{E_n^*}=\tilde{\mathcal{J}}_n^*\1_{E_n^*}+o_{P^*}(n^{-1/6})$ where
	\[
	\tilde{\mathcal{J}}_n^*=n^{p/3}\int_{\lambda_{\hat\theta}^*(M)}^{\lambda_{\hat\theta}^*(\epsilon)}\left|\hat{U}_n^*(a)-U^*_{\hat\theta}(a)\right|^p\left|(U^*_{\hat\theta})'(a)\right|^{1-p}\,\text{d}a.
	\]
	Here we use \eqref{eqn:inv2*}, \eqref{eqn:bound_expectation_L_p1*}- \eqref{eqn:bound_expectation_L_p3*}, \ref{le:boundaries*} and the fact that $(\lambda^*_{\hat\theta})'$ satisfies \eqref{eqn:assumption_derivative}.
	Note that now we need $$\max\left(p-1/2,1,\frac{4(p-1)}{3(1-\alpha_0)}\right)<p'< \min\left(p,\frac{2}{1+\alpha_0}\right).$$ Such a choice is possible since $\alpha_0$ can be chosen arbitrarily small and $p<5/2$. Then we get
	\[
	\begin{split}
	\1_{E_n^*}\int_{\epsilon+n^{-1}\log n}^{\epsilon+n^{-1/2}\sqrt{\log n}}\left[\hat\lambda_n^*(t)-\lambda_{\hat\theta}^*(t)\right]_+^{p'}\,\dd t&=O_{P^*}\left(n^{-p'/2}n^{-\frac12\{-(1+\alpha_0)p'/2+1\}}\sqrt{\log n}^{-(1+\alpha_0)p'/2+1}\right)\\
	&=O_{P^*}\left(n^{-(1-\alpha_0)p'/4-1/2}\sqrt{\log n}^{-(1+\alpha_0)p'/2+1}\right)\\
	&=o_{P^*}\left(n^{-p/3-1/6}\right)
	\end{split}
	\]
	\textbf{Step 2.} Let 
	\[
	W_t(u)=n^{1/6}\left[W_n\left(H^{uc,*}(t)+n^{-1/3}u\right)-W_n\left(H^{uc,*}(t)\right)\right]
	\]
	and $d(t)$, $
	\tilde{V}(t)$ as in the proof of Theorem \ref{theo:CLT}. We argue conditionally on $\hat\theta_n^*$ such that $|\hat\theta_n^*-\hat\theta_n|\leq n^{-1/2+\alpha}$ and we
	show that 
	\begin{equation}
	\label{eqn:step2*}
	\tilde{\mathcal{J}}_n^*\1_{E_n^*}=\int_{\epsilon+b}^{M-b}\left|\tilde{V}(t)-n^{-1/6}\eta_n^*(t) \right|^p\left|\frac{\lambda'_0(t)}{h^{uc}(t)}\right|^p\,\dd t+o_{P^*_\theta}(n^{-1/6}),
	\end{equation}
	for some $\eta_n^*$ such that $\sup_{t\in[\epsilon+b,M-b]}|\eta_n^*(t)|=O_{P^*_\theta}(1)$.
	Let 
	\[
	\Psi_n^{s,*}(t)=\int k_b(t-v)\left(\Phi_n^*(v;\tilde\beta_n)-\Phi^*(v;\tilde\beta_n)\right)\,\dd v.
	\]
	Note that, from Lemma \ref{le:Phi*},
	\[
	\sup_{t\in[\epsilon+b,M-b]}\left|\Psi_n^{s,*}(t) \right|\leq \sup_{v\in[\epsilon,M]}\left| \Phi_n^*(v;\tilde\beta_n)-\Phi^*(v;\tilde\beta_n)\right|=O_{P^*}(n^{-1/2})
	\]
	and
	\[
	\sup_{t\in[\epsilon+b,M-b]}\left|\frac{\dd\Psi_n^{s,*}(t)}{\dd t}\right|\leq b^{-1}\sup_{|v-t|\leq b}\left| (\Phi_n^*(v;\tilde\beta_n)-\Phi^*(v;\tilde\beta_n))-(\Phi_n^*(t;\tilde\beta_n)-\Phi^*(t;\tilde\beta_n))\right|=o_{P^*}(n^{-1/6}).
	\]
	For every $a\in\R$ and $\tilde{a}=\lambda_{\hat\theta}(U_{\hat\theta}^*(a))$ and define
	\[
	\begin{split}
	a^\xi=\tilde{a}-n^{-1/2}\xi_n\frac{{h}^{uc}(U^*_{\hat\theta}(a))}{\Phi(U_{\hat\theta}^*(a);\beta_0)}+(\hat\beta_n^*-\tilde\beta_n)A'_0(U^*_{\hat\theta}(a))-\Psi_n^{s,*}(U^*_{\hat\theta}(a))\frac{\lambda_0(U^*_{\hat\theta}(a))}{\Phi(U^*_{\hat\theta}(a);\beta_0)}.
	\end{split}
	\]
	We can assume that on the event $E_n^*$ it holds
	\begin{equation}
	\label{eqn:E_n_add*}
	\sup_{a\in\R}\left|a-a^\xi\right|\leq cn^{-1/2+\alpha}.
	\end{equation}
	Then, from Lemma 6(i) in \cite{durot2007} and the change of variable $a\to a^{\xi}$ we have
	\[
	\begin{split}
	\mathcal{J}_n^*\1_{E_n^*}&=\1_{E_n^*}n^{p/3}\int_{\lambda_{\hat\theta}^*(M)}^{\lambda_{\hat\theta}^*(\epsilon)}\left|\frac{H^{uc}(\hat{U}^*_n(a))-H^{uc}(U^*_{\hat\theta}(a))}{h^{uc}(U^*_{\hat\theta}(a))}\right|^p\left|(U^*_{\hat\theta})'(a)\right|^{1-p}\,\text{d}a+o_{P^*_{\theta}}(n^{-1/6})\\
	&=\1_{E_n^*}n^{p/3}\int_{\lambda_{\hat\theta}^*(M)}^{\lambda_{\hat\theta}^*(\epsilon)}\left|\frac{H^{uc,*}(\hat{U}^*_n(a))-H^{uc,*}(U^*_{\hat\theta}(a))}{h^{uc}(U^*_{\hat\theta}(a))}\right|^p\left|(U^*_{\hat\theta})'(a)\right|^{1-p}\,\text{d}a+o_{P^*_\theta}(n^{-1/6})\\
	&=\1_{E_n^*}n^{p/3}\int_{J_n}\left|\frac{H^{uc,*}(\hat{U}^*_n(a^\xi))-H^{uc,*}(U^*_{\hat\theta}(a^\xi))}{h^{uc}(U^*_{\hat\theta}(a))}\right|^p\left|(U^*_{\hat\theta})'(a)\right|^{1-p}\,\text{d}a+o_{P^*_\theta}(n^{-1/6}),
	\end{split}
	\]
	with
	\[
	J_n=\left[\lambda_{\hat\theta}^*(M)+n^{-1/6}/\log n,\lambda_{\hat\theta}^*(\epsilon)-n^{-1/6}/\log n\right].
	\]
	Here we use that, for $r<3$, 
	\[
	\begin{split}
	&n^{r/3}\sup_{a\in[\lambda_{\hat\theta}^*(M),\lambda_{\hat\theta}^*(\epsilon)]}\E^*_\theta\left[\left|\frac{(H^{uc,*}-H^{uc})(U^*_n(a))-(H^{uc,*}-H^{uc})(U^*_{\hat\theta}(a))}{h^{uc}(U^*_{\hat\theta}(a))}\right|^r\right]\\
	&\lesssim \sup_{t\in[\epsilon,M]}\left|h^{uc,*}(t)-h^{uc}(t)\right|\sup_{a\in[\lambda_{\hat\theta}^*(M),\lambda_{\hat\theta}^*(\epsilon)]}\E^*_\theta\left[n^{r/3}\left|U^*_n(a)-U^*_{\hat\theta}(a)\right|^r\right]\\
	&=O\left(n^{-1/2}\right)=o\left(n^{-r/6}\right)
	\end{split}
	\]
	because, as in Lemma \ref{lemma:approx_Phi*}, we also have that with probability converging to one, given the data
	\begin{equation}
	\label{eqn:h^uc*}
	\sup_{t\in[\epsilon,M]}\left|h^{uc,*}(t)-h^{uc}(t)\right|=O(n^{-1/2}).
	\end{equation}
	Let $  t_n^*=(H^{uc,*})^{-1}(H^{uc,*}(U^*_{\hat\theta}(a))+n^{-1/3}u)$. Using properties of the $\argmax$ function we obtain
	\[
	n^{1/3}\left(H^{uc,*}(\hat{U}^*_n(a^\xi))-H^{uc,*}(U^*_{\hat\theta}(a))\right)=\argmax_{u\in I_n^*(a)}\left\{D_n^*(a,u)+S_n^*(a,u) \right\},
	\]
	where
	\[
	I_n^*(a)=\left[-n^{1/3}\left(H^{uc*,}(U^*_{\hat\theta}(a))-H^{uc,*}(\epsilon)\right),n^{1/3}\left(H^{uc,*}(M)-H^{uc,*}(U^*_{\hat\theta}(a))\right)\right],
	\]
	\[
	\begin{split}
	D_n^*(a,u)=n^{2/3}\Phi(U^*_{\hat\theta}(a);\beta_0)\left\{\left(\Lambda_{\hat\theta}(t_n^*)-\tilde{a}t_n^*\right)- \left(\Lambda_{\hat\theta}(U^*_{\hat\theta}(a))-\tilde{a}U^*_{\hat\theta}(a)\right) \right\}
	\end{split}
	\]
	and
	\begin{equation*}
	\begin{split}
	S_n^*(a,u)=n^{2/3}\Phi(U^*_{\hat\theta}(a);\beta_0)\left\{ (\tilde{a}-a^\xi)\left[t_n^*-U^*_{\hat\theta}(a)\right]+(\Lambda_n^*-\Lambda_{\hat\theta})\left(t_n^*\right)-(\Lambda_n^*-\Lambda_{\hat\theta})(U^*_{\hat\theta}(a))\right\}.
	\end{split}
	\end{equation*}
	Using Lemma \ref{le:Lambda_n-Lambda*} and the definition of $a^\xi$, we get
	\[
	S_n^*(a,u)=W_{U^*_{\hat\theta}(a)}(u)+R^{3,*}_n(a,u)+R^{4,*}_n(a,u)+R^{5,*}(a,u)
	\]
	where
	\[
	R^{3,*}_n(a,u)=n^{2/3}\left\{\Psi_n^{s,*}(U^*_{\hat\theta}(a)) -\left[\Phi_n^*(U^*_{\hat\theta}(a);\tilde\beta_n)-\Phi^*(U^*_{\hat\theta}(a);\tilde\beta_n)\right] \right\}\lambda_{0}(U^*_{\hat\theta}(a))\left[t_n^*-U^*_{\hat\theta}(a)\right],
	\]
	\[
	R^{4,*}_n(a,u)=n^{2/3}\Phi(U^*_{\hat\theta}(a);\beta_0)\,r_n^*\left(U^*_{\hat\theta}(a),t_n^*\right),
	\]
	with $r_n^*$ as in Lemma 
	\ref{le:Lambda_n-Lambda*} and
	\[
	R^{5,*}_n(a,u)=n^{1/6}\xi_nh^{uc}(U^*_{\hat\theta}(a))\left[t_n^*-(H^{uc})^{-1}(H^{uc}(U^*_{\hat\theta}(a))+n^{-1/3}u)\right].
	\]
	Note that, by the mean value theorem, 
	\[
	\begin{split}
	&\sup_{|u|\leq T_n}\left|(H^{uc,*})^{-1}(H^{uc,*}(U^*_{\hat\theta}(a))+n^{-1/3}u)-(H^{uc})^{-1}(H^{uc}(U^*_{\hat\theta}(a))+n^{-1/3}u)\right|\\
	&=n^{-1/3}T_n\sup_{|u|\leq T_n}\left|\frac{1}{h^{uc,*}(\zeta^1_n(u))}-\frac{1}{h^{uc}(\zeta_n^2(u))}\right|\\
	&\lesssim n^{-1/3}T_n\left(\sup_{t\in[\epsilon,M]}\left|h^{uc,*}(t)-h^{uc}(t)\right|+\sup_{|u|\leq T_n}\left|h^{uc}(\zeta^1_n(u))-h^{uc}(\zeta^2_n(u))\right|\right)\\
	&\lesssim n^{-2/3}T_n^2.
	\end{split}
	\]
	for some $\zeta^1_n(u),\,\zeta_n^2(u)$ such that $|\zeta^i_n(u)-U^*_{\hat\theta}(a)|\lesssim n^{-1/3}u$, $i=1,2$.
	Let $R_n^*(a,u)=R^{3,*}_n(a,u)+R^{4,*}_n(a,u)+R^{5,*}_n(a,u)$.	
	Now we use Lemma 5 in \cite{durot2007} to show that $R_n^*$ is negligible. As in the proof of Theorem \ref{theo:CLT}, using Lemma \ref{le:inv_tail_prob1*} and the fact that $h^{uc,*}$ and $(U^*)'$ are uniformly bounded (since they converge to $h^{uc}$ and $U$), we can localize obtaining that
	\[
	\tilde{\mathcal{J}}_n^*\1_{E_n^*}=\1_{E_n^*}\int_{J_n}\left|\tilde{U}_n^*(a)-n^{-1/6}\eta_n(a) \right|^p\frac{|(U^*_{\hat\theta})'(a)|^{1-p}}{h^{uc}(U^*_{\hat\theta}(a))^p}\,\dd a+o_{P^*_\theta}(n^{-1/6})
	\]
	where
	\[
	\tilde{U}_n^*(a)=\argmax_{|u|\leq T_n}\left\{D_n^*(a,u)+W_{U^*(a)}(u)+R_n^*(a,u)\right\},
	\]
	\[
	\eta_n^*(a)=n^{1/2}(a-a^\xi)|(U^*_{\hat\theta})'(a)|h^{uc,*}(U^*_{\hat\theta}(a))=O_{P^*_\theta}(1).
	\]
	Afterwards, using Lemmas \ref{le:Lambda_n-Lambda*}, \ref{le:R_3*} and 
	\[
	P^*_\theta\left[\sup_{|u|\leq T_n}\left|R^{5,*}_n(a,u)\right|>x\right]\lesssim P^*_\theta\left[\left|\xi_n\right|>cn^{1/2}T_n^{-2}x\right]\lesssim x^{-q}n^{1-q/3},
	\]
	we show as in the proof of Theorem \ref{theo:CLT} that $R_n^*$ is negligible. 
	For $D^*_n(a,u),$ in this case, by a Taylor expansion we have
	\[
	D_n^*(a,u)=n^{2/3}\Phi\left(U^*_{\hat\theta}(a);\beta_0\right)\lambda'_{\hat\theta}(\zeta_n)\left[(H^{uc,*})^{-1}(H^{uc,*}(U^*_{\hat\theta}(a))+n^{-1/3}u)-U^*_{\hat\theta}(a)\right]^2
	\]
	for some $\zeta_n$ such that $|\zeta_n-U^*_{\hat\theta}(a)|\lesssim n^{-1/3}\log n$. Since $\lambda'_{\hat\theta}$ and $h^{uc,*}$ are uniformly bounded we again obtain
	\[
	\left|\frac{\partial}{\partial u}D_n^*(a,u)\right|\leq K|u| \quad\text{and}\quad D_n^*(a,u)\leq -cu^2,
	\]
	for some positive constants $K,\,c$, $a\in J_n$ and $|u|\leq T_n$. Moreover, from \eqref{eqn:par_est-true*} it follows that 
	\[
	\sup_{|u|\leq \log n}\left|D^*_n(a,u)-n^{2/3}\Phi\left(U^*_{\hat\theta}(a);\beta_0\right)\lambda'_{0}(\zeta_n)\left[t_n^*-U^*_{\hat\theta}(a)\right]^2\right|\leq n^{-1/2+\alpha}(\log n)^2.
	\] 
	Then, \eqref{eqn:assumption_derivative} yields 
	\[
	\sup_{|u|\leq \log n}\left|D^*_n(a,u)-n^{2/3}\Phi\left(U^*_{\hat\theta}(a);\beta_0\right)\lambda'_{0}(U^*_{\hat\theta}(a))\left[t_n^*-U^*_{\hat\theta}(a)\right]^2\right|\leq n^{-s/3}(\log n)^3.
	\]
	Finally, using \eqref{eqn:h^uc*}, we obtain
	\[
	\begin{split}
	n^{2/3}\left[t_n^*-U^*_{\hat\theta}(a)\right]^2&=\frac{u^2}{h^{uc,*}(\zeta_n)^2}=\frac{u^2}{h^{uc}(\zeta_n)^2}+O\left(n^{-1/2}(\log n)^2\right)\\
	&=\frac{u^2}{h^{uc}(U^*(a))^2}+O\left(n^{-1/3}(\log n)^3\right)
	\end{split}
	\]
	for some $\zeta_n$ such that $|\zeta_n-U^*_{\hat\theta}(a)|\lesssim n^{-1/3}\log n$, and as a result 
	\[
	\sup_{|u|\leq \log n}\left|D^*_n(a,u)-d(U^*_{\hat\theta}(a))u^2\right|\leq n^{-s/3}(\log n)^3.
	\]
	Then, exactly as in Theorem \ref{theo:CLT2}, we obtain \eqref{eqn:step2*} with 
	\begin{equation}
	\label{def:eta_t*}
	\eta_n^*(t)=\frac{h^{uc}(t)}{|\lambda'_0(t)|}\left\{\xi_n\lambda_0(t)-n^{1/2}(\hat\beta_n^*-\tilde\beta_n)A'_0(t)+n^{1/2}\Psi_n^{s,*}(t)\frac{\lambda_0(t)}{\Phi(t;\beta_0)}-n^{1/2}\left(\lambda^*_{\hat\theta}(t)-\lambda_{\hat\theta}(t)\right) \right\}.
	\end{equation}
	Step 3 and 4 are exactly as in Theorem \ref{theo:CLT2}.
\end{proof}

\section{{Estimation in the Cox model}}\label{sec:Cox}
In this section, we gather properties of the maximum partial likelihood estimator $\hat\beta_n$ and the empirical estimator $\Phi_n$ in the Cox model that are used in our proofs. Although the results seem to be well know in the literature, we are not aware of a reference where they are proved. However, it is out of the scope of the current paper to provide a proof. 

We have
\begin{equation}\label{eq: moments betan}
\E\left[n^{1/2}\vert\hat\beta_n-\beta_0\vert\right]^p \lesssim 1
\end{equation}
and

\begin{equation}\label{eq: moments Phin}
\E\left[n^{1/2}\sup_t\vert\Phi_n(t,\beta_0)-\Phi(t,\beta_0)\vert\right]^p \lesssim 1
\end{equation}
for all $p\geq 1$.

\end{document}